\providecommand{\pgfsyspdfmark}[3]{}

\documentclass{article}
\usepackage{mathrsfs}
\usepackage{amssymb}
\usepackage{cite}
\usepackage[all]{xy}
\usepackage{graphicx}
\usepackage{textcomp}
\usepackage{charter}
\usepackage[vflt]{floatflt}
\usepackage{mathtools}
\usepackage{enumerate}
\usepackage{wasysym}
\usepackage{rotating}
\usepackage{pdflscape}
\usepackage{fullpage}
\usepackage{authblk}
\usepackage{comment}
\usepackage{bm}
\usepackage[pdftex,colorlinks]{hyperref}
\hypersetup{bookmarksnumbered,pdfstartview={FitH}}
\usepackage{tikz-cd}
\usepackage{nicematrix}
\usepackage{colortbl}
\usepackage{color}

\usepackage{blkarray}
\usepackage{amsmath}
\usepackage{amsthm}
\usepackage{esint}

\usepackage{setspace}

\usepackage{amsfonts}
\usepackage{eucal}
\usepackage{latexsym}
\usepackage{mathdots}
\usepackage{mathrsfs}
\usepackage{enumitem}

\newtheorem{thm}{Theorem}[section]
\newtheorem{prop}[thm]{Proposition}

\newtheorem{lemma}[thm]{Lemma}

\newtheorem{cor}[thm]{Corollary}
\newtheorem*{thm*}{Theorem}
\newtheorem*{cor*}{Corollary}
\newtheorem*{prop*}{Proposition}

\theoremstyle{definition}
\newtheorem{defn}[thm]{Definition}

\theoremstyle{remark}
\newtheorem{remark}[thm]{Remark}

\newtheorem{eg}[thm]{Example}
\numberwithin{equation}{section}

\newcommand{\be}{\begin{equation}}
\newcommand{\ee}{\end{equation}}
\def\ba{\begin{eqnarray*}}
\def\ea{\end{eqnarray*}}

\newcommand{\bi}{\begin{itemize}}
\newcommand{\ei}{\end{itemize}}
\newcommand{\bn}{\begin{enumerate}}
\newcommand{\en}{\end{enumerate}}
\newcommand{\bbm}{\begin{bmatrix}}
\newcommand{\ebm}{\end{bmatrix}}
\newcommand{\bpm}{\begin{pmatrix}}
\newcommand{\epm}{\end{pmatrix}}
\newcommand{\bsm}{\left ( \begin{smallmatrix}}
\newcommand{\esm}{\end{smallmatrix} \right) }

\newcommand{\mr}{\ensuremath{\mathrm}}
\newcommand{\scr}{\ensuremath{\mathscr}}

\newcommand{\ov}{\ensuremath{\overline}}
\newcommand{\sm}{\ensuremath{\setminus}}
\newcommand{\wt}{\ensuremath{\widetilde}}

\newcommand{\ga}{\ensuremath{\gamma}}
\newcommand{\Om}{\ensuremath{\Omega}}

\newcommand{\La}{\ensuremath{\Lambda }}
\newcommand{\la}{\ensuremath{\lambda }}
\newcommand{\om}{\ensuremath{\omega}}

\def\C{\mathbb{C}}

\def\D{\mathbb{D}}
\def\Z{\mathbb{Z}}
\def\N{\mathbb{N}}
\def\B{\mathbb{B}}

\def\fr{\mathfrak{r}}

\def\fz{\mathfrak{z}}

\def\fp{\mathbb{C} \langle \fz \rangle }
\def\fps{\mathbb{C} \langle \! \langle  \fz  \rangle \! \rangle}
\def\onefps{\mathbb{C} [ \! [ \fz ] \! ]}
\def\mrt{\mathrm{t}}

\def\hardy{\mathbb{H} ^2 _d}
\def\mult{\mathbb{H} ^\infty _d}
\newcommand{\ip}[2]{\ensuremath{\langle {#1} , {#2} \rangle}}
\def\nbdom{\mr{Dom} \, }

\def\nbran{\mr{Ran} \, }

\def\nn{\nonumber}

\def\fskew{\C \ \mathclap{\, <}{\left( \right.}   \fz  \mathclap{  \, \, \, \, \, >}{\left. \right)} \, \, }
\def\ratfps{\C _0 \ \mathclap{\, <}{\left( \right.}  \fz  \mathclap{ \, \, \, \, \, >}{\left. \right)} \, \, }

\def\fskewd{\C \ \mathclap{\, <}{\left( \right.}   \fz _1, \cdots, \fz _d  \mathclap{  \, \, \, \, \, >}{\left. \right)} \, \, }
\def\fskewk{\C \ \mathclap{\, <}{\left( \right.}   \fz _1, \cdots, \fz _k  \mathclap{  \, \, \, \, \, >}{\left. \right)} \, \, }

\def\cdn{\mathbb{C} ^{(n\times n)\cdot d}}

\def\cdm{\mathbb{C} ^{(m\times m)\cdot d}}

\def\cH{\mathcal{H}}
\def\cJ{\mathcal{J}}

\def\F{\mathbb{F} ^+ _d}
\def\Fn{\mathbb{F} ^+ _n}
\def\ncu{\mathbb{C} ^{(\N \times \N) \cdot d}}

\def\cball{\mathbb{B} ^{(\N \times \N) \cdot d}}
\def\bdn{\mathbb{B} ^{(n\times n) \cdot d}}
\def\bdm{\mathbb{B} ^{(m\times m) \cdot d}}

\def\cint{\ointctrclockwise}

\title{Operator realizations about a matrix-centre}

\author{Ali Karoobi}

\author{Robert T.W. Martin\thanks{Supported by NSERC grant 2020-05683}}

\author{Maximilian Tornes}

\affil{\footnotesize University of Manitoba}

\date{}

\begin{document}
\maketitle
\vspace{-.75cm}

\begin{abstract}
We develop a general theory of operator realizations, or ``linear representations" of analytic functions in several non-commuting variables about a matrix--centre. In particular we show that a non-commutative function has a matrix-centre realization about any matrix tuple, $Y$, in its domain, if and only if it is a uniformly analytic non-commutative function defined in a uniformly open neighbourhood of $Y$. This extends the finite--dimensional realization theory of non-commutative rational functions maximally -- to all uniformly analytic non-commutative functions. 
\end{abstract}

\section{Introduction}

A univariate \emph{realization} is a triple, $(A,b,c) \in \scr{B} (\cH) \times \cH \times \cH$ where $\cH$ is a complex, separable or finite--dimensional Hilbert space and $\scr{B} (\cH)$ denotes the Banach space of bounded linear operators on $\cH$. Any realization defines an analytic function in an open neighbourhood of $0$ by the realization formula,
$$ h (z) := b^* (I -zA) ^{-1} c, $$ and we write $h \sim (A,b,c)$ if $h$ is defined by the realization in this way. Here, if $|z| < \| A \| ^{-1}$, then the \emph{linear pencil}, $L_A (z) := (I-zA) ^{-1}$ can be expanded as an operator-norm convergent geometric series so that 
$$ h(z) = \sum _{n=0} ^\infty b^* A^n c \, z^n; \quad \quad |z| < \| A \| ^{-1}, $$ is an absolutely convergent Taylor series. Conversely, an elementary argument shows that a one-variable formal power series, $h \in \onefps$, is the Taylor series of an analytic function in an open neighbourhood of $0$ if and only if it has a realization, $h \sim (A,b,c)$, in which case we say that $h$ is \emph{familiar} \cite[Lemma 3.1]{AMS-opreal}.

Realizations are also naturally defined for analytic functions of several non-commuting variables. Namely, let $f \in \fps := \C \langle \! \langle \, \fz _1, \cdots, \fz _d \rangle \! \rangle$ denote a \emph{free formal power series} (free FPS), \emph{i.e.} a formal power series in the several non-commuting formal variables $\fz := (\fz _1, \cdots, \fz _d )$. That is,
$$ f (\fz) := \sum _{\om \in \F} \hat{f} _\om \, \fz ^\om; \quad \quad \hat{f} _\om \in \C, $$
where $\F$ denotes the \emph{free monoid}, the set of all \emph{words}, $\om = i_1 \cdots i_n$, in the letters $i _k \in \{ 1, \cdots, d \}$, with product given by concatenation of words, and with unit, $\emptyset$, the \emph{empty word}, consisting of no letters. The \emph{free monomials}, $\fz ^\om$ are then defined as $\fz ^\om = \fz _{i_1} \cdots \fz _{i_n}$, and $\fz ^\emptyset =: 1$. Given $\om = i_1 \cdots i_n \in \F$ we further define the \emph{length} of $\om$ as $|\om| := n$, the number of letters it contains, and the \emph{transpose}, $\om ^\mrt$ of $\om$ as $\om ^\mrt := i_n \cdots i_1$, an involution on the free monoid. As first proven by G.F. Popescu, one can extend the Cauchy--Hadamard radius of convergence formula to free FPS: 
$$ \frac{1}{R_f} := \limsup _{n\rightarrow \infty} \sqrt[2n]{ \sum _{|\om | = n } | \hat{f} _\om | ^2}. $$ If $R_f >0$, then by \cite[Theorem 1.1]{Pop-freeholo}, the free FPS, $f$, can be evaluated at any $d-$tuple $X := (X_1, \cdots, X_d) \in \scr{B} (\cH ) ^{1\times d}$, provided that 
$$ \| X \| _{\mr{row}} := \| X \| _{\scr{B} (\cH \otimes \C ^d, \cH)} < R_f. $$ In this case, if $\| X \| _{\mr{row}} < R_f$, then the series, 
$$ f(X) := \sum _{\om \in \F} \hat{f} _\om X^\om, $$ converges absolutely in operator-norm, in the sense that 
$$ \sum _{n=0} ^\infty \left\| \sum _{|\om | = n} \hat{f} _\om X^\om \right\| _{\scr{B} (\cH)} < +\infty, $$ and this convergence is uniform on any operator `row-ball' of radius $r < R_f$, consisting of all $X \in \scr{B} (\cH) ^{1\times d}$ so that $\| X \| _{\mr{row}} \leq r$. (Alternatively, an easy modification of Popescu's theorem shows that we have absolute and uniform convergence of the free FPS, $f$, on any `column-ball' consisting of $X = \bsm X_1 \\ \vdots \\ X_d \esm \in \scr{B} (\cH) ^d$, with $\| X \| _{\mr{col}} = \| X \| _{\scr{B} (\cH, \cH \otimes \C ^d)} \leq r < R_f$. This also follows from the more general result \cite[Theorem 8.11]{KVV}.)

It follows that any free FPS, $f$, with non-zero radius of convergence, defines a (uniformly) analytic non-commutative (NC) function in an open neighbourhood of $0 := \bsm 0 \\ \vdots \\ 0 \esm \in \C ^d$, in the $d-$dimensional complex \emph{NC universe},
$$ \ncu := \bigsqcup _{n=1} ^\infty \cdn; \quad \quad \cdn := \C ^{n\times n} \otimes \C ^d, $$ of all column $d-$tuples of square matrices, of any fixed size, $n$, in the sense of modern NC function theory \cite{KVV,AgMcY,Taylor,Taylor2,Voic,Voic2}. Namely, if we define the \emph{NC column-ball} of radius $r>0$,
$$ r \cdot\cball := \bigsqcup _{n=1} ^\infty r \cdot \bdn; \quad \quad r \cdot \bdn := \left\{ \left. X \in \cdn \right| \ \| X \| _{\mr{col}} < r \right\}, $$ then as a function on $R_f \cdot \cball$, $f$: (i) is \emph{graded}, \emph{i.e.} $f(X) \in \C ^{n \times n}$ for any $X \in \cdn$ with $\| X \| _{\mr{col}} < R_f$, (ii) \emph{respects direct sums} in the sense that if $X \in R_f \cdot \bdn$ and $Y \in R_f \cdot \bdm$, then $f(X \oplus Y) = f(X) \oplus f(Y)$, and (iii) \emph{respects joint similarity} in the sense that if $X, Y \in R_f \cdot \bdn$ and there is an invertible $S \in \mr{GL} _n$ so that $Y = S^{-1} X S := \bsm S^{-1} X_1 S \\ \vdots \\ S^{-1} X_d S \esm$, then $f(Y) = S^{-1} f(X) S$. 

There are several natural topologies that one can define on the NC universe. The most useful topology on $\ncu$ for our purposes in this paper is the  \emph{uniform topology}. We will formally define the uniform topology in the following section, this is essentially the topology arising from identifying $\ncu$ as $\C ^d$ equipped with an operator space structure \cite[Section 7.2]{KVV}. (We will work with the $d-$ component column operator space over $\C$. More generally, the uniform topology can be defined relative to any operator space structure on $\C ^d$.)  A remarkable and elementary fact about NC functions is that under very mild assumptions on an NC function, $f$, and its domain, $\nbdom f$, $f$ is automatically holomorphic and analytic. In particular, if $\nbdom f$ is a uniformly open NC set, and $f$ is locally bounded in the uniform topology, then it is Fr\'echet and (hence) G\^ateaux differentiable at any point in its domain, and it has a convergent Taylor-type power series expansion with non-zero radius of convergence about any point in its domain, see \cite[Corollary 7.26 and Corollary 7.28]{KVV} as well as  \cite{AgMcY,Taylor,Taylor2}.  Non-commutative function theory is a natural extension or ``quantization" of classical complex analysis and several complex variables, and many classical results extend to this NC setting with a mixture of analytic and algebraic proofs \cite{Augat-freeGrot,Pascoe-IFT,SSS,SSS2,KVV-local,HMS-realize,HKV-poly,KS-free,KV-freeloci,AHKMc-bianalytic,Ball-sys}.

As for analytic functions of a single complex variable, an elementary argument shows that an NC function, $h$, is uniformly analytic in an open neighbourhood of $0 \in \ncu$, if and only if it admits a realization, $h \sim (A,b,c) \in \scr{B} (\cH) ^{1\times d} \times \cH \times \cH$, in the sense that for any $X \in \cdn$ so that $\| X \| _{\mr{col}} < R_h$, 
\be h(X) = I_n \otimes b^* \left( I_n \otimes I_\cH - \sum _{j=1} ^d X_j \otimes A_j \right) ^{-1} I_n \otimes c, \label{realform} \ee \cite[Lemma 3.1]{AMS-opreal}. Here, $$ L_A (X) := I_n \otimes I_\cH - \sum _{j=1} ^d X_j \otimes A_j =: I - X \otimes A, $$ is called the \emph{linear pencil} of $A \in \scr{B} := \scr{B} (\cH) ^{1 \times d}$, and it follows that if $\| X \| _{\mr{col}} < \| A \| _{\mr{row}} ^{-1}$, where $\| A \| _{\mr{row}} = \| (A_1, \cdots, A_d) \| _{\scr{B} (\cH \otimes \C ^d, \cH)}$, then 
$$ \| X \otimes A \| = \left\| \left( I_n \otimes A_1, \cdots, I_n \otimes A_d \right) \bpm X_1 \otimes I_\cH \\ \vdots \\ X_d \otimes I_\cH \epm \right\| \leq \| A \| _{\mr{row}} \cdot \| X \| _{\mr{col}}, $$ so that the above realization formula, Equation (\ref{realform}), can be expanded as an operator-norm convergent geometric series to yield:
\be h(X) = \sum _{n=0} ^\infty I_n \otimes b^* (X \otimes A) ^n I_n \otimes c = \sum _{\om \in \F } b^* A^\om c \, X^\om. \ee (In particular, $R_h > \| A \| _{\mr{col}} ^{-1} >0$.)

However, a uniformly analytic NC function has a realization if and only if its uniformly open NC domain includes the origin, $0:= (0, \cdots, 0) \in \C ^d \subset \ncu$. For example, the NC rational function $\fr (X_1, X_2) := [X_1, X_2] ^{-1} = (X_1 X_2 - X_2 X_1) ^{-1}$ is not defined at any scalar point $x \in \C ^d \subseteq \ncu$. Nevertheless, it does have a `matrix-centre' realization about any matrix point $Y \in \cdm$, in its NC domain, see \cite[Example 2.5]{PV1}. (Here, an NC set is any direct sum-closed subset of $\ncu$, and the domain of a uniformly analytic NC function is a uniformly open NC set.) The set of all NC rational functions is the skew field, $\fskew := \fskewd$, of all NC functions that are obtained, essentially, by applying the arithmetic operations of multiplication, addition and inversion to the free algebra, $\fp := \C \langle \fz_1, \cdots, \fz _d \rangle$, of all NC or `free' polynomials in the $d$ non-commuting variables, $\fz := (\fz _1, \cdots, \fz _d)$. Any NC rational function, $\fr$, has a finite--dimensional ``matrix-centre realization" about any matrix-point, $Y \in \mr{Dom} _m \, \fr$, in its NC domain, $\nbdom \fr = \bigsqcup _{m=1} ^\infty \nbdom _m \fr$ and the theory of matrix-centre realizations for NC rational functions has been developed in detail by Porat and Vinnikov in \cite{PV1,PV2,Pthesis}. 

In order to motivate and describe matrix-centre realizations, consider that, even in one variable, an analytic function, $h$, has a realization, $(A,b,c) \in \scr{B} (\cH) \times \cH \times \cH$, if and only if it is analytic in an open neighbourhood of $0$. The coefficients of the Taylor series of $h$ at $0$ are then given by the moments of the realization, $\hat{h} _n = b^* A^n c$, as described above. In this sense, the realization is `centred at $0$'. More generally, if $h$ is analytic in an open neighbourhood of $0 \neq y \in \C$, we say that $h \sim _y (A',b',c') \in \scr{B} (\cH) \times \cH \times \cH$ is a realization of $h$ centred at $y$ if 
$$ h(z) = b^{'*} (I-(z-y)A') ^{-1} c', $$ for all $z$ in an open neighbourhood of $y$. In particular, if $|z-y| < \| A ' \| ^{-1}$, then this realization formula can be expanded as a convergent geometric series, so that
$$ h(z) = \sum _{n=0} ^\infty b^{'*} A^{'n} c' \, (z-y) ^n, $$ is the Taylor series of $h$ at $y$. If $h \sim _0 (A,b,c)$ and $L_A (y) = I -yA$  is invertible, then one can `translate' the realization $(A,b,c)$ to the point $y$ to obtain $h \sim _y (A',b',c')$. Namely, assuming that $L_A (z)$ is invertible, 
\ba h(z) & = & b^* (I-zA) ^{-1} c = b^* (I-yA -(z-y)A) ^{-1} c \\ 
& = & b^* (I- (z-y) L_A (y) ^{-1} A ) ^{-1} L_A (y) ^{-1} c, \ea so that $h\sim _y (A' ,b',c')$ with 
$$ A' = (I-yA) ^{-1} A, \quad b' =b, \quad \mbox{and} \quad c' = (I-yA) ^{-1} c. $$ 
Similarly, if $h \sim _0 (A,b,c) \in \scr{B} (\cH) ^{1\times d} \times \cH \times \cH$ is the realization of a uniformly analytic NC function in a uniformly open neighbourhood of $0$, and $Y \in \cdn$ is a matrix point in the invertibility domain of $A$, \emph{i.e.} $L_A (Y)$ is invertible, then we can `translate' the realization at $0$ to a realization centred at the matrix-point, $Y$. Namely, if $Z \in \C ^{(mn \times mn) \cdot d}$ is any other matrix point (at a level which is a multiple of the size, $n$, of $Y$) so that $L_A (Z)$ is invertible, 
\ba h(Z) & = & I_{mn} \otimes b^* L_A (Z) ^{-1} I_{mn} \otimes c \\
& = & I_{mn} \otimes b^* \left( L_A (I_m \otimes Y) - (Z- I_m \otimes Y)\otimes A \right) ^{-1} I_{mn} \otimes c \\
& = & I_m \otimes (I_n \otimes b^*) \left( I_{mn} - (I_m \otimes L_A ( Y) ^{-1})((Z -I_m \otimes Y)\otimes A) \right) ^{-1} I_m \otimes L_A (Y) ^{-1} (I_n \otimes c). \ea 
Hence, we can define $\hat{A} _j : \C ^{n \times n} \rightarrow \scr{B} (\C ^n\otimes \cH )$, and $\hat{b},\hat{c} \in \scr{B} (\C ^n, \C ^n\otimes \cH )$ by 
$$ \hat{A} _j (G) := L_A (Y) ^{-1} (G\otimes A_j), \quad \hat{b} := I_n \otimes b, \quad \mbox{and} \quad \hat{c}:= L_A(Y)^{-1}(I_n \otimes c). $$ 

If we define the \emph{linear pencil} of $\hat{A}$ on the $d-$dimensional column NC universe over $\C ^{n\times n}$,   
\begin{align*}  
\C ^{(\N n \times \N n) \cdot d} := \bigsqcup _{m=1} ^\infty \C ^{(mn\times mn)\cdot d}, & \quad \quad \mbox{as} \\
L_{\hat{A}} (X) :=  I_{mn} \otimes I_\cH - \sum _{j=1} ^d (\mr{id} _m \otimes \hat{A} _j) (X_j); & \quad \quad X \in \C ^{(mn\times mn)\cdot d} \simeq (\C ^{m \times m} \otimes \C ^{n \times n}) \otimes \C ^d, \\
 =:  I_{mn} \otimes I_\cH - \hat{A} (X); & \quad \quad \mbox{where} \\
 \quad \quad X_j  = \bpm X_j ^{(1,1)} & \cdots  & X_j ^{(1,m)} \\ \vdots & \ddots & \\ X_j ^{(m,1)} & \cdots & X_j ^{(m,m)} \epm \in \C ^{mn \times mn}, &  \quad \quad X_j ^{(k,\ell)} \in \C ^{n \times n}, \end{align*}
$\mr{id} _m : \C ^{m \times m} \rightarrow \C ^{m \times m}$ denotes the identity map on $m \times m$ matrices and 
$$ \mr{id} _m \otimes \hat{A} _j  (X_j) = \bpm \hat{A} _j (X_j ^{(1,1)}) & \cdots  & \hat{A} _j (X_j ^{(1,m)}) \\ \vdots & \ddots & \\ \hat{A} _j (X_j ^{(m,1)}) & \cdots & \hat{A}_j (X_j ^{(m,m)}) \epm, $$ denotes the $m-$fold ampliation of $\hat{A} _j$ acting on $X_j$, then 
\be h(Z) = I_m \otimes \hat{b} ^* L_{\hat{A}} (Z-I_m \otimes Y) ^{-1} I_m \otimes \hat{c}, \label{realformula} \ee and we say that $h \sim _Y (\hat{A},\hat{b},\hat{c})$ is a matrix-centre realization of $h$ centred at $Y \in \cdn$. More generally, a realization centred at a matrix point $Y \in \cdn$, is a triple, $(A,b,c)_Y$, so that each $A_j : \C^{n\times n}\to \scr{B} (\cH) $ is linear, and $b,c \in \scr{B} (\C ^n , \cH )$. A priori, this definition has no relation to the point $Y \in \cdn$; however, when we say a realization, $(A,b,c)_Y$, is \emph{centred at $Y$}, we mean that we are considering the function, $h\sim_Y (A,b,c)$, which defined in a uniformly open neighbourhood of $Y$ and given by the realization formula (\ref{realformula}). If a function, $h$, is defined by a matrix-centre realization, $(A,b,c)_Y$, about $Y \in \cdn$ in the sense of Equation (\ref{realformula}), we say that $h$ is the \emph{quantized transfer function} of $(A,b,c)_Y$ and write $h \sim _Y (A,b,c)$. It is important to note that if $h \sim _Y (A,b,c)$ is the quantized transfer function of an arbitrary matrix-centre realization about $Y \in \cdn$, for $n>1$, it is not necessarily a non-commutative function. Namely, it can always be defined on the uniformly open NC subset, $\scr{D} ^Y (A)$, consisting of all $X \in \C ^{(mn \times mn)\cdot d}$, $m \in \N$, for which the linear pencil, $L_A (X-I_m \otimes Y)$, is invertible, it preserves the grading and direct sums, but it \emph{may not preserve joint similarities}. If a quantized transfer function is an NC function, it is automatically uniformly analytic. We will explicitly characterize when a quantized transfer function is a uniformly analytic NC function in Section \ref{sec:LAC}.

It is well-known that a uniformly analytic NC function, $f$, with $0 \in \nbdom f$, is an NC rational function, $f = \fr \in \fskew$, if and only if it admits a finite--dimensional realization, $f \sim _0 (A,b,c)$, with $A \in \cdm = \C ^{m \times m} \otimes \C ^{1\times d}$ and $b,c \in \C ^m$. In \cite{AMS-opreal}, M.L. Augat, the second author and E. Shamovich initiated the development of a general theory of not necessarily finite-dimensional operator realizations of NC functions, and showed that an NC function is uniformly analytic in a uniformly open neighbourhood of $0$ if and only if it is given by a (generally infinite--dimensional) realization \cite[Lemma 3.2]{AMS-opreal}. Given $\scr{S}\subseteq \scr{B}(\cH)^{1\times d}$, let $\scr{O} _d ^\scr{S}$ be the set of all FPS with realizations in $\scr{S}$.
Now, if $\scr{B} := \scr{B} (\cH) ^{1\times d}$, $\scr{C} := \scr{C} (\cH) ^{1\times d}$ and $\scr{T} _p := \scr{T} _p (\cH) ^{1\times d}$, $p \in [1, +\infty)$ and $\scr{F} := \scr{F} (\cH) ^{1\times d}$ where $\scr{C} (\cH)$ denotes the compact linear operators on $\cH$, $\scr{T} _p (\cH)$ denotes the Schatten $p-$class operators on $\cH$, and $\scr{F} (\cH)$ denotes the finite--rank operators on $\cH$, then the sets $\scr{O} _d ^\scr{S}$, with $\scr{S} \in \left\{ \scr{F}, \scr{T} _p, \scr{C}, \scr{B} \right\}$ are (semi-free ideal) rings that admit (universal, skew) fields of fractions, $\scr{M} _d ^\scr{S}$ obeying 
$$ \fskew = \scr{M} _d ^\scr{F} \subsetneqq \scr{M} _d ^{\scr{T} _p} \subsetneqq \scr{M} _d ^{\scr{T} _q} \subsetneqq \scr{M} _d ^\scr{C} \subsetneqq \scr{M} _d ^\scr{B}, $$ for any $1 \leq p < q < +\infty$ \cite[Theorem 6.10]{AMS-opreal}. Any element $f \in \scr{M} _d ^\scr{S}$, where $\scr{S} \in \{ \scr{F}, \scr{T} _p, \scr{C}, \scr{B} \}$ as above, defines a uniformly analytic NC function on a uniformly open NC set. We also remark that the aforementioned \cite[Lemma 3.2]{AMS-opreal} shows that $\scr{M} _d ^{\scr{B}}$ can be identified with the skew field of \emph{uniformly meromorphic NC germs} at $0$, the universal skew field of fractions of the (semi-free ideal) ring of all germs of uniformly analytic NC functions at $0 \in \C ^d$, the origin of the NC universe, as defined and studied in \cite{KVV-local}. 

Our development of realization theory about a matrix-centre in the case of infinite--dimensional realizations allows us to extend realization theory to a much broader class of uniformly analytic NC functions which need not be defined at $0$. In particular, we prove that any uniformly analytic NC function $f$, which belongs to one of the skew fields, $\scr{M} _d ^\scr{S}$, $\scr{S} \in \{ \scr{F}, \scr{T} _p, \scr{C}, \scr{B} \}$ has a matrix-centre realization, $(A,b,c)_Y$, about any point in its uniformly open NC domain, where the linear maps, $A_j : \C ^{n \times n} \rightarrow \scr{B} (\cH)$ can be chosen to take values in $\scr{F} (\cH), \scr{T} _p (\cH), \scr{C} (\cH)$ if $\scr{S} = \scr{F}, \scr{T} _p$, or $\scr{C}$, respectively. Moreover, in Corollary \ref{anyNCreal}, we significantly extend \cite[Lemma 3.2]{AMS-opreal} to show that an NC function is uniformly analytic in a uniformly open neighbourhood of $Y \in \cdn$ if and only if it has a realization centred at $Y$. This is accomplished by constructing a matricial Fock space and a ``free Hardy space about $Y$", $\hardy (Y)$, consisting of uniformly analytic NC functions in an ``open NC column-ball" centred at $Y$.  

\subsection{Outline}

The following section reviews background material on NC function theory, the uniform topology, Taylor--Taylor series and realizations. Our first main results appear in Section \ref{conobminreal}, in which we introduce a notion of ``minimality" of a matrix-centre realization extending the usual concepts of minimal, controllable and observable for realizations at $0$. Theorem \ref{uniqueness} shows that any two minimal matrix-centre realizations at $Y \in \cdn$ of the same quantized transfer function, $f$, are unique up to a closed, injective linear map with dense range, \emph{i.e.} a generally unbounded ``similarity". Theorem \ref{Kalman} extends the Kalman decomposition to general matrix-centre realizations. This constructs a minimal matrix-centre realization from any matrix-centre realization by compression to a certain ``minimal" subspace. Section \ref{sec:translate} shows that given any matrix-centre realization at $Y \in \cdn$, $(A,b,c)_Y$ can be ``translated" to any point $X \in \C ^{(mn \times mn) \cdot d}$ for which $L_A (X - I_m \otimes Y)$ is invertible to obtain a realization at $X$, $(A',b',c')_X$ which is minimal if and only if $(A,b,c)_Y$ is minimal by Theorem \ref{matrixtrans}. In Section \ref{sec:LAC}, following the work of Porat and Vinnikov in \cite{PV2} for the finite-dimensional matrix-centre realizations of NC rational functions, we express the \emph{Lost Abbey conditions at $Y$}, $Y \in \cdn$, in terms of a given matrix-centre realization, $(A,b,c)_Y$. These are set of conditions that determine when the quantized transfer function, $f \sim _Y (A,b,c)$ of a matrix-centre realization respects joint similarities, and hence is an NC  function. In Section \ref{sec:matrixFock}, we introduce a matricial Fock space and a ``free Hardy space at $Y$" and use these constructions to prove that an NC function is uniformly analytic in a uniformly open neighbourhood of $Y \in \cdn$ if and only if it has a matrix-centre realization at $Y$, see Corollary \ref{anyNCreal}. Finally in Section \ref{sec:app}, we develop an application by constructing matrix-centre realizations for elements in the skew fields of NC functions generated by realizations at $0$ belonging to certain classes of linear operators. 

\section{Background} \label{sec:back}

\subsection{The NC universe}

Consider the $d-$dimensional complex column NC universe, $\ncu := \bigsqcup _{n=1} ^\infty \cdn$, where 
$\cdn := \C ^{n \times n} \otimes \C ^d$, denotes column $d-$tuples of complex $n\times n$ matrices. That is, any $Z \in \cdn$ has the form $Z = \bsm Z_1 \\ \vdots \\ Z_d \esm$, with $Z_j \in \C ^{n\times n}$.

The NC universe, $\ncu$, can be viewed as an abstract operator space, namely the complex \emph{column operator space} with $d$ components \cite{Paulsen-cbmaps}. Recall that a concrete operator space is a subspace of the Banach space of bounded linear operators on some complex Hilbert space. Further recall that an abstract operator space is a vector space, $V$, equipped with a system of matrix--norms, $\| \cdot \| _{m, n}$, on $V ^{m \times n}$ for all $m,n \in \N$, satisfying Ruan's axioms:
Given $A \in \C^{j \times k}$, $B \in \C^{\ell \times m}$ and $X \in V ^{k \times \ell}$, $Y\in V^{a\times b}$, then 
\be \| AXB \| _{j, m} \leq \| A \|  \| X \| _{k, \ell} \| B\| , \label{Ruan1} \ee and 
\be \| X \oplus Y \|_{k+a, \ell+b}  = \max \{  \| X \| _{k,\ell}, \| Y \| _{a, b} \}. \label{Ruan2} \ee

A vector space with a system of matrix norms satisfying Equation (\ref{Ruan1}) is called a matrix--normed space, and Ruan's theorem asserts that a matrix--normed space is completely isometrically isomorphic to a concrete operator space if and only if it is an abstract operator space, \emph{i.e.} a matrix--normed space obeying Equation (\ref{Ruan2}). 

If $\scr{M}$ is any operator space, it is clear that $\scr{M} ^{m \times n}$ is also an operator space for any fixed $m,n \in \N$, and $\scr{M} ^d$ is called the \emph{column operator space} (with $d$ components) over $\scr{M}$. Namely, a system of matrix norms obeying Ruan's axioms for $(\scr{M} ^d) ^{m \times n}$  is naturally defined by identifying $(\scr{M} ^d) ^{m \times n} \simeq (\scr{M} ^{m \times n}) ^d \simeq \scr{M} ^{md \times n}$, 
$$ \| X \| _{(\scr{M} ^d) ^{m \times n}}  := \| X \| _{\scr{M} ^{md \times n}} .$$ 
It follows that if we view elements of $\cdn$ as columns, $X = \bsm X_1 \\ \vdots \\ X_d \esm$, with $X_j \in \C ^{n \times n}$, we can equip this with the system of `column' matrix norms,
$$ \| X \| _{\mr{col}} := \| X \| _{\scr{B} (\C ^n, \C ^n \otimes \C ^d)} = \| X \| _{\C ^{nd \times n}}, $$ and thus identify $\ncu$ with the $d-$component column--operator space over (the operator space) $\C$. (Alternatively, we could work with any other operator space structure on $\C ^d$, such as the system of `row' matrix norms on $\C ^d$, and view the $d-$dimensional complex universe as $\C ^d$ equipped with this operator space structure \cite{EliOrr} \cite[Section 7.2]{KVV}.)

A subset, $\Om \subseteq \ncu$ is called a \emph{non-commutative (NC) set} if it is closed under direct sums, in which case we write 
$$ \Om = \bigsqcup \Om _n; \quad \quad \Om _n := \Om \cap \cdn. $$ 
It will be useful to equip the NC universe with a natural topology, the \emph{uniform topology}, generated by a basis of NC sets. First, the \emph{column pseudo-metric} is defined as 
$$ d _{\mr{col}} (X, Y) := \| X ^{\oplus m} - Y ^{\oplus n} \| _{\mr{col}}; \quad \quad X \in \cdn, \ Y \in \cdm. $$ The uniform topology is then the topology on $\ncu$ generated by the basis (or sub-base) of open sets, $r \cdot \B ^d _{\N n} (Y)$, where $Y \in \cdn$, $r>0$ are arbitrary and 
$$ r \cdot \B ^d _{\N n} (Y) := \bigsqcup _{m=1} ^\infty \left\{ \left. X \in \C ^{(mn\times mn)\cdot d} \right| \ d _{\mr{col}} (X,Y) < r \right\}, $$ \cite[Section 7.2]{KVV}. Here, recall that as described in the introduction, an NC function defined on a uniformly open NC set $\Om \subseteq \ncu$, $f: \Om \rightarrow \C ^{\N \times \N}$ is locally bounded in the uniform topology if and only if it is continuous with respect to the uniform topologies on $\ncu$ and $\C ^{\N \times \N}$, if and only if it is uniformly analytic \cite[Corollary 7.26 and Corollary 7.28]{KVV}. 

\begin{defn}
A (descriptor) realization about the matrix-centre, $Y \in \cdn$, is a triple $(A,b,c) _Y$, where $$A = (A_1, \cdots, A_d) : \cdn \rightarrow \scr{B} (\cH)$$ is a row $d-$tuple of linear maps, $A_j : \C ^{n \times n} \rightarrow \scr{B} (\cH)$, and $b,c \in \scr{B} (\C ^n , \cH )$.
\end{defn}

By \cite[Theorem 18]{Paulsen-CB}, since each $A_j : \C ^{n\times n} \rightarrow \scr{B} (\cH)$ is a linear map on the finite--dimensional operator space $\C ^{n \times n}$, each $A_j$ is automatically completely bounded with $\| A _j \| _{\mr{CB}} \leq n \sqrt{n} \| A_j \| _{\scr{B} (\C ^{n \times n}, \scr{B} (\cH))}$. Hence we can view $A$ as a completely bounded linear map from the column operator space (with $d$ components) over $\C ^{n \times n}$ into $\scr{B} (\cH)$, $A \in \mr{CB} \left( \cdn , \scr{B} (\cH) \right)$ and we define
$$ \| A \| _{\mr{CB; row}} := \| (A_1, \cdots, A_d ) \| _{\mr{CB}}.$$
(We emphasize the fact that $A$ is completely bounded, as if $f \sim _Y (A,b,c)$ is an NC function then the completely bounded norm of $A= (A_1, \cdots, A_d)$ is related to the radius of convergence for the Taylor--Taylor series expansion of $f$ at $Y$, when we equip the NC universe with the column operator--space structure, see Subsection \ref{ss:realTT} and Remark \ref{CBreason}.)

Given such a realization, we define the linear pencil of $A \in \mr{CB} (\cdn, \scr{B} (\cH) )$ on the $d-$dimensional column NC universe over $\C ^{n\times n}$,
$$ \C ^{(\N n \times \N n) \cdot d} := \bigsqcup _{m=1} ^\infty \C ^{(mn \times mn) \cdot d}, $$ 
\be X \in \C ^{(mn \times mn) \cdot d} \ \mapsto \ L_A (X) := I_m \otimes I _\cH - \sum _{j=1} ^d (\mr{id} _m \otimes A_j) (X_j) =: I_m \otimes I_\cH - A(X). \label{lpencil} \ee
The \emph{invertibility domain} of the linear pencil, $L_A (X)$ (relative to $Y \in \cdn$), is then 
\be \scr{D} ^Y (A) := \bigsqcup _{m=1} ^\infty \scr{D} ^Y _m (A); \quad \quad \scr{D} ^Y _m (A) := \{ X \in \C ^{(mn \times mn) \cdot d} | \ L_A (X - I_m \otimes Y) ^{-1} \ \exists \}. \label{invdom} \ee
It is easy to see that $\scr{D} ^Y (A)$ is always a uniformly open NC set, since the inversion map in $\scr{B} (\cH)$ is operator--norm continuous and the set of invertible elements in $\scr{B} (\cH)$ is operator--norm open. It is not obvious, however, whether or not the invertibility domain is joint similarity invariant, in contrast to the case where $Y=0 \in \C ^d$ (or when $Y =y \in \C ^d$ is any scalar tuple). 

Any matrix--centre realization, $(A,b,c)_Y$, defines a \emph{quantized transfer function}, $f$, on the uniformly open NC set, $\scr{D} ^Y (A)$. As we will see, unlike the standard case where $Y =0$ (or any scalar point), the quantized transfer function of a matrix--centre realization at $Y$ is not necessarily an NC function. Namely, while such a quantized transfer function is always graded and respects direct sums on the uniformly open NC set, $\scr{D} ^Y (A)$, it respects joint similarities if and only if $(A,b,c)_Y$ obeys the \emph{Lost--Abbey conditions at $Y$}, LAC (Y), also called the \emph{canonical intertwining condtions} at $Y$; see Section \ref{sec:LAC} and \cite{PV1,PV2,KVV}.

Given $(A,b,c)_Y$, $Y \in \cdn$, its quantized transfer function is defined as follows. If $X \in \scr{D} ^Y _m (A) \subseteq \C ^{(mn\times mn)\cdot d}$,
$$ f(X) := I_m\otimes b^* L_A (X - I_m \otimes Y) ^{-1} I_m \otimes c \in \C ^{mn\times mn}, $$ and we write $f\sim (A,b,c)_Y$ or $f\sim _Y (A,b,c)$ in this case.

\begin{defn} \label{anequiv}
Any two matrix-centre realizations, $(A,b,c)_Y$ and $(A',b',c')_Y$ are said to be \emph{analytically equivalent at $Y$}, written $(A,b,c)_Y \sim (A',b',c')_Y$ or $(A,b,c)\sim_Y (A',b',c')$, if they define the same quantized transfer function on a uniformly open neighbourhood of $Y$.
\end{defn}

In general, even if $f \sim (A,b,c) _Y$, $g \sim (A',b',c') _Y$ and $(A,b,c) _Y \sim (A',b',c') _Y$, it may be that $f (X) \neq g(X)$ for some $X \in \scr{D} ^Y (A) \cap \scr{D} ^Y (A')$, as the following example (with $Y=0$, $d=1$) shows. 

\begin{eg}
Let $U$ be the bilateral shift on $\ell ^2 (\Z)$, and consider the realization $(U, e_0, e_0)$. Then, $\sigma (U) = \partial \D$, so that the resolvent set of $U$, $\sigma (U) ^c = \C \sm \sigma (U)$, is disconnected. Hence $\scr{D} _1 (U) = \C \sm \partial \D$ is also disconnected. Then, for $|z| <1$, 
 $$ e_0 ^* (I -zU ) ^{-1} e_0 = \sum _{j=0} ^\infty e_0 ^* U ^j e_0 z^j = 1, $$ while for $|z| > 1$, 
\ba e_0 ^* (I -z U) ^{-1} e_0 & = & - z ^{-1} e_0 ^* U^* (I - z ^{-1} U^* ) ^{-1} e_0 \\
& = & -\sum _{j=0} ^\infty z^{-j-1} e_0 ^* U^{*(j+1)} e_0 = 0. \ea 

That is, the realization $(U,e_0, e_0) \sim f$ defines the analytic function which is $1$ in $\D$ and $0$ in $\C \sm \ov{\D}$. On the other hand, the one-dimensional realization $(0,1,1)$ defines the constant function $1$ everywhere. This shows that realizations do not need to evaluate to the same analytic function on the intersection of their invertibility domains, in general.
\end{eg}

\begin{remark}
Later, we will define a realization $(A,b,c)$ to be \emph{minimal}, if $c$ is $A-$cyclic and $b$ is $A^*-$cyclic. (A minimal realization is, in some sense as `small' as possible.)  In the above example, $(U, e_0, e_0)$ is a non-minimal realization, but $(0,1,1)$ is minimal. We do not know whether one can construct an example of two minimal and analytically equivalent realizations, $(A,b,c) \sim _0 (A',b',c')$, with disconnected invertibility domains, so that they define different analytic functions on a connected component disconnected from $0$. 
\end{remark}

\begin{remark}
In general, if $f\sim (A,b,c) _Y$ and $g \sim (A',b',c') _Y$ and $(A,b,c) _Y \sim (A',b',c') _Y$, then $f(X) = g(X)$ for all $X \in \scr{D} ^{\,  \leadsto \! Y} _n   (A) \cap \scr{D} ^{\, \leadsto \! Y} _n (A')$ where $\scr{D} ^{\, \leadsto \! Y} _n (A)$ denotes the path-connected component of $\scr{D} ^Y _n (A)$ containing $Y$. This follows readily from the identity theorem in several complex variables. However, it is not clear that `global evaluation' is transitive, even in the univariate setting, and even if $Y=y \in \C$.
\end{remark}

Observe that if $f \sim _Y (A,b,c)$ is given by a matrix-centre realization at $Y \in \cdn$, that for any $X \in \C ^{(mn \times mn) \cdot d}$, 
\ba \| A (X) \| _{\scr{B} (\C ^m \otimes \cH )}  & = & \left\| \sum _{j=1} ^d (\mr{id} _m \otimes A_j) (X_j) \right\|  \\
& = & \left\| (\mr{id} _m \otimes A_1, \cdots, \mr{id} _m \otimes A_d ) \circ \bpm X_1 \\ \vdots \\ X_d \epm \right\| \\
& \leq & \| (A_1, \cdots, A_d) \| _{\mr{CB}} \left\| \bpm X_1 \\ \vdots \\ X_d \epm \right\| _{\scr{B} (\C ^{mn}, \C ^{mn} \otimes \C ^d)} = \| A \| _{\mr{CB;row}} \| X \| _{\mr{col}}. \ea 
It follows that if $X \in r \cdot \B ^d _{\N n} (Y)$, so that $\| X - I_m \otimes Y \| _{\mr{col}} < r$, where $r:= \| A \| _{\mr{CB;row}} ^{-1}$, then $\| A (X-I_m \otimes Y) \| <1$ is strictly contractive so that $L_A (X - I_m \otimes Y) ^{-1}$ can be expanded as an operator--norm convergent geometric sum. That is, for such $X$,
\begin{eqnarray} f(X) & = & I_m \otimes b^* L_A (X- I_m \otimes Y) ^{-1} I_m \otimes c \nn \\
& = & \sum _{\ell =0} ^\infty I_m \otimes b^* ( (\mr{id} _m \otimes A) \circ (X-I_m \otimes Y)) ^\ell I_m \otimes c \nn \\ 
& = & \sum _{\om \in \F } I_m \otimes b^* \left((\mr{id} _m \otimes \hat{A}) \circ (X-I_m \otimes Y) \right) ^\om I_m \otimes c, \label{realTT}
\end{eqnarray} 
where, if $\om = i_1 \cdots i_k$, then we have defined
\be \left( (\mr{id} _m \otimes \hat{A}) \circ (X-I_m \otimes Y) \right) ^\om  :=  (\mr{id} _m \otimes A_{i_1}) (X_{i_1}-I_m \otimes Y_{i_1} ) \cdots (\mr{id} _m \otimes A_{i_k}) (X_{i_k} - I_n \otimes Y_{i_k}). \label{realTTcoeff} \ee

\subsection{Realizations and Taylor--Taylor series}\label{ss:realTT}

If $\Om \subseteq \ncu$ is a uniformly open NC set containing $Y \in \cdn$, and $f: \Om \rightarrow \C ^{\N \times \N}$ is uniformly analytic on $\Om$, then consider the multi-linear maps, 
$\hat{f} _\ell (Y) [ (\cdot), \cdots, (\cdot)] : (\cdn) ^{\times ^\ell}  \rightarrow \C^{n\times n}$, where $(\cdn) ^{\times ^\ell} = \cdn \times \cdn \cdots \times \cdn$ denotes the $\ell-$fold Cartesian product of $\cdn$ with itself, defined by 
$$ \hat{f} _\ell (Y) [H ^{(1)}, \cdots, H ^{(\ell)}] := \frac{1}{\ell !} \partial _{H^{(\ell)}} \cdots \partial _{H  ^{(1)}} f  (Y), $$ where $\partial _H f (Y)$ denotes that G\^ateaux derivative of $f$ at $Y \in \Om _n = \mr{Dom} _n \, f$ in the direction $H \in \cdn$. Then the Taylor--Taylor series of $f$ at $Y$, defined for $X \in \nbdom_{mn} \, f$ by 
\ba & &   \sum _{j=0} ^\infty \frac{1}{j!} \partial ^{(j)} _{X-I_m \otimes Y} f (I_m \otimes Y) = \sum _{j=0} ^\infty \hat{f} _j(Y) [X - I_m \otimes Y]; \\
& &  \hat{f} _j(Y) [H] = \frac{1}{j!} \partial ^{(j)} _{H} f (Y) := \frac{1}{j!} \underbrace{\partial _H \cdots \partial _H}_{j \times} f (Y) = \hat{f} _j(Y) [ H, \cdots , H], \ea  converges absolutely and uniformly to $f(X)$ on a uniformly open column-ball, $r \cdot \B ^d _{m\N} (Y)$, of non-zero radius, $r>0$ \cite[Theorem 8.11]{KVV}. Namely, as in \cite[Chapter 8]{KVV}, one can define the \emph{radius of convergence}, $R_f ^Y$, of this Taylor--Taylor series centred at $Y \in \cdn$ by the Cauchy--Hadamard type quantity, 
$$ \frac{1}{R_f ^Y} := \limsup _{\ell \rightarrow \infty} \sqrt[\ell]{\| \hat{f} _\ell \| _{\mr{CB}}}, $$ where $\| \hat{f} _\ell \| _{\mr{CB}}$ is the completely bounded norm of the multi-linear map, $\hat{f} _\ell$, in the sense of Christensen and Sinclair, see \cite[p. 116]{KVV} or \cite[Chapter 17]{Paulsen-cbmaps}. Namely, each $\hat{f} _\ell$ is viewed as a completely bounded $\ell-$linear map on the column operator space over $E:= \cdn = \C ^{n\times n} \otimes \C ^{d}$ into the operator space $F:= \C ^{n \times n}$. 
To explain what this means, we first need to recall the definiton of the \emph{Haagerup tensor norm}; see \cite[Chapter 17]{Paulsen-cbmaps}. 
The Haagerup tensor norm $\|\cdot\|_h$ on $E\otimes E$ is defined as follows. (Here, $E$ can be any operator space.) First, given $A \in E^{i \times j}$ and $B \in E^{j \times k}$, their Haagerup tensor product is 
$$ A \odot B := \left(\sum _{\gamma =1} ^j A_{\alpha,\gamma} \otimes B _{\gamma, \beta}\right)_{\substack{1\leq \alpha\leq i\\ 1\leq \beta\leq k}} \in (E \otimes E ) ^{i \times k}, $$ and for $C \in (E \otimes E) ^{i \times k}$, its Haagerup tensor norm is
$$ \| C \| _h := \inf \left\{ \left. \| A \|_E \| B \|_E  \right|  \ A \in E ^{i \times j}, \ B \in E ^{j \times k}, \ A \odot B = C \right\}. $$ 
The vector space, $E\otimes E$, equipped with the system of matrix norms, $\|\cdot\|_h$, will then be an operator space, which is usually denoted by $E\otimes_h E$.
More generally, $E ^{\otimes ^\ell} = \left( \cdn \right) ^{\otimes ^\ell}$, equipped with the Haagerup tensor norm is an operator space. Here, given $C \in (E ^{\otimes ^\ell}) ^{i \times k}$, setting $m_0 :=i$ and $m_{\ell } := k$,
$$  \| C \| _h = \inf \left\{ \left. \| A ^{(1)} \| _{E} \cdots \| A ^{(\ell)} \| _E  \right|  \ A ^{(s)} \in E  ^{m_{s-1} \times m_s}, 1 \leq s \leq \ell, \ A^{(1)} \odot A^{(2)} \cdots \odot A^{(\ell)} = C  \right\}. $$
Next, consider the linear map 
\begin{align*}
    f^{(\ell)}: \underbrace{E \otimes_h \cdots \otimes_h E} _{\ell \times } \rightarrow F,
\end{align*}
which is given by
 \be f^{(\ell)} (H^{(1)} \otimes H^{(2)} \cdots \otimes H^{(\ell)}) := \hat{f} _\ell (H^{(1)}, H^{(2)}, \cdots, H^{(\ell)}); \quad \quad H ^{(i)} \in E = \cdn. \label{multilin} \ee 
We then say that $\hat{f}_\ell$ is a completely bounded $\ell-$linear map if $f^{(\ell)}$ is completely bounded, and one sets $\|\hat{f}_\ell\|_\mathrm{CB} :=\|f^{(\ell)}\|_\mathrm{CB}$. (Here, note that by a similar argument to \cite[Theorem 18]{Paulsen-cbmaps}, any $\ell-$linear map on the $\ell-$fold Cartesian product of $\cdn$ with itself is automatically completely bounded.) In this way one obtains a one-to-one correspondence between completely bounded linear and completely bounded multi-linear maps, where corresponding maps are often treated as one and the same. In particular, whenever $H^{(1)},\dots, H^{(\ell)}\in E^{m\times m}$, we can extend the definition of the multi-linear map, $\hat{f} _\ell$, from tuples of matrices in $E = \C ^{n\times n}$ to such tuples of matrices over $E$ via
\begin{align*}
    \hat{f}_\ell(H^{(1)},\dots, H^{(\ell)}) := (\mr{id} _m \otimes f^{(\ell)}) \circ (H^{(1)}\odot \cdots \odot H^{(\ell)}) \in \C ^{mn \times mn}; \quad \quad H^{(i)} \in \C ^{(mn \times mn) \cdot d}.
\end{align*}

As shown in \cite[Chapter 4]{KVV}, if $X \in \C ^{(mn \times mn) \cdot d}$, 
$$ \partial ^{(\ell)} _{X - I_m \otimes Y} f (I_m \otimes Y) = \left(\mr{id} _m \otimes \partial ^{(\ell)} _{(\cdot)} f (Y) \right)  \circ \left( X - I_m \otimes Y \right) ^{\odot ^\ell} =: \partial ^{(\ell)} _{(\cdot)} f (Y)   \circ \left( X - I_m \otimes Y \right) ^{\odot ^\ell}, $$ where $\left( X - I_m \otimes Y \right) ^{\odot ^\ell}$ denotes the $\ell-$fold Haagerup tensor product of $X-I_m \otimes Y$ with itself. This means that $\left( X - I_m \otimes Y \right) ^{\odot ^\ell}$ is viewed as an $m \times m$ block matrix where each block belongs to $(\cdn) ^{\otimes ^\ell}$, and the $\ell-$linear map, $\partial ^{(\ell)} _{(\cdot)} f (Y)$ is applied to each block. 

By \cite[Corollary 7.26 and Theorem 8.11(5)]{KVV}, if the Taylor--Taylor series for an NC function, $f$, at $Y \in \cdn$ has radius of convergence, $R_f ^Y>0$, then this series converges absolutely and uniformly to $f(X)$ on $r \cdot \B ^d _{\N n} (Y)$, for any $r< R_f ^Y$, and it fails to converge uniformly on every ball $r \cdot \B ^d _{\N n} (Y)$ for $r > R_f ^Y $. In particular, $f$ is uniformly analytic in $R_f ^Y \cdot \B ^d _{\N n } (Y)$.

Here, by absolute convergence of the Taylor--Taylor series of $f (X)$ about $Y$ to $f(X)$, where $X \in \C ^{(mn \times mn) \cdot d}$, we mean that 
$$ \sum _{\ell =0} ^\infty \| \hat{f} _\ell (Y) \circ \left(X-I_m \otimes Y \right) ^{\odot ^\ell} \| _{\C ^{mn \times mn}} < + \infty, $$ and by absolute and uniform convergence in the uniform column-ball $r \cdot \B ^d _{\N n } (Y)$ we mean that
$$ \sum _{\ell =0} ^\infty \sup _{\substack{ X \in r \cdot \B ^d _{mn} (Y); \\ m \in \N}} \| \hat{f} _\ell (Y) \circ \left(X-I_m \otimes Y \right) ^{\odot ^\ell} \| _{\C ^{mn \times mn}} < + \infty. $$

The Taylor--Taylor series of a uniformly analytic $f$ at $Y \in \cdn$ can be expanded further to a series indexed by the free monoid \cite[Chapter 4]{KVV}. Namely, one can define certain \emph{partial difference--differential operators} acting on $f$, 
$$ \Delta ^\om _{H_1, \cdots, H_{|\om|}} f (X^{(0)}, \cdots, X^{(|\om|)}); \quad \quad H_i\in \C^{n\times n}, X^{(j)} \in \cdn, \ \om \in \F, $$ see \cite[Section 3.5]{KVV}.  Each partial difference--differential of $f$, $\Delta ^\om _{H_1, \cdots, H_{|\om|}} f (X^{(0)}, \cdots, X^{(|\om|)})$, is $|\om|-$linear in the directional arguments, $H_i$, for fixed $X^{(j)}$. 

It then follows that the Taylor--Taylor series of $f (X)$ at $Y$ can be expanded further as
$$ f(X) = \sum _{\om \in \F} \Delta ^{\om ^\mrt} _{(\cdot)} f (Y) \circ \left( X - I_m \otimes Y \right) ^{\odot ^{\om }}, $$ where the short-form notation, $\Delta ^{\om ^\mrt} _{(\cdot)} f(Y)$, means that $Y$ is repeated $|\om| +1-$times, and 
\begin{align*}
    \left( X - I_m \otimes Y \right) ^{\odot ^{\om }} := (X_{i_1}- I_m\otimes Y_{i_1})\odot \cdots \odot(X_{i_{|\omega|}}-I_m\otimes Y_{i_{|\omega|}}).
\end{align*}
It can be shown that each $\Delta ^{\om ^\mrt} _{(\cdot)}f (Y) \circ H ^{\odot ^\om}$, is a `generalized monomial' in $H=(H_1,\cdots,H_d) \in \cdn$ of the form 
$$ \Delta ^{\om ^\mrt} _{(\cdot)} f (Y)\circ H^{\odot^\omega} = \sum _{j_0, \dots, j_{|\om|}} T_{j_0} H_{i_1} T_{j_1} H_{i_2} \cdots H_{i_{|\om|}} T_{j_{|\om|}}; \quad \quad \om = i_1 \cdots i_{|\om |}, $$ where each $T_{j} \in \C ^{n\times n}$ belongs to the double centralizer of the point $Y$; see \cite[Theorem 4.7]{KS-free} and \cite[Arguments following Remark 4.5]{KVV}. 

Now suppose that $f \sim _Y (A,b,c)$ is a quantized transfer function that has a matrix-centre realization at $Y \in \cdn$. Then, by Equation (\ref{realTT}), for any $X \in \scr{D} ^Y _m (A)$ with $\|X-I_m\otimes Y\|_\mathrm{col}<\|A\|_\mathrm{CB;row}^{-1}$,
$$ f(X)  =  \sum _{\om \in \F } I_m \otimes b^* \left( (\mr{id} _m \otimes \hat{A})  \circ(X-I_m \otimes Y) \right) ^\om I_m \otimes c, $$ where if $\om = i_1 \cdots i_{|\om|}$,
\begin{align*}
    \left( (\mr{id} _m \otimes \hat{A}) \circ (X-I_m \otimes Y) \right) ^\om  =  (\mr{id} _m \otimes A_{i_1}) (X_{i_1}-I_m \otimes Y_{i_1} ) \cdots (\mr{id} _m \otimes A_{i_k}) (X_{i_{|\omega|}} - I_n \otimes Y_{i_{|\omega|}}),
\end{align*} 
as defined in (\ref{realTTcoeff}). 

For simplicity of notation, we define the $|\om|-$multilinear maps, $A^\om : \C ^{n\times n} \otimes \C ^{1\times |\om |}  \rightarrow \scr{B} (\cH)$ by  
$$ A^\om (H_1, \cdots, H_{|\om |}) = A_{i_1} (H_1) \cdots A_{i_{|\om |}} (H_{|\om |});\quad\quad H_i\in\C^{n\times n}.$$ 
Next, a straightforward calculation (similar to the one on \cite[p. 65]{KVV}) shows that
\begin{align*}
     \left( (\mr{id} _m \otimes \hat{A})  \circ(X-I_m \otimes Y) \right) ^\om=(\mr{id} _m \otimes A^\om) \circ (X - I_m \otimes Y) ^{\odot ^{\om}}.
\end{align*}
It follows that $f \sim _Y (A,b,c)$ can be written as a power series of multi-linear maps as
\begin{eqnarray} f(X)  & = & \sum _{\om \in \F } I_m \otimes b^* \left( (\mr{id} _m \otimes \hat{A})  \circ(X-I_m \otimes Y) \right) ^\om I_m \otimes c  \nn \\
& = & \sum _{\om \in \F } I_m \otimes b^* (\mr{id} _m \otimes A^\om) \circ (X - I_m \otimes Y) ^{\odot ^{\om}} I_m \otimes c. \end{eqnarray}

Hence, if $f$ is a uniformly analytic NC function, then by uniqueness of Taylor--Taylor series coefficients \cite[Theorem 7.9]{KVV}, \cite[Lemma 2.12]{PV1},
\be \Delta ^{\om ^\mrt} _{(\cdot)} f (Y) \circ H^{\odot^\omega}= b^* A_{i_1}(H_{i_1})  \cdots A_{i_{|\omega|}} (H_{i_{|\om|}}) c; \quad \quad H_i\in \C^{n\times n} \label{TTunique} \ee for $\om =i_1 \cdots i_{|\om|}$. In particular, it follows that $\| A \| _{\mr{CB}; \mr{row}} ^{-1} < R_f ^Y$ by \cite[Corollary 7.26 and Theorem 8.11(5)]{KVV}. 

\begin{remark} \label{CBreason} Again, this is one reason why we view the linear maps $A_j$ as completely bounded linear maps; the completely bounded norm of $A$ in a matrix-centre realization $(A,b,c)_Y$ of an NC function, $f$, is related to the radius of convergence of the Taylor--Taylor series of $f$ at $Y$ since 
this radius of convergence is 
$$ \frac{1}{R_f ^Y} = \limsup _{\ell \rightarrow \infty} \sqrt[\ell]{\| \hat{f} _\ell \| _{\mr{CB}}}, $$ where 
$$ \hat{f} _\ell = \sum _{|\om | = \ell} b^* A^\om (\cdot ) c. $$ 
\end{remark}

\begin{lemma}
Let $(A,b,c)_Y$ and $(A',b',c')_Y$ be two realizations about the matrix centre $Y \in \cdn$. Then $(A,b,c)_Y \sim (A',b',c')_Y$ are analytically equivalent at $Y$ if and only if 
$$ b^* A^\om (G) c = b^{'*} A^{' \om} (G) c', $$ for all $\om \in \F$ and for all $G \in \C ^{n\times n} \otimes \C ^{1\times |\om|}$. 
\end{lemma}
\begin{proof}
If the quantized transfer functions of $(A,b,c)_Y$ and $(A',b',c')_Y$ are NC functions, this follows immediately from the uniqueness of Taylor--Taylor series coefficients of NC functions, \cite[Theorem 7.9]{KVV}, \cite[Lemma 2.12]{PV1} and Equation (\ref{TTunique}).

In the general case, if $b^* A^\om (G) c = b^{'*} A^{' \om} (G) c'$ for all $\om \in \F$ and all $G \in \C ^{n\times n} \otimes \C ^{1\times |\om|}$, then this clearly implies that $(A,b,c)_Y \sim (A',b',c')_Y$ since all the terms of their convergent power series expansions in the uniform row-ball $r \cdot \B ^d _{\N n} (Y)$, $r:= \min \{ \| A \| _{\mr{CB;row}} ^{-1}, \| A ' \| _{\mr{CB;row}} ^{-1} \}$ are the same. 

Conversely if $(A,b,c)_Y \sim (A',b',c')_Y$, $G = (G_1, \cdots, G _\ell)$ and $\om =i_1 \cdots i_\ell$ is an arbitrary word, we define a point $X (\om) \in \scr{D} ^Y _{\ell +1} (A) \cap \scr{D} ^Y _{\ell +1} (A')$ as follows. (The following construction is taken from \cite[Section 4.2]{AMS-opreal}.) Let $E_{j,k}$ be the standard matrix units for $\C^{(\ell +1) \times (\ell +1)}$ and set $T_\ell := \sum _{j=1} ^\ell E_{j,j+1}$, 
$$ T_\ell = \bpm 0 &1  & & \\ & \ddots & \ddots  &  \\ & & & 1 \\ & & & 0 \epm \in \C ^{(\ell +1) \times (\ell +1)}. $$ We will call $T_\ell$ the truncated backward shift matrix of size $\ell$. We then define 
$$ X (\om ) _k := I_{\ell +1} \otimes Y_k +   T_\ell \otimes I_n \sum _{j=1} ^\ell \delta _{i_j,k}  E_{j+1} \otimes r \cdot G_{i_j}, $$ where we write $E_j := E_{j,j}$ for the diagonal matrix units and we choose $r>0$ sufficiently small so that $\|X(\om)-I_{\ell+1}\otimes Y\|_\mathrm{col}< \min \{ \| A \| _{\mr{CB;row}} ^{-1}, \| A ' \| _{\mr{CB;row}} ^{-1} \}$. For example, if $d=2$, $\om =12$ and $r=1$, 
$$ X(12) _1 = \bpm Y_1 & G_1 & 0 \\ 0 & Y_1 & 0 \\ 0 & 0 & Y_1 \epm, \quad \quad X(12)_2 = \bpm Y_2 & 0 & 0 \\ 0 & Y_2 & G_2 \\ 0 & 0 & Y_2 \epm. $$ 
Then the tuple $X (\om) \in \C ^{( (\ell +1)n \times (\ell +1) n ) \cdot d}$ is \emph{jointly nilpotent at $Y$} of order $\ell +1$ in the sense that 
$$ (X (\om) - I_{\ell +1} \otimes Y) ^{\odot ^\alpha} =0, $$ for any $|\alpha | > \ell$. Moreover, if $|\alpha | = \ell >0$,
$$(X (\om) - I_{\ell +1} \otimes Y) ^{\odot ^\alpha} = r ^\ell \cdot \delta _{\alpha, \om} \, E_{1, \ell +1} \otimes \underbrace{G_{i_1} \otimes G_{i_2} \cdots \otimes G_{i_\ell}}_{=: G ^{\otimes ^\om}} = \delta _{\alpha, \om} \bpm 0 & \cdots & 0 & r ^\ell G ^{\otimes ^\om} \\ & \ddots & & 0 \\& & & \vdots \\ & & & 0 \epm. $$
Hence, 
\ba & &  I_{\ell +1} \otimes b^* L_A (X (\om) - I_{\ell +1} \otimes Y) ^{-1}  I_{\ell +1} \otimes c =  
\sum _{\alpha \in \F } (\mr{id} _{\ell +1} \otimes A^\alpha) \circ ( X (\om) - I_{\ell +1} \otimes Y) ^{\odot ^\alpha}  \\ 
& = & \sum _{|\alpha| < \ell} (\mr{id} _{\ell +1} \otimes A^\alpha) \circ ( X (\om) - I_{\ell +1} \otimes Y) ^{\odot ^\alpha} + (\mr{id} _{\ell +1} \otimes A^\om) \circ ( X (\om) - I_{\ell +1} \otimes Y) ^{\odot ^\om} \\
& = & \bpm b^*c & * & \cdots & * & b^* A^\om (r^\ell G ^{\otimes ^\om}) c \\ & \ddots &  & & * \\ & & & & \vdots \\ & & & & * \\
&  & & & b^*c \epm. \ea 
By analytic equivalence, the last formula must be equal to 
$$ I_{\ell +1} \otimes b^{'*} L_{A'} (X(\om) - I_{\ell +1} \otimes Y) ^{-1} I_m \otimes c' = 
\bpm b^{'*}c' & * & \cdots & * & b^{'*} A^{';\om} (r^\ell G ^{\otimes ^\om}) c' \\ & \ddots &  & & * \\
& & & & \vdots \\ & & & & * \\ &  & & & b^{'*}c' \epm. $$
Hence, the top right blocks of these matrices must be equal, and by multi-linearity, 
$$ r^\ell b^* A^\om ( G ^{\otimes ^\om}) c = b^* A^\om (r^\ell G ^{\otimes ^\om}) c = 
b^{'*} A^{';\om} (r^\ell G ^{\otimes ^\om}) c' = r^\ell b^{'*} A^{';\om} ( G ^{\otimes ^\om}) c'. $$ We conclude that 
\ba b^* A^\om ( G ^{\otimes ^\om}) c &= & b^* A_{i_1} (G_1) \cdots A_{i_\ell} (G_\ell) c \\
& = & b^{'*} A' _{i_1} (G_1) \cdots A' _{i_\ell} (G_\ell) c' \\
& = & b^{'*} A^{';\om} ( G ^{\otimes ^\om}) c', \ea for any word $\om \in \F$ and any $G \in \C ^{n\times n} \otimes \C ^{1\times \ell}$. 
\end{proof}

\section{Controllable, observable and minimal matrix-centre realizations} \label{conobminreal}

This section and the following Kalman decomposition section are straightforward extensions of the results of \cite[Section 3]{AMS-opreal} for realizations centred at $0$ to realizations about a matrix-centre.

Given any realization, $(A,b,c)_Y$, $Y \in \cdn$, we define the \emph{controllable subspace}, 
$$ \scr{C} _{A,c} := \bigvee _{\om \in \F} \nbran A ^\om c, $$ where, if $\om = i_1 \cdots i_{|\om |}$, then
$A^\om : \C ^{n\times n} \otimes \C ^{1 \times |\om|} \rightarrow \scr{B} (\cH)$ is the $|\om |-$multilinear map defined by 
$$ A^\om (G_1, \cdots, G_{|\om|}) := A_{i_1} (G_1) A_{i_2} (G_2) \cdots A_{i_{|\om|}} (G_{|\om|}) \in \scr{B} (\cH), \quad G_i \in \C ^{n\times n}. $$ 
Hence,
$$ \nbran A^\om c = \bigvee _{\substack{G_i \in \C ^{n \times n} \\  v \in \C ^n}} A_{i_1} (G_1) \cdots A_{i_{|\om|}} (G_{|\om|}) \, c (v). $$ The realization, $(A,b,c)_Y$, is said to be \emph{controllable} if $\cH = \scr{C} _{A,c}$.  
Similarly, we define the \emph{observable subspace}, 
$$ \scr{O} _{A^\dag, b} := \bigvee _{\om \in \F} \nbran A^{\om; \dag} b, $$ where, if $\om = i_1 \cdots i_{|\om|}$, and $G_i \in \C ^{n \times n}$ then 
$$ A^{\om ; \dag} ( G_1, \cdots, G_{|\om|} ) := A_{i_{|\om|}} (G_{|\om|} ^*) ^* A_{i_{|\om|-1}} (G_{|\om|-1} ^*) ^* \cdots A_{i_1} (G_1 ^*) ^* \in \scr{B} (\cH). $$ That is, the $|\om |-$linear map,$A^{\om; \dag}$, is defined so that 
$$ A^{\om; \dag} (\vec{G}) = (A^{\om} (\vec{G} ^*))^*; \quad \quad \vec{G} := (G_1, \cdots, G_{|\om|}), $$ and $\vec{G} ^*$ denotes component-wise adjoint.  As before, $(A,b,c)_Y$ is said to be \emph{observable}, if $\scr{O} _{A^\dag, b} = \cH$. Finally, $(A,b,c)_{Y}$ is said to be \emph{minimal}, if it is both controllable and observable. 

Given any $p \in \fp$, $p = \sum _{|\om | \leq N} \hat{p} _\om \fz ^\om$, we can view $p(A)$ as multi-affine--linear map in $n_p \in \N$ arguments in $\C ^{n \times n}$ into $\scr{B} (\cH)$, where 
$$ n_p := \sum _{\substack{\om \in \F \\ \hat{p}_\om \neq 0}} | \om |. $$ 

Namely, given two words, $\alpha, \om \in \F$ and $c \in \C$, define $A^\alpha + c A^\om : \C ^{n \times n} \otimes \C ^{1\times (|\alpha | + |\om| )} \rightarrow \scr{B} (\cH)$ in the obvious way:
$$ (A^{\alpha} + c A^\om) (G_1, \cdots, G_{|\alpha|}; H_1, \cdots H_{|\om |} ) := A^\alpha (G_1, \cdots, G_{|\alpha|}) + c A^\om (H_1, \cdots, H_{|\om|}), $$
for $G_i, H_j \in \C ^{n \times n}$. For simplicity of notation, when we view $p(A)$ for $p \in \fp$ as a multi-affine--linear map in $n_p$ copies of $\C^{n\times n}$, we will often simply write $p(A)[\vec{G}]$ for the image of $\vec{G} := (G_1, \cdots, G_{n_p}) \in \C^{n\times n} \otimes \C ^{1 \times n_p}$ under $p(A)$ and we employ the notation:

$$ \nbran \C \langle A \rangle c := \bigvee _{p \in \fp} \nbran p(A)c. $$ 
It is then clear that 
$$\nbran \C \langle A \rangle c = \bigvee _{\om \in \F} \nbran A^\om c  = \scr{C} _{A,c}. $$ 

We will say that a closed and densely--defined linear map, $S : \nbdom S \subseteq \cH \rightarrow \cH '$ is a \emph{pseudo-similarity}, if it is injective and has dense range. (Hence it has a closed and densely--defined inverse, $S^{-1} : \nbran S \rightarrow \cH$, which is then also a pseudo-similarity.)
The following theorem is a matrix-centre analogue of \cite[Theorem 3.6]{AMS-opreal} and the proof is similar. 

\begin{thm}[Uniqueness of minimal realizations] \label{uniqueness}
Suppose that $(A,b,c) _Y \sim (A',b',c')_Y$ are minimal analytically equivalent realizations at $Y \in \cdn$ that take values in $\cH$ and $\cH'$, respectively. Then there is a unique pseudo-similarity, $S: \nbdom S \subseteq \cH \rightarrow \cH '$ so that:
\bi
    \item[(i)] $\nbran \C \langle A \rangle c$ is a core for $S$ and $S p(A) [\vec{G}] c(v) = p(A')[\vec{G} ] c' (v)$ for all $\vec{G} \in \C^{n\times n} \otimes \C ^{1 \times n_p}$ and $v \in \C ^n$.  
    \item[(ii)] $\nbran \C \langle A^{'\dag} \rangle b'$ is a core for $S^*$ and $S^* q(A^{'\dag}) [\vec{G}] b'(u) = q(A^\dag) [\vec{G}] b (u)$ for all $\vec{G}=(G_1, \cdots, G_{n_q})$, $G_i \in \C ^{n \times n}$ and $u \in \C ^n$. 
\ei
\end{thm}
\begin{proof}
We need to first check that $S$ is a well-defined linear map. Namely, if $p(A)[\vec{G}] c(v) =0 \in \cH$, we need to show that $p(A') [\vec{G}] c' (v) =0 \in \cH'$. For any $q \in \fp$, let $m=n_q$. Then, for any $H=(H_1, \cdots H_m)$, $H_j \in \C ^{n \times n}$ and $u \in \C ^n$, since $(A,b,c) \sim_Y (A',b',c')$,
\ba 0 & = & b(u) ^* \ov{q}(A) [\vec{H}^*] p(A) [\vec{G}] c(v) := \ip{q(A ^\dag) [\vec{H}] b(u)}{p(A)[\vec{G}]c(v)}_\cH \\
& = & b' (u) ^* \ov{q}(A') [\vec{H}^*] p(A') [\vec{G}] c(v) = \ip{q(A^{';\dag}) [\vec{H}] b'(u)}{p(A')[\vec{G}]c'(v)}_{\cH '}, \ea where if $q (\fz) = \sum \hat{q} _\om \, \fz ^\om$, then $\ov{q} (\fz) := \sum \ov{\hat{q} _\om} \, \fz ^\om$, is obtained by taking the complex conjugate of each coefficient.
However, since $(A',b',c')_Y$ is assumed to be minimal, 
$$ \cH ' = \scr{O} _{A^{'\dag}, b'} = \bigvee _{\om \in \F} \nbran A^{' \om; \dag} b = \bigvee _{q \in \fp} \nbran q(A^{';\dag}) b', $$ and we conclude that 
$p(A')[\vec{G}]c' (v) =0$, as required.

Since $(A,b,c)_Y$ and $(A',b',c')_Y$ are both minimal realizations, they are controllable,
$$ \cH = \scr{C} _{A,c} = \bigvee p(A) [\vec{G}] c (v), $$ and $\cH' = \scr{C} _{A',c'}$. Hence the non-closed linear span, $\nbdom S_0 := \bigvee p(A)[\vec{G}] c(v)$, is a dense domain in $\cH$, and $S_0 : \nbdom S_0 \twoheadrightarrow \bigvee p(A')[\vec{G}] c'(v)$ is a densely-defined linear map with dense range in $\cH$. To show that $S_0$ is closeable, it suffices to show that $S_0 ^*$ is densely-defined. Given any $q \in \fp$, consider 
\ba \ip{q(A^{' \dag})[\vec{H}] b'(u)}{S_0 p(A)[\vec{G}]c(v)}_{\cH'} & = & \ip{q(A^{' \dag})[\vec{H}] b'(u)}{p(A')[\vec{G}]c'(v)}_{\cH'} \\
& = & \ip{q(A^\dag)[\vec{H}] b(u)}{p(A)[\vec{G}]c(v)}_{\cH}. \ea It follows that 
$$ q(A^{' \dag})[\vec{H}] b'(u) \in \nbdom S_0 ^*, \quad \mbox{and} \quad S_0 ^* q(A^{' \dag})[\vec{H}] b'(u) = q(A^{ \dag})[\vec{H}] b(u), $$ for any $\vec{H} \in \C^{n\times n} \otimes \C ^{1 \times n_q}$ and any $u \in \C ^n$. Since $(A',b',c')_Y$ is also observable, it follows that $S_0 ^*$ is densely-defined, so that $S_0$ is closeable with closure $S$ and $S^* = S_0^*$. Proof of the second statement, (b), follows from a symmetric argument. 
\end{proof}

\section{The Kalman decomposition}

Given $(A,b,c)_Y$, $Y\in \cdn$, we define the \emph{minimal subspace} as the space,
$$ \scr{C} _{A,c} \ominus (\scr{C} _{A,c}  \cap \scr{O} _{A ^\dag,b} ^\perp). $$ 
Observe that since
$$ \scr{C} _{A,c} = \bigvee _{\om \in \F} \nbran A^\om c, $$ where $A^\om$ takes values in $\scr{B} (\cH)$ and $c$ takes values in $\cH$, that $\scr{C} _{A,c}$ is invariant for any $A^\om [\vec{G}] \in \nbran A^\om$, $\vec{G} = (G_1, \cdots, G_{|\om|} )$, $G_i \in \C ^{n\times n}$. Similarly, $\scr{O} _{A^\dag, b}$ is co-invariant for each element of $\nbran A^\om$. It follows that $\scr{M} _{A,b,c}$ is $\nbran A^\om-$semi-invariant so that if $Q_0$ denotes orthogonal projection onto $\scr{M}_{A,b,c} \subseteq \cH$, and we define $(A^{(0)}, b_0, c_0)$ by $A^{(0)} _j (G):= Q_0 A_j (G) | _{\scr{M}_{A,b,c}}$, $G \in \C ^{n\times n}$, $b_0 = Q_0 b$ and $c_0 = Q_0 c$, then for any $\om = i_1 \cdots i_{|\om|} \in \F$, 
$$ Q_0 A_{i_1} (G_1) Q_0 \cdots Q_0 A_{i_{|\om|}} (G_{|\om|} ) Q_0 = 
Q_0 A_{i_1} (G_1)  \cdots A_{i_{|\om|}} (G_{|\om|}) Q_0 = Q_0 A^\om (G) Q_0, $$ by semi-invariance \cite{Sarason}. The following theorem is a matrix-centre analogue of \cite[Theorem 3.12]{AMS-opreal}.

\begin{thm}[Kalman decomposition] \label{Kalman}
Let $f \sim (A,b,c)_Y$ be a realization centred at $Y \in \cdn$ and let $Q_0$ be the orthogonal projection onto $\scr{M} _{A,b,c} \subseteq \cH$. Then $(A^{(0)},b_0,c_0)_Y$ is minimal and $(A^{(0)}, b_0, c_0) \sim _Y (A,b,c)$ are analytically equivalent realizations that define the same quantized transfer function, $f$, on the level-wise path-connected component of $Y$ in $\scr{D}  ^Y (A) \cap \scr{D} ^Y  (A^{(0)})$, $\scr{D} ^{\, \leadsto \! Y}  (A) \cap \scr{D} ^{\, \leadsto \! Y}  (A^{(0)})$.

\end{thm}
\begin{proof}
Given the realization, $(A,b,c)_Y$, centred at $Y \in \cdn$, let $(A',b',c')_Y$ be the controllable realization obtained by `restricting' $(A,b,c)_Y$ to the $A-$invariant controllable subspace $\scr{C} _{A,c} \subseteq \cH$. That is, if $P'$ is the orthogonal projection onto $\scr{C} _{A,c}$, then $A' _j := A _j| _{\scr{C} _{A,c,}}$, $1\leq j \leq d$, $b:= P' b$ and $c' = P' c = c$. Then $(A',b',c')_Y$ is a controllable realization on $\cH' = \scr{C} _{A,c}$ since 
\ba \scr{C} _{A',c'} & = & \bigvee _{\om \in \F} \nbran A^{'\om}c' \\
& = & \bigvee _{\substack{\vec{G} \in \C^{n\times n} \otimes \C ^{1\times |\om|} \\ v \in \C ^n}} A^{' \om} [\vec{G}] c' (v)  \\
& = & \bigvee (P' A P') ^{\om} [\vec{G}] c' (v) \\
& = &  \bigvee A^\om [\vec{G}] \underbrace{P'c (v)}_{=c(v)} = \scr{C} _{A,c} = \cH ', \ea
using the $A-$invariance of $P'$.
In the above, if $\om = i_1 \cdots i_{|\om|} \in \F$ and $\vec{G} = (G_1, \cdots, G_{|\om |})$, $G_i \in \C ^{n\times n}$, the notation  $(P' A P') ^\om [\vec{G}]$ denotes 
$$ P' A_{i_1} (G_1) P' A_{i_2} (G_2) P' \cdots P' A_{i_{|\om |}} (G_{|\om |}) P'. $$ 

Moreover, $(A,b,c)\sim_Y (A',b',c')$ since for any $\om \in \F$ and $\vec{G} \in \C ^{n\times n} \otimes \C ^{1\times |\om|}$, 
\ba b^{'*} A^{' \om} [\vec{G}] c'  & = & b^* P' (P'A P') ^\om [\vec{G}] P' c \\
& = & b^* A^\om [\vec{G}] P' c = b^* A^\om [\vec{G}] c. \ea  

We next define a new realization by compressing the controllable realization, $(A',b',c')_Y$, to its $A'$-co-invariant observable subspace, $\scr{O} _{A^{'\dag},b'}$. We first claim that $\scr{M} := \scr{O} _{A^{'\dag},b'} = \scr{M} _{A,b,c}$ is the minimal subspace of the original realization, $(A,b,c)_Y$. Indeed, $x \in \scr{C} _{A,c} \ominus \scr{O} _{A^{'\dag},b'}$, if and only if for any $\om \in \F$, 
\ba 0 & = & \ip{A^{'\dag;\om} [\vec{H}] b' (u)}{x} = \ip{(P' AP') ^{\dag; \om}[\vec{H}] P' b(u)}{x} \\
& = & \ip{P' A^{\dag;\om} [\vec{H}] b(u)}{x} = \ip{A^{\dag; \om} [\vec{H}] b(u)}{x}, \ea since $P'x =x$. It follows that $x \in \scr{C} _{A,c} \ominus \scr{O} _{A^{'\dag}, b'}$ if and only if $x \in \scr{C} _{A,c}$ is orthogonal to $\scr{O} _{A^\dag, b}$. This proves that 
$$\scr{O} _{A^{'\dag},b'} = \scr{C} _{A, c} \ominus \left( \scr{C} _{A,c} \cap \scr{O} _{A^\dag , b} ^\perp \right) = \scr{M} _{A,b,c}, $$ as claimed. 

Finally, let $Q_0$ be the orthogonal projection of $\cH$ onto the minimal subspace, $\scr{M} _{A,b,c}$, and consider the realization $(A^{(0)}, b_0,c_0)_Y$ obtained by compressing $(A,b,c)_Y$ to this minimal space.
Namely, $A^{(0)} _j (G) := Q_0 A_j (G) | _{\scr{M} _{A,b,c}}$, $b_0 = Q_0b$ and $c_0 = Q_0 c$. First we claim that $(A^{(0)},b_0,c_0) \sim _Y (A,b,c)$ is a matrix-centre realization of the same NC function. Indeed,
\ba b_0 ^* A^{(0);\om} [\vec{G}] c_0 & = & b^{'*} Q_0 (Q_0 A' Q_0 ) ^\om [\vec{G}] Q_0 c' \\
& = & b^{'*} Q_0 A^{'\om} [\vec{G}] Q_0 c' \\
& = & b^{'*} Q_0 A^{'\om} [\vec{G}] c ' \quad \quad \mbox{($Q_0$ is $A'-$co-invariant)} \\
& = & b^{'*}  A^{'\om} [\vec{G}] c' \quad \quad (b' = Q_0 b' = b_0) \\
& = & b^* A^\om [\vec{G}] c, \ea since $(A',b',c') \sim _Y (A,b,c)$ as calculated above. Moreover,
$$ \bigvee A^{(0) \dag;\om}[\vec{G}] b_0 (u)  =  \bigvee A^{'\dag; \om}[\vec{G}] b' (u) = \scr{M} _{A,b,c}, $$ follows since $b_0 = b'$, $A^{(0) \dag ; \om } [\vec{G}] = A^{(0); \om} [\vec{G} ^*]^*$, $\scr{M} _{A,b,c}$ is $A'-$co-invariant and $A^{'\dag} _j (G) | _{\scr{M}_{A,b,c}} = A^{(0)\dag} _j (G)$. Similarly,
\ba \bigvee A^{(0); \om} [\vec{G}] c_0 (v) & = & \bigvee Q_0 A^{'\om} [\vec{G}] Q_0 c' (v) \\
& = & Q_0 \bigvee A^{'\om} [\vec{G}] c' (v) = Q_0 \bigvee A^\om [\vec{G}] c (v) \\
& = & Q_0 \scr{C} _{A,c} = \scr{M} _{A,b,c}, \ea so that $(A^{(0)},b_0,c_0) _Y$ is both controllable and observable, hence minimal. 

If $h \sim_Y (A,b,c) \sim _Y (A^{(0)}, b_0, c_0)$, the final statement follows from the identity theorem in several complex variables. 
\end{proof}

\begin{cor} \label{Kalcompact}
If $A$ takes values in compact operators, $\scr{C} (\cH)$, then $\scr{D} ^Y (A^{(0)}) \supseteq \scr{D} ^Y (A)$.
\end{cor}
We will use the following lemma.
\begin{lemma}{\cite[Lemma 3.16]{AMS-opreal}}
Let $T \in \scr{B} (\cH)$ have a closed, invariant subspace, $\cJ$, so that with respect to the orthogonal direct sum decomposition, $\cH = \cJ \oplus \cJ ^\perp$, 
$$ T \simeq \bpm T_1 & T_2 \\ 0 & T_3 \epm. $$
If $T$ is invertible then $T_1$ is injective, $T_1 ^*$ is surjective on $\cJ$, and $T_3 \simeq P_\cJ ^\perp T | _{\cJ ^\perp}$ is surjective on $\cJ ^\perp$. Moreover, $T$ and $T_1$ are both invertible if and only if $T$ and $T_3$ are both invertible, in which case 
$$ T^{-1} = \bpm T_1 ^{-1} & -T_1 ^{-1} T_2 T_3 ^{-1} \\ 0 & T_3 ^{-1} \epm, $$ so that $\cJ$ is also $T^{-1}-$invariant.
\end{lemma}
\begin{proof}[Proof of Corollary \ref{Kalcompact}]
If $X \in \scr{D} _m ^Y (A)$, $X \in \C ^{(mn \times mn) \cdot d}$, then with respect to the orthogonal decomposition, 
$$ \scr{H} = \scr{M} _{A,b,c} \oplus ( \scr{C} _{A,c} \ominus \scr{M} _{A,b,c} ) \oplus \scr{C} _{A,c} ^\perp, $$ we have that,
$$ A _j (X_j) =\begin{pmatrix}   \begin{array}{cc}  \cellcolor{blue!15} A ^{(0)} _j (X_j) & \cellcolor{green!15} 0   \\ \cellcolor{green!15} *  & \cellcolor{green!15} A^{(1)} _j (X_j) \end{array} &  \mbox{\Large $*$} \\  \begin{array}{cc}   0 \quad   &   \quad 0   \\ 0  \quad   &  \quad 0 \end{array} &  \mbox{\Large $A ^{(2)} _j (X_j)$ } \epm. $$

Hence, the entire operator, $A_j (X_j)$, is block upper triangular with respect to the orthogonal decompostion $\cH = \scr{C} _{A,c} \oplus \scr{C} _{A,c} ^\perp$, while the upper left block corresponding to the decomposition of $\scr{C} _{A,c}$ is block lower triangular. Let, $A' _j (X_j)$ denote this upper left block so that 
$$ A_j (X_j) = \bpm A' _j (X_j) & * \\ 0 & A ^{(2)} _j (X_j) \epm. $$ 
By the previous lemma, if $X \in \scr{D} _m ^Y (A)$, so that $L_A (X)$ is invertible, it follows that $L _{A'} (X)$ is injective and $L _{A^{(2)}} (X)$ is surjective. However, if $L_{A'} (X) = I_m \otimes I_{\cH'} - A' (X)$ is not invertible, it is not surjective, so that $\la =1$ belongs to the spectrum of $A'(X)$. Since each $A_j (X_j)$ is compact, it follows that $A' (X) = \sum _{j=1} ^d A' _j (X_j)$ is compact, so that any non-zero point in its spectrum is an eigenvalue. This contradicts injectivity of $L_{A'} (X)$ and we conclude that $L_{A'} (X)$ is invertible. Similarly, since $L_{A'} (X)$ is invertible, we must have that $L_{A^{(0)}} (X)$ is surjective and $L_{A^{(1)}} (X)$ is injective by the previous lemma. Again, both $A^{(0)} (X)$ and $A^{(1)} (X)$ are compact, so that both $L_{A^{(0)}} (X)$ and $L_{A^{(1)}} (X)$ must be invertible. Hence $X \in \scr{D} _m ^Y (A^{(0)})$ and $\scr{D} ^Y (A) \subseteq \scr{D} ^Y ( A^{(0)} )$.
\end{proof}

\section{Translations} \label{sec:translate}

\subsection{Translations of one-dimensional realizations}

Let $f\sim _0 (A,b,c) \in \scr{B}(\cH) \times \cH \times \cH$ be a one-dimensional realization centred at $0 \in \C$. We will simply write $(A,b,c) = (A,b,c)_0$, in this case.  Then, as described in the introduction, for any $z \in \scr{D} _1 (A) \subseteq \C \sm \{ 0 \}$, we can construct a realization of $f$ centred at $z$, $f \sim  (A',b',c') _{z}$,
$$ A' := (I-z A) ^{-1} A, \quad b' = b, \quad \mbox{and} \quad c' = (I-z A) ^{-1} c. $$
\begin{thm}[Minimal translations]
If $f \sim (A,b,c) \in \scr{B} (\cH) \times \cH \times \cH$ is a minimal realization at $0 \in \C$, then $f \sim _{z} (A',b',c')$ is also a minimal realization of $f$ at $z \in \C \sm \{ 0 \}$. 
\end{thm}
\begin{proof}
We prove controllability, observability is analogous. The controllable subspace of $(A',b',c')$ is
$$ \scr{C} _{A',c'} = \bigvee _{j=0} ^\infty A^{'j} c' = \bigvee _{j=0} ^\infty A^j (I-zA)^{-j-1} c. $$ 
Hence, $(I-zA) ^{-1} c \in \scr{C} _{A',c'}$, and 
$$ zA(I-zA) ^{-2} c = (I-zA) ^{-2}c - (I-zA) ^{-1} c \in \scr{C} _{A',c'}. $$ In particular it follows that $\scr{C} _{A',c'}$ contains the linear span of $(I-zA) ^{-2} c$ and $(I-zA) ^{-1}c$. Iterating this argument yields
$$ \scr{C} _{A',c'} \supseteq \bigvee _{j=1} ^\infty (I-zA) ^{-j} c. $$ Set $T := (I-zA) ^{-1}$, a bounded, invertible operator and choose any $\la \in \C$ so that $| \la | > \| T \|$. Then, $\la I - T$ is invertible and 
$$ (\la I - T ) ^{-1} = \sum _{n=0} ^\infty \la ^{-n-1} T ^n, $$ is a convergent geometric series since $\| \la ^{-1} T \| = |\la | ^{-1} \| T \| < 1$. Hence, 
$$ \scr{C} _{A',c'} \supseteq \bigvee _{|\la| > \| T \|} (\la I -  T ) ^{-1} T c. $$ Since $T$ is invertible, $0 \notin \sigma (T)$, so that the functions $e_{-n} (\la) := 1/\la^n$, $n \in \N$ are analytic in an open set containing the spectrum of $T$, $e_{-n} \in \scr{O} (\sigma (T))$. Hence, by the Riesz--Dunford holomorphic functional calculus,
$$ T ^{-n} = \frac{1}{2\pi i} \cint _{r\cdot \partial \D} \la ^{-n} (\la I - T ) ^{-1} d\la; \quad \quad r > \| T \|, $$ so that 
\ba \scr{C} _{A',c'} & \supseteq & \bigvee _{n \in \N} T ^{-n} T c \\
& = & \bigvee _{n=0} ^\infty T ^{-n} c. \\
& = & \bigvee _{n=0} ^\infty (I-zA) ^{n} c. \ea 
Hence, $c \in \scr{C} _{A',c'}$ and 
$$ (I-zA) c = c -zAc \in \scr{C} _{A',c'}, $$ so that $Ac \in \scr{C} _{A',c'}$. Continuing in this manner,
$$ (I-zA) ^{2} c = c -2zA C +z^2 A^2 c \in \scr{C} _{A',c'}, $$ so that $A^2 c \in \scr{C} _{A',c'}$, and we conclude that 
$$ \scr{C} _{A',c'} \supseteq \bigvee _{n=0} ^\infty A^n c = \scr{C} _{A,c} = \cH, $$ so that $(A',b',c')_{z}$ is controllable. 
\end{proof}

\subsection{Translations of matrix-centre realizations} \label{ss:mtrans}

More generally, given a matrix-centre realization $f \sim (A,b,c)_Y$, with $Y \in \cdn$ and $X,Z \in \scr{D} _m ^Y (A)$, we have that 
\ba f(Z) & = & I_m \otimes b^* L_A (Z - I_m \otimes Y) ^{-1}  I_m \otimes c  \\
& = & I_m \otimes b^* (I_m \otimes I_\cH - A (Z-I_m\otimes Y)  ) ^{-1} I_m \otimes c \\
& = & I_m \otimes b^* ( I_m \otimes I - (\mr{id} _m \otimes A) (Z-X) - (\mr{id} _m \otimes A) (X-I_m \otimes Y) ) ^{-1} I_m \otimes c \\ 
& = & I_m \otimes b^* \left( I_m \otimes I - L_A (X-I_m \otimes Y) ^{-1} (\mr{id} _m \otimes A) (Z-X) \right) ^{-1} L_A (X-I_m \otimes Y) ^{-1} I_m \otimes c \\
& = : & b^{'*} L_{A'} (Z-X) ^{-1} c',  \ea 
where $b', c' \in \scr{B} (\C ^{mn}, \cH')$, $A' _j : \C ^{mn \times mn} \rightarrow \scr{B} (\cH ')$ are linear maps, $\cH' := \C ^m \otimes \cH$, and for any $G \in \C ^{mn \times mn}$,
$$ A' _j (G) := L_{A} (X - I_m \otimes Y) ^{-1} (\mr{id}_m \otimes A_j) (G), \quad b' = I_m \otimes b, \quad \mbox{and} \quad c' = L_A (X- I_m \otimes Y) ^{-1} I_m \otimes c. $$ Here, for any $u \in \C ^m$ and $v \in \C ^n$, $b'$ is defined on $\C ^{mn} \simeq \C ^m \otimes \C ^n$ via 
$$ I_m \otimes b (u \otimes v) := u \otimes b (v) \in \C ^m \otimes \cH. $$ 
This is essentially the same calculation as in \cite[Section 2.7, Theorem 2.21]{PV1}. Hence we have shown that $f \sim (A',b',c') _X$.

\begin{thm}[Minimal matrix-centre translations] \label{matrixtrans}
If $f \sim (A,b,c) _Y$ is a minimal realization at $Y \in \cdn$, and $X \in \scr{D} ^{Y} _m (A)$  then $(A',b',c')_{X}$ is also a minimal realization of $f$ at $X$ so that $\scr{D} ^X _k (A') = \scr{D} ^Y _{km} (A)$ for all $k \in \N$.
\end{thm}

In the case where $(A,b,c)_Y$ and $(A',b',c')_X$ are finite--dimensional matrix-centre realizations of an NC function (and hence are matrix-centre realizations of an NC rational function), this was established by Porat and Vinnikov in the proof of \cite[Theorem 2.21]{PV1}. Since our realizations are generally infinite--dimensional, a different proof approach is required.

\begin{proof}
The controllable subspace of $(A',b',c')_X$ is 
$$ \scr{C} _{A',c'} = \bigvee _{\substack{ \om \in \F; \ v \in \C ^{mn} \\ \vec{G} \in \C ^{mn\times mn} \otimes \C ^{1 \times |\om |}}} \left( L_A (X - I_m \otimes Y) ^{-1} A\right) ^\om (\vec{G}) L_A (X-I_m \otimes Y) ^{-1} (I_m \otimes c) (v), $$ and we want to prove that this is equal to $\C ^{m} \otimes \cH$. Setting $T := A (X - I_m \otimes Y)$, $L := (I - T )  = L_A (X - I_m \otimes Y)$ and $\La := L ^{-1} $, we have that 
$$ \scr{C} _{A',c'} \supseteq \bigvee _{\substack{ \om \in \F; \ v \in \C ^{mn} \\ \vec{G} \in \C ^{mn\times mn} \otimes \C ^{1 \times |\om |} \\ 1 \leq j \leq d }} (\La A ) ^\om (\vec{G}) \La (\mathrm{id}_m\otimes A_j )(X_j - I_m \otimes Y_j)  \La  (I_m \otimes c) (v), $$ and hence, by summing over $j$,
$$ \scr{C} _{A',c'} \supseteq \bigvee _{\substack{ \om \in \F; \ v \in \C ^{mn} \\ \vec{G} \in \C ^{mn\times mn} \otimes \C ^{1 \times |\om |} }} (\La A ) ^\om (\vec{G}) \La  \underbrace{A (X - I_m \otimes Y)}_{=T}  \La  (I_m \otimes c) (v). $$ Iterating this argument yields that
\be \scr{C} _{A',c'} \supseteq \bigvee _{\substack{ \om \in \F; \ v \in \C ^{mn} \\ \vec{G} \in \C ^{mn\times mn} \otimes \C ^{1 \times |\om |} \\ j \in \N \cup \{0 \}}} (\La A ) ^\om (\vec{G}) (\La T)^j \La (I_m \otimes c). \label{powers} \ee
Using that $T$ and $\La = (I-T) ^{-1}$ commute and $T \La = T (I-T)^{-1} = (I-T) ^{-1} - I = \La - I$, the previous formula becomes
$$ \mbox{(\ref{powers})}= \bigvee _{\substack{ \om \in \F; \ v \in \C ^{mn} \\ \vec{G} \in \C ^{mn\times mn} \otimes \C ^{1 \times |\om |} \\ j \in \N}} (\La A ) ^\om (\vec{G}) \La ^{j} (I_m \otimes c). $$   In particular, it follows as before that 
$$ \scr{C} _{A',c'} \supseteq \bigvee _{\substack{ \om \in \F; \ v \in \C ^{mn} \\ \vec{G} \in \C ^{mn\times mn} \otimes \C ^{1 \times |\om |} \\ |\la | > \| \La \|}} (\La  A ) ^\om (\vec{G}) (\la I - \La ) ^{-1} \La (I_m \otimes c) (v), $$ since for any $\la \in \C$ so that $|\la | > \| \La \|$, 
$$ (\la I - \La )^{-1} = \sum _{n=0} ^\infty \la ^{-n-1} \La ^n, $$ converges in operator norm. Again, by the Riesz--Dunford functional calculus, since $\La$ is invertible, $0 \notin \sigma (\La)$ and $1/ \la ^j$ is holomorphic in an open neighbourhood of the spectrum of $\La$ so that 
\begin{eqnarray} \scr{C} _{A',c'} & \supseteq & \bigvee _{h \in \scr{O} (\sigma (\La))} (\La  A ) ^\om (\vec{G}) h (\La ) \La (I_m \otimes c) (v) \nn \\
& \supseteq & \bigvee _{j=1} ^\infty (\La  A ) ^\om (\vec{G}) \La ^{-j}  \La (I_m \otimes c) (v) \nn \\
& = & \bigvee _{j=0} ^\infty (\La  A ) ^\om (\vec{G}) \La ^{-j}  (I_m \otimes c) (v) \nn \\
& \supseteq & \bigvee (\La  A ) ^\om (\vec{G})  (I_m \otimes c) (v) \label{step1} \\
& \supseteq & \bigvee (I_m \otimes c) (v). \nn \end{eqnarray} 
We now iterate this argument. Fix any $1 \leq k \leq d$ and any $G \in \C^{mn\times mn}$. By Equation (\ref{step1}),
\ba \scr{C} _{A',c'} & \supseteq & \bigvee (\La  A ) ^\om (\vec{G})  (I_m \otimes c) (v) \\
& \supseteq & \bigvee  (\La  A ) ^\om (\vec{G}) \La (\mr{id} _m \otimes A_k) (G) (I_m \otimes c) (v) \\
& \supseteq &  \bigvee _{\substack{1\leq j \leq d; \ v \in \C ^{mn} \\ G \in \C ^{mn \times mn}}}(\La  A ) ^\om (\vec{G}) \La (\mr{id} _m \otimes A_j) ( X_j - I_m \otimes Y_j) \La (\mr{id} _m \otimes A_k) (G) (I_m \otimes c) (v) \\ 
& \supseteq & \bigvee  (\La A) ^\om (\vec{G}) \La T \La  (\mr{id} _m \otimes A_k) (G) (I_m \otimes c) (v), \ea
where the last formula is obtained by summing over $1\leq j \leq d$. Iterating this argument then yields
$$ \scr{C} _{A',c'}  \supseteq  \bigvee _{j=0} ^\infty (\La  A ) ^\om (\vec{G}) (\La T) ^j \La (\mr{id} _m \otimes A_k) (G) (I_m \otimes c) (v). $$ Again, we now use that $\La, T$ commute and that $\La T = \La - I$ to conclude that 
$$ \scr{C} _{A',c'}  \supseteq  \bigvee _{j=1} ^\infty (\La  A ) ^\om (\vec{G}) \La ^j (\mr{id} _m \otimes A_k) (G) (I_m \otimes c) (v), $$ and hence as before,
\ba
\scr{C} _{A',c'} & \supseteq &  \bigvee _{|\la| > \| \La \| } (\La  A ) ^\om (\vec{G}) (\la I - \La ) ^{-1} \La  (\mr{id} _m \otimes A_k) (G) (I_m \otimes c)  (v) \\
& \supseteq & \bigvee _{k=0} ^\infty (\La  A ) ^\om (\vec{G}) \La ^{-j} (\mr{id} _m \otimes A_k) (G) (I_m \otimes c) (v) \\
& \supseteq & \bigvee  (\La  A ) ^\om (\vec{G})  (\mr{id} _m \otimes A_k) (G) (I _m \otimes c) (v). \ea 
In particular, 
$$ \C ^m \otimes \bigvee _{\substack{G \in \C ^{n\times n} \\ v \in \C ^n}} A_j (G) c (v) + \C ^m \otimes \bigvee _{v \in \C ^n} c (v) \subseteq \scr{C} _{A',c'}, $$ and iterating this argument shows that for any $\om \in \F$, $\scr{C} _{A',c'}$ contains $\C ^m \otimes \nbran A^\om c.$ In conclusion, 
$$ \C ^m \otimes \cH = \C ^m \otimes \scr{C} _{A,c} = \C ^m \otimes \bigvee _{\substack{ \om \in \F; \ v \in \C ^{n} \\ \vec{G} \in \C ^{n\times n} \otimes \C ^{1 \times |\om |}}} A^\om (\vec{G}) c (v) \subseteq \scr{C} _{A',c'}, $$ so that $(A',b',c')_X$ is controllable. Proof of observability now follows from a symmetric, analogous argument and is omitted.

The final statement will follow from our construction of $(A',b',c') _X$. If $Z \in \scr{D} ^Y _{km} (A)$, then, since $X \in \scr{D} ^Y _m (A)$, by assumption,
\ba L_A (Z - I_{km} \otimes Y) ^{-1} & = & \left( I _{\cH} \otimes I _{km} - A (I _k \otimes X - I _{km} \otimes Y) - A (Z - I_k \otimes X) \right) ^{-1} \\
& = & \left( I \otimes I - I_k \otimes A (X - I_m \otimes Y) - A (Z - I_k \otimes X) \right) ^{-1} \\
& = & \left( I \otimes I - I_k \otimes L_A (X-I_m \otimes Y) ^{-1} A (Z-I_k \otimes X) \right) ^{-1} I_k \otimes L_{A} ( X-I_m \otimes Y ) ^{-1} \\
& = & L_{A'} (Z-I_k \otimes X) ^{-1} I_k \otimes L_A (X-I_m \otimes Y) ^{-1}. \ea That is, under the assumption that $X \in \scr{D} ^Y _{m} (A)$ so that $L_A (X - I_{m} \otimes Y)$ is invertible, invertibility of $L_A (Z - I_{km} \otimes Y)$ is equivalent to invertibility of $L_{A'} (Z-I_k \otimes X)$. That is, $Z \in \scr{D} ^Y _{km} (A)$ if and only if $Z \in \scr{D} ^X _k (A')$. 
\end{proof}

\section{NC transfer functions and the Lost-Abbey conditions}\label{sec:LAC}

Given a power series of multi-linear maps on $\cdn$ of the form, 
$$ \wt{f} (X) := \sum _{\ell =0} ^\infty \hat{f} _{\ell} \circ \left( X - I_m \otimes Y \right) ^{\odot ^\ell}, $$ where $X \in \C ^{(mn \times mn) \cdot d}$, each $\hat{f} _\ell$ is $\ell-$linear, and $\odot$ denotes the Haagerup tensor product for the $d-$dimensional column operator space over $\C^{n\times n}$, one can ask when this is equal to a uniformly analytic NC function, assuming it converges absolutely in $r \cdot \B ^d _{\N n} (Y)$ for some $r>0$. The answer, given in \cite[Chapter 8, Theorem 8.11]{KVV}, is that it is neccessary and sufficient that the sequence of multi-linear maps obey the \emph{Lost--Abbey conditions} or \emph{canonical intertwining conditions} at $Y$, denoted by $\mr{LAC}(Y)$ or $\mr{CIC} (Y)$. Namely, given any $H^{(i)} \in \cdn$, and $T \in \C ^{n\times n}$,
\be [T, \hat{f} _0]  = \hat{f} _1 \left( [ T, Y] \right), \ee
and for any $\ell \geq 1$,
\begin{align}
 &T \hat{f} _\ell ( H^{(1)}, \cdots, H^{(\ell)} ) - \hat{f} _\ell (T H^{(1)}, \cdots, H^{(\ell)} )  =  
 \hat{f} _{\ell +1} ( [T,Y], H^{(1)}, \cdots, H^{(\ell)} ), \\
 &\hat{f} _\ell ( H^{(1)}, \cdots, H ^{(j-1)}, H^{(j)}T, H^{(j+1)}, \cdots, H^{(\ell)} ) -  
\hat{f} _\ell ( H^{(1)}, \cdots, H^{(j)}, T H^{(j+1)}, H^{(j+2)}, \cdots, H^{(\ell)} )  \nonumber \\
&= \hat{f} _{\ell +1} ( H^{(1)}, \cdots, H ^{(j)}, [T,Y], H^{(j+1)}, \cdots H^{(\ell)} ), \quad \quad \mbox{and}  \\
&\hat{f} _\ell ( H^{(1)}, \cdots, H^{(\ell)} T ) - \hat{f} _\ell ( H^{(1)}, \cdots, H^{(\ell)} )T = 
\hat{f} _{\ell +1} ( H^{(1)}, \cdots, H^{(\ell)}, [T,Y] ). \end{align} In the above $[T,Y] = TY-YT = ([T,Y_1], \cdots, [T, Y_d])$ denotes the component-wise commutator. 

Equivalently, the Lost--Abbey conditions at $Y$ can be given for the terms of a power series of multilinear maps indexed by the free monoid,
$$ \sum _{\om \in \F} \hat{f} _\om \circ \left( X - I_m \otimes Y \right) ^{\odot ^{\om}}, $$ where each $\hat{f} _\om$ is $|\om |-$linear. The conditions become: Given $H_i \in \C ^{n\times n}$ and setting $\ell := | \om |$, $\om \neq \emptyset$,
\begin{align}
& [T, \hat{f} _\emptyset] = \sum _{j=1} ^d \hat{f} _j ([T, Y_j]) \label{LACy1} \\
 &T \hat{f} _\om ( H_1, \cdots, H_\ell ) - \hat{f} _\om (T H_1, \cdots, H_\ell )  =  
 \sum _{j=1} ^d \hat{f} _{j \om} ( [T,Y_j], H_1, \cdots, H_\ell ) \label{LACy2} \end{align}
\begin{align}
 &\hat{f} _\om ( H_1, \cdots, H_{k-1}, H_{k} T, H_{k+1}, \cdots, H_{\ell} ) -  
\hat{f} _\om ( H_1, \cdots, H_k, T H_{k+1}, H_{k+2}, \cdots, H_\ell )  \nonumber \\
&= \sum _{j=1} ^d \hat{f} _{i_1 \cdots i_k  j i_{k+1} \cdots i_\ell } ( H_1, \cdots, H_k, [T,Y_j], H_{k+1}, \cdots H_\ell ) \label{LACy3}  \\
&\hat{f} _\om ( H_1, \cdots, H_\ell T ) - \hat{f} _\om ( H_\om, \cdots, H_\ell )T = 
\sum _{j=1} ^d \hat{f} _{\om j} ( H_1, \cdots, H_\ell, [T,Y_j] ), \label{LACy4} \end{align}
see \cite[Equations (4.14--4.17)]{KVV}, \cite[Equations (1.13--1.16)]{PV2}.

If $\wt{f} \sim _Y (A,b,c)$ is the quantized transfer function of a matrix-centre realization at $Y \in \cdn$, recall that it can expanded as a power series of multi-linear maps,
\ba \wt{f} (X) & = & I_m \otimes b^* L_A (X- I_m \otimes Y) ^{-1} I_m \otimes c  \\
& = & \sum _{\om \in \F} I_m \otimes b^* (\mr{id} _m \otimes A ^\om) \circ (X - I_m \otimes Y) ^{\odot ^\omega} I_m \otimes c, \ea as described in Subsection \ref{ss:realTT}. Assuming $(A,b,c)_Y$ is minimal, we can then write the Lost Abbey conditions for the sequence of multi-linear maps, $(b^* A^\om (\cdot ) c )_{\om \in \F}$, in terms of the realization, $(A,b,c)_Y$. The theorem below was proven for finite--dimensional matrix-centre realizations of NC rational functions by Porat and Vinnikov in \cite[Lemma 2.2, Theorem 2.6]{PV2}. We provide the proof below, for the convenience of the reader, although the proof is essentially the same as the proofs in \cite{PV2}. (The reference \cite{PV2} works with finite--dimensional Fornasini--Marchesini matrix-centre realizations while the theorem below is stated in terms of descriptor matrix-centre realizations which are generally infinite--dimensional. We have not yet introduced Fornasini--Marchesini realizations, we will do this in Section \ref{sec:FM}.) 

\begin{thm} \label{realLAC}
Let $f \sim _Y (A,b,c)$, where $(A,b,c) _Y$ is a minimal matrix-centre realization at $Y \in \cdn$. Then $f$ is an NC function if and only if $(A,b,c)_Y$ obeys the \emph{linearized Lost Abbey conditions} at $Y$:
For any $G, H, T \in \C ^{n\times n}$,
\begin{align}
[T, b^*c] &= b^* A([T,Y]) c, \label{LAC1} \\
Tb^* A_i (H) - b^* A_i (TH) & = b^* A ( [T,Y] )  A_i (H), \label{LAC2} \\
A_i (HT)c - A_i (H)c T & = A_i (H) A([T,Y])c, \quad \quad \mbox{and} \label{LAC3}\\
A_i (GT)A_j (H) - A_i (G) A_j (TH) & = A_i (G) A ([T,Y])A_j (H). \label{LAC4} \end{align}
\end{thm}
In the above theorem statement, we have used the short-form notation, 
$$ A ([T,Y]) := \sum _{j=1} ^d A_j ([T,Y_j]). $$  We will refer to the above four Equations (\ref{LAC1}--\ref{LAC4}) as the \emph{linearized Lost Abbey conditions} at $Y$, LLAC(Y). It follows from this theorem, that if $(A,b,c)_Y$ is minimal and obeys the linearized Lost-Abbey conditions at $Y$, then $f \sim _Y (A,b,c)$ is a uniformly analytic NC function in a uniform column-ball centred at $Y$ of radius at least $\| A \| _{\mr{CB; row}} ^{-1}$.

\begin{proof}[Proof of Theorem \ref{realLAC}]
As in Subsection \ref{ss:realTT}, assuming that $f \sim _Y (A,b,c)$ is a uniformly analytic NC function, then by uniqueness of Taylor--Taylor series, if 
$$ f(X) = \sum _{\om \in \F}  \hat{f} _\om \circ \left(X - I_m \otimes Y \right) ^{\odot ^\omega} $$ is the Taylor--Taylor series of $f$ at $Y$, then as an $|\om |-$linear map 
$$ \hat{f} _\om ( H_1, \cdots, H_{|\om |}) = \Delta ^{\om ^t} _{{H_1}, \cdots, H_{|\om |}} f(Y) = b^* A^\om (H_1, \cdots, H_{|\om |}  ) c, $$ where $\om = i_1 \cdots i_{|\om |}$. Hence, the sequence of $|\om|-$linear maps, $(b^* A^\om (\cdot) c)_{\om \in \F}$, must obey the Lost--Abbey conditions at $Y$, as stated in Equations (\ref{LACy1}--\ref{LACy4}). 

The first condition (\ref{LACy1}), $[T, \hat{f} _\emptyset] = \hat{f} _1 ([T,Y])$ becomes 
$$ [T,b^*c] = \sum _{j=1} ^d b^* A_j ([T,Y_j]) c = b^* A ([T,Y])c, $$ which is Equation (\ref{LAC1}). Given $G,T \in \C ^{n\times n}$, $\om \in \F$ and $H \in \C^{n\times n} \otimes \C ^{1 \times |\om |}$, the second condition (\ref{LACy2}), 
$$ T \hat{f} _{k\om} (G, H) - \hat{f} _{k\om} (TG,H) = \sum _{j=1} ^d \hat{f} _{jk\om} \left([T, Y_j], G, H \right), $$ then becomes 
$$ T b^* A_k (G) A^\om (H) c - b^* A_k (TG) A^\om (H) c = b^* A([T,Y]) A _k (G)  A^\om (H) c. $$ 
By controllability of $(A,b,c) _Y$, we conclude that 
$$ T b^* A_k (G)  - b^* A_k (TG)  = b^* A([T,Y]) A_k (G), $$ so that condition (\ref{LAC2}) holds. The final two conditions, (\ref{LAC3}) and (\ref{LAC4}) follow similiarly using observability and minimality of $(A,b,c)_Y$, respectively. 

Conversely, we claim that if $f\sim _Y (A,b,c)$, where $(A,b,c)_Y$ is minimal and obeys the linearized $\mr{LAC} (Y)$, Equations (\ref{LAC1}--\ref{LAC4}), then $f$ is a uniformly analytic NC function in a uniform column-ball centred at $Y$. Since $f$ is the quantized transfer function of a matrix-centre realization, it respects direct sums and the grading, so that it suffices to prove that $f$ respects joint similarities. (This will imply that $f$ is an NC function. Since $f \sim _Y (A,b,c)$, it is automatically uniformly bounded in $r \cdot \B ^d _{\N n } (Y)$, $r= \| A \| _{\mr{CB;row}} ^{-1}$, hence uniformly analytic in this uniformly open NC set.) Hence, assume that $X \in \scr{D} ^Y _m (A) \subseteq \C ^{(mn \times mn) \cdot d}$, $S \in \mr{GL}_{mn}$ is invertible and $X' := S^{-1} X S \in \scr{D} ^Y _m (A)$. First, we observe that the Lost--Abbey conditions for the realization $(A,b,c)_Y$ can be extended to $T \in \C ^{mn \times mn}$ in a natural way, by ampliating the realization. For example, given such a $T$, Equation (\ref{LAC1}) becomes
$$ [T, (I_m \otimes b^*c)] = I_m \otimes b^* (\mr{id}_m \otimes A) ( [T, I_m \otimes Y]) I_m \otimes c.$$ In the following argument we will assume, for simplicity of notation, that $X \in \scr{D} ^Y _1 (A) \subseteq \cdn$ and $S \in \mr{GL} _n$, the argument is easily extended to the general case. 

We will prove that $f(X)S = S f(X')$. First,
$$ f(X) S = b^*c S + b^* L_A (X-Y) ^{-1} A(X-Y)cS. $$ Applying condition (\ref{LAC1}) to $b^*c S$ yields
$$ b^*c S = S b^*c - b^* A ([S,Y])c. $$ Now consider $A(X-Y) cS$, 
$$ A (X-Y) c S = A ((X-Y)S)c - A(X-Y) A([S,Y])c, $$ by condition (\ref{LAC3}). Hence,
\ba f(X) S & = & Sb^*c - b^* A ([S,Y])c + b^* L_A (X-Y) ^{-1} A((X-Y)S) c - b^* L_A (X-Y) ^{-1} A(X-Y) A([S,Y])c \\ 
& = & Sb^*c - b^* A([S,Y])c + b^* L_A (X-Y) ^{-1} A((X-Y)S) c - b^* L_A (X-Y) ^{-1} A ([S,Y])c + b^* A ([S,Y]) c \\
& = & Sb^*c + b^* L_A (X-Y) ^{-1} A((X-Y)S) c - b^* L_A (X-Y) ^{-1} A ([S,Y])c \\
& = & Sb^*c + b^* L_A (X-Y) ^{-1} \left( A((X-Y)S) - A ([S,Y]) \right)c. \ea 
Observe that 
$$ A ((X-Y)S) - A([S,Y]) = A (SX' - SY), $$ by linearity of $A$, so that the previous equation becomes
\be f(X) S = S b^*c + b^* L_A (X-Y) ^{-1} A (S(X' - Y))c. \label{expr} \ee 
Now consider the expression
\be A ((X-Y)S) A (X' - Y) c - A (X-Y) A ([S,Y]) A(X'-Y)c. \label{expr1} \ee Applying condition (\ref{LAC4}) to the first term of this expression yields
$$  A ((X-Y)S) A (X' - Y) c  =  A (X-Y) A(S (X'-Y))c + A (X-Y) A([S,Y])A(X'-Y)c. $$
Hence Equation (\ref{expr1}) becomes
\ba (\mbox{\ref{expr1}}) & = & A (X-Y) A(S (X'-Y))c \\
& = & - L_A (X-Y)A(S(X'-Y))c + A (S (X'-Y))c. \ea 
Rearranging this last equation then yields that 
\ba A (S (X'-Y))c &=& L_A (X-Y) A(S(X'-Y)) c + A ((X-Y)S)A(X'-Y)c - A(X-Y) A([S,Y]) A(X'-Y)c \\
& = & L_A (X-Y) A(S(X'-Y))c + A((X-Y)S)A(X'-Y)c \\ & &  +L_A (X-Y) A([S,Y])A(X'-Y)c - A([S,Y])A(X'-Y)c \\
& = & L_A (X-Y) \left( A (S(X'-Y)) + A([S,Y])A(X'-Y) \right)c + \left( A((X-Y)S) - A ([S,Y])\right)A(X'-Y)c. \ea 
We now substitute this identity for $A(S(X'-Y))c$ into Equation (\ref{expr}) to obtain
\ba f(X)S - Sb^*c & = & b^* L_A (X-Y) ^{-1} A (S (X'-Y))c \\
& = & b^* A (S (X'-Y))c + b^* A([S,Y]) A(X'-Y)c \\ & & + b^* L_A (X-Y) ^{-1} \underbrace{\left( A ((X-Y)S)-A([S,Y]) \right)}_{=A(S(X'-Y))} A(X'-Y)c \\
& = & b^* A (S (X'-Y))c + b^* A([S,Y]) A(X'-Y)c + b^* L_A (X-Y) ^{-1} A(S(X'-Y)) A(X'-Y)c. \ea 
Applying condition (\ref{LAC2}) to the first term in this final formula yields
\begin{eqnarray} f(X) S - S b^*c & = & S b^* A (X'-Y)c - b^* A ([S,Y]) A(X'-Y)c + b^* A([S,Y]) A (X'-Y)c \\ & & + b^* L_A (X-Y) ^{-1} A (S (X'-Y)) A(X'-Y)c \nonumber \\ 
& = & S b^* A (X'-Y) c + b^* L_A (X-Y) ^{-1} A (S (X'-Y)) A(X'-Y)c. \label{expr2} \end{eqnarray} 

We now apply condition (\ref{LAC4}) to obtain the expression
$$ A(X-Y) A(S(X'-Y))  =  A ((X-Y)S) A(X'-Y) - A(X-Y) A ([S,Y]) A (X'-Y). $$ Hence,
\ba A (S (X'-Y)) & = & L_A (X-Y) A (S (X'-Y))+ A (X-Y) A(S(X'-Y)) \\
& = & L_A (X-Y) A (S (X'-Y)) + A ((X-Y)S)A(X'-Y) - A(X-Y) A([S,Y]) A (X'-Y) \\
& = & A ((X-Y)S)A(X'-Y) +L_A (X-Y) A (S (X'-Y)) \\ & & 
+ L_A (X-Y) A ([S,Y]) A(X'-Y) - A([S,Y])A(X'-Y) \\
& = & \left( A ((X-Y)S) - A ([S,Y]) \right) A(X'-Y) + L_A (X-Y) \left( A ([S,Y]) A(X'-Y) + A (S (X'-Y)) \right) \\
& = & A (S (X'-Y)) A(X'-Y) + L_A (X-Y)  \left( A ([S,Y]) A(X'-Y) + A (S (X'-Y)) \right). \ea 
Hence,
\begin{eqnarray} b^* L_A (X-Y) ^{-1} A (S(X'-Y)) & = & b^* L_A (X-Y) ^{-1}A (S (X'-Y)) A(X'-Y) \nonumber \\ 
& & + b^* \left( A ([S,Y]) A(X'-Y) + A (S (X'-Y)) \right). \label{expr5} \end{eqnarray} 
Applying condition (\ref{LAC2}) to the final term gives
$$ b^* A (S (X'-Y)) = S b^* A (X'-Y) - b^* A([S,Y])A(X'-Y). $$
Substituting this into Equation (\ref{expr5}) then gives 
$$ b^* L_A (X-Y) ^{-1} A (S (X'-Y)) = b^* L_A (X-Y) ^{-1} A (S (X'-Y)) A (X'-Y) + S b^* A (X'-Y). $$
Rearranging this expression yields
$$ b^* L_A (X-Y) ^{-1} A (S (X'-Y)) \underbrace{(I - A(X'-Y))}_{=L_A (X'-Y)}  = S b^* A (X'-Y), $$ or equivalently,
\be b^* L_A (X-Y) ^{-1} A (S (X'-Y)) = S b^* A(X'-Y) L_A (X'-Y) ^{-1} = Sb^* L_A (X'-Y) ^{-1} - Sb^*. \label{expr6}\ee
Finally, substituting the previous expression (\ref{expr6}) into expression (\ref{expr}) gives
\ba f(X) S & = & S b^*c + b^* L_A (X-Y) ^{-1} A (S (X'-Y))c  \\
& = & S b^*c + Sb^* L_A (X'-Y) ^{-1}c - Sb^*c \\
& = & S b^* L_A (X'-Y) ^{-1} c = S f(X'). \ea 
\end{proof}

\begin{remark}
In \cite[Theorem 3.5]{PV2}, Porat and Vinnikov further show that if a minimal and \emph{finite--dimensional} matrix-centre realization at $Y \in \cdn$ obeys the linearized Lost-Abbey conditions at $Y$, then the invertibility domain, $\scr{D} ^Y (A)$, is joint similarity invariant. (Recall $\scr{D} ^Y (A)$ is always a uniformly open NC set, regardless of whether the realization is minimal or obeys LAC(Y).) Their proof, however, seems to rely on the assumption that their realizations take values in a finite--dimensional space, as they make use of so-called finite-rank \emph{controllability} and \emph{observability operators} that are right and left invertible, respectively, under the assumption that the realization is minimal. So far we have not been able to extend this approach. While we can formally define these controllability and observability operators in the infinite dimensional settings, it is not immediately clear to us why these operators should even be closeable linear maps. 
\end{remark}

\section{Fornasini--Marchesini matrix-centre realizations} \label{sec:FM}

It will be sometimes convenient to consider a second type of realization, \emph{Fornasini--Marchesini} (FM) \emph{realizations}. An FM realization at a matrix-point $Y \in \cdn$, is a quadruple, $(A,B,C,D)_Y$, where, as before, $A = (A_1, \cdots, A_d) : \cdn \rightarrow \scr{B} (\cH)$ is a row $d-$tuple of linear maps, $B = \bsm B_1 \\ \vdots \\ B_d \esm \in \scr{B} ( \C ^{n\times n}, \scr{B} (\C ^n , \cH) ) \otimes \C ^d$, $C \in \scr{B} (\cH, \C ^n )$, and $D \in \C ^{n \times n}$. That is, each $B_j : \C ^{n\times n} \rightarrow \scr{B} (\C ^n , \cH)$ is a bounded linear map. The linear pencil of $A$, $L_A$, and its invertibility domain (at $Y$), $\scr{D} ^Y (A) \subseteq \C ^{(\N n \times \N n)\cdot d}$ are defined as before, and $(A,B,C,D)_Y \sim h$ defines the \emph{quantized transfer function} on $\scr{D} ^Y (A)$ by the formula:
$$ f (X) = I_m \otimes D + I_m \otimes C L_A (X-I_m \otimes Y) ^{-1} \underbrace{\sum _{j=1} ^d B_j (X_j - I_m \otimes Y_j)}_{=: B (X - I_m \otimes Y)}; \quad \quad X \in \scr{D} ^Y _m (A) \subseteq \C ^{(mn \times mn) \cdot d}. $$
If $(A,B,C,D)_Y$ is a matrix-centre realization at $Y \in \cdn$, we define, as before, the \emph{controllable subspace}, 
$$ \scr{C} _{A,B} := \bigvee _{\substack{\om \in \F \\ 1 \leq j \leq d}} \nbran A^\om B_j, $$ and the \emph{observable subspace}, 
$$ \scr{O} _{A^\dag, C^*} := \bigvee _{\om \in \F} \nbran A^{\om ; \dag} C^*. $$ Again, as before, we say that $(A,B,C,D)_Y$ is minimal if it is both controllable and observable. (Both Theorems \ref{uniqueness} and \ref{Kalman} as well as all of the results of previous sections are readily adapted to FM realizations.)

As described in \cite[Section 3]{AMS-opreal} (for the case of realizations at $0$) one can construct an FM realization from a descriptor one and vice versa. Namely, let $(A,b,c)_Y$ be a descriptor realization at $Y \in \cdn$. Then define the $A-$invariant subspace 
$$ \cH' := \bigvee_{\om \neq \emptyset} \nbran A^\om c, $$ with orthogonal projection $P'$, $$ A' _j ( G ) := A _j (G) | _{\cH '}, \quad  B_j (G) := A_j (G) c,  \quad C := b^*P', \quad \mbox{and} \quad D := b^*c \in \C ^{n\times n}, $$ where $G \in \C ^{n\times n}$. It is easy to check that minimality of $(A,b,c)_Y$ implies minimality of $(A',B,C,D)_Y$ and both of these realizations define the same quantized transfer function in a uniform row-ball of positive radius centred at $Y$. 

Conversely, if $(A,B,C,D)_Y$ is an FM realization about $Y \in \cdn$, let $\hat{\cH} := \cH \oplus \C ^n$ and define
$$ \hat{A}_j (G) := \bpm A_j (G) & B_j (G) \\ 0 & 0 \epm, \quad b := C^* \oplus D^*, \quad \mbox{and} \quad c =0 \oplus I_n. $$ Then $(\hat{A},b,c)_Y$ is a descriptor realization centred at $Y$, which defines the same quantized transfer function in a uniformly open neighbourhood of $Y$. If $(A,B,C,D)_Y$ is controllable, then, 
\ba \scr{C} _{\hat{A},c} & = & \bigvee _{\om \in \F } \hat{A} ^\om (\vec{G}) c (v) \\
& = & \bigvee _{\substack{G \in \C ^{n \times n}; v \in \C ^n \\ 1 \leq j \leq d} } B_j (G) (v) + \bigvee _{\substack{\om = \om ' j \in \F \\ \om ' \in \F, \ 1 \leq j \leq d}} A^{\om'} (\vec{G}) B_j (G) (v) \\
& = & \scr{C} _{A, B}, \ea so that $(\hat{A},b,c)$ is also controllable.  If $(A,B,C,D)_Y$ is observable, then $(\hat{A},b,c)_Y$ need not be observable, but its observable subspace has codimension at most $n$.

\subsection{The realization algorithm} \label{ss:realalg}

Also as in \cite[Theorem 2.4]{PV1} and \cite[Subsection 3.4]{AMS-opreal}, if $f \sim _Y (A,B,C,D)$ and $g \sim _Y (A', B', C', D')$, then $f+g$, $f\cdot g$ and, assuming $f(Y) =D$ is invertible, $f^{-1}$, all have matrix-centre realizations about $Y$ given by a natural extension of the Fornasini--Marchesini algorithm for realizations at $0$. Namely, a matrix-centre realization for $f+g$ is given by $(A ^+ , B ^+ , C ^+, D ^+)$ where
\be A _j ^+ := A _j \oplus A' _j, \quad B _j ^+ := \bpm B _j \\ B ' _j \epm, \quad C ^+ := (C , C ' ), \quad \mbox{and} \quad D ^+ := D + D'. \label{matFMsum} \ee A matrix-centre realization for $f\cdot g$ is given by $(A^\times , B^\times, C ^\times, D ^\times )$, where 
\be A ^\times _j := \bpm A _j & B _j (\cdot ) C '  \\ & A _j ' \epm, \quad B ^\times := \bpm B _j (\cdot ) D ' \\ B _j ' \epm, \quad C ^\times := ( C , D C ' ), \quad \mbox{and} \quad D ^\times := D D '. \label{matFMmult} \ee Finally, a matrix-centre realization for $f^{-1}$, assuming that $D = f (Y)$ is invertible, is given by \\ $(A ^{(-1)}, B ^{(-1)}, C ^{(-1)}, D ^{(-1)} )$, where
\be A ^{(-1)} _j := A _j - B _j (\cdot ) D ^{-1} C, \quad B _j ^{(-1)} := -B _j (\cdot ) D ^{-1}, \quad C ^{(-1)} := D ^{-1} C, \quad \mbox{and} \quad D ^{(-1)} := D ^{-1}. \label{matFMinv} \ee Verification of these formulas is a straightforward computation, and we refer the reader to \cite[Theorem 2.4]{PV1} for their proofs (in the NC rational/ finite--dimensional setting). 

\begin{remark}
An analogous algorithm exists for descriptor realizations centred at $Y \in \cdn$. In our opinion the formulas for the FM algorithm are simpler and easier to work with in many calculations. Another advantage of the FM algorithm over the descriptor realization algorithm is that if $(A,B,C,D)_Y$ minimal (and $D \in \C ^{n\times n}$ is invertible) then the realization $(A^{(-1)}, B^{(-1)}, C^{(-1)}, D ^{(-1)})_Y$ is also minimal. (This is not the case for the corresponding descriptor algorithm formula for the realization of the inverse of the quantized transfer function.)
\end{remark}

\section{Every uniformly analytic NC function has a matrix-centre realization} \label{sec:matrixFock}

In \cite[Lemma 3.2]{AMS-opreal} it was shown that an NC function, $h$, is uniformly analytic in a uniformly open neighborhood of $0$ if and only if it has a realization, $h \sim _0 (A,b,c)\in \scr{B}(\mathcal{H})^{1\times d} \times \mathcal{H}\times\mathcal{H}$. The proof of the existence of such a realization, given such a uniformly analytic $h$, was accomplished with the aid of the full Fock space over $\C ^d$. Here, given any complex Hilbert space, $V$, the full Fock space over $V$ is the Hilbert space
defined by
\begin{align*}
    \scr{F}(V) := \bigoplus _{\ell =0} ^\infty V ^{\otimes ^\ell} = \C\oplus V \oplus (V\otimes V) \oplus (V\otimes V\otimes V) \oplus \cdots.
\end{align*}
One can identify the full Fock space over $V=\C ^d$ with the \emph{free Hardy space} of free formal power series in $d$ NC variables with square--summable coefficients:
\begin{align*}
    \hardy := \left\{ \left.  h(\fz)=\sum_{\omega\in\F} \hat{h}_\omega \fz^\omega \right| \ \sum_{\omega\in\F}|\hat{h}_\omega|^2<+\infty \right\}.
\end{align*}
Any $h\in \hardy$ has radius of convergence at least $1$, so that $\hardy$ can be viewed as a space of uniformly analytic NC functions on the unit row ball $\mathbb{B}^{(\N\times \N)\cdot d}_\text{row}$ or the unit column ball $\B ^{(\N \times \N) \cdot d}$. Equipping $\hardy$ with the $\ell^2$-norm of the coefficients turns it into a Hilbert space, which generalizes the classical Hardy space $H^2$. Moreover, $\hardy$ is an NC reproducing kernel Hilbert space (NC-RKHS); see \cite{NCRKHS} for background on NC reproducing kernel Hilbert spaces. Denote by $\mathbb{H}_d^\infty$ the unital Banach algebra of uniformly bounded, and hence uniformly analytic NC functions in the unit row-ball $\mathbb{B}^{(\N\times \N)\cdot d}_\text{row}$ equipped with the supremum norm. We say that a uniformly analytic NC function, $f$, in the NC unit row-ball, is a \emph{left multiplier} of $\hardy$, if $f \cdot h \in \hardy$ for every $h \in \hardy$. It is clear that the set of all left multipliers is a unital algebra, and, by the closed graph theorem, any left multiplier, $f \in \mr{Mult} ^L  ( \hardy )$, defines a bounded \emph{left multiplication operator} on the free hardy space via $h \mapsto fh$, $h \in \hardy$. We will later need the following fact, see \cite[Theorem 3.1]{Pop-freeholo} or \cite[Theorem 3.1]{SSS}.

\begin{prop}\label{prop:MultiplierHinfty}
    $h\in \mathbb{H}_d^\infty$ if and only if $h$ belongs to the left multiplier algebra of $\mathbb{H}_d^2$. 
\end{prop}

Note that since $1 \in \hardy$, it follows that $\mult \subseteq \hardy$.
In the remainder of this section we will prove a similar result for NC functions which are uniformly analytic in a uniformly open neighborhood of some matrix point $Y$. Namely, we will prove that an NC function is uniformly analytic in a uniformly open neighbourhood of a matrix point $Y \in \cdn$ if and only if it is given by a matrix-centre realization at $Y$.

Our approach is similar to the one mentioned above, as we will make use of a matrix-valued version of the full Fock space and define a ``free Hardy space centred at $Y$", $\hardy (Y)$, consisting of uniformly analytic NC functions in the unit uniform column-ball centred at $Y$, $\B ^d _{\N n} (Y)$, which have ``square--summable Taylor--Taylor series at $Y$". Throughout this section let $Y\in \cdn$ be a fixed matrix point.

\subsection{A matricial Fock space} \label{sec:matrixFockspace}

\begin{defn}\label{defn:matrixFockspace}
View $\C^{n\times n}$ as a Hilbert space by equipping it with the Hilbert--Schmidt inner product and define
\ba
    \scr{F}_n(\C ^d) & := & \C ^{n\times n} \otimes \scr{F} \left( \C ^d \otimes \C ^{n\times n} \right) \\
    & = & \C^{n\times n}\oplus (\C^{n\times n}\otimes \C ^d \otimes \C^{n\times n})\oplus(\C^{n\times n}\otimes \C ^d \otimes \C^{n\times n}\otimes \C ^d \otimes\C^{n\times n})\oplus \cdots . \ea
\end{defn}

If $\{ e_1,\cdots,e_d \}$ is the standard orthonormal basis of $\C ^d$, an orthonormal basis of $\C^{n\times n}\otimes (\C ^d \otimes \C ^{n\times n} ) ^{\otimes ^\ell}$, where $\ell \in \N$, is then given by
\begin{align*}
    \{ \left. E_{i_0,j_0}\otimes e_{k_1}\otimes E_{i_1,j_1}\otimes \cdots\otimes e_{k_\ell}\otimes E_{i_\ell, j_\ell} \right| \ 1\leq i_0,\cdots,i_\ell,j_0,\cdots,j_\ell\leq n,\; 1\leq k_1,\cdots,k_\ell\leq d\},
\end{align*}
where the matrices $E_{i,j}=e_i e_j $ are the standard matrix units of $\C ^{n\times n}$. For convenience and brevity, we introduce the short-form notation: For $\alpha = a_0 \cdots a_{\ell} \in \mathbb{F} _n ^+$, $\beta = b_0 \cdots b_{\ell} \in \mathbb{F} _n ^+$ and $\om = w_1 \cdots w_\ell\in \F$,
$$ E_{\alpha, \beta} \star e_\om := E_{a_0, b_0} \otimes e_{w_1} \otimes E_{a _1, b_1} \otimes \cdots \otimes e_{w_\ell} \otimes E_{a_\ell, b_\ell}. $$ Hence, an orthornormal basis of $\scr{F} _n (\C ^d)$ is then 
$$\{ E_{\alpha, \beta} \star e_\om | \alpha, \beta \in \Fn, \ \om \in \F, \ |\alpha | = |\beta | = |\om | +1 \}.$$ In particular, the orthonormal basis of the ``vacuum space", $\C ^{n \times n} \oplus \{ 0 \} \subseteq \scr{F} _{n} (\C ^d)$ is then written
$$ \{ E_{i,j} \star e_\emptyset | \ 1\leq i,j \leq n \}. $$

For $1\leq i,j\leq n$ and $1\leq k\leq d$, define the left creation operators on $\scr{F} _n (\C ^d)$ by
\begin{align*}
    L_{i,j;k}:\scr{F}_n(\C ^d) \rightarrow \scr{F}_n(\C ^d)\\
    E_{\alpha, \beta} \star e_\om \mapsto E_{i\alpha, j\beta} \star e_{k\om}.
\end{align*}
Note that $L_{i,j;k}$ defines a $dn^2-$tuple of linear isometries on $\scr{F}_n(\C ^d)$ with pairwise orthogonal ranges. Given any $\alpha  = a_0 \cdots a_\ell$, $\beta = b_0 \cdots b_\ell$ and $\om = w_1 \cdots w_\ell$, $\ell \in \N$, so that $\alpha, \beta \in \Fn$, $\om \in \F$ and $|\alpha | = |\beta |  =|\om | +1 = \ell +1$, consider $\alpha ' = a_0 \cdots a_{\ell -1}$ and $\beta ' = b_0 \cdots b_{\ell -1}$. Then, if we define
$$ L_{\alpha ' , \beta '; \om} :=  L_{a_0, b_0; w_1} \cdots L_{a_{\ell-1}, b_{\ell-1}; w_\ell}, $$ notice that 
$$ E_{\alpha, \beta} \star e_\om = L_{\alpha' , \beta' ; \om} E_{a_\ell, b_\ell} \star e_\emptyset.$$

Analogously, we define the right creation operators on $\scr{F} _n (\C ^d)$: For $1\leq a,b\leq n$, $1 \leq w \leq d$, $\alpha, \beta \in \mathbb{F} _n ^+$ and $\om \in \F$ obeying $|\alpha | = | \beta |  = | \om | +1$,
$$ R_{a,b;w} E_{\alpha, \beta} \star e_\om := E _{\alpha a, \beta b} \star e_{\om w}, $$
which also yields a family of linear isometries on $\scr{F}_n(\C ^d)$ with pairwise orthogonal ranges. Similarly to before, given $\alpha = a_0 \cdots a_\ell$, $\beta = b_0 \cdots b_\ell$ $\alpha, \beta \in \Fn$, $\om = w_1 \cdots w_\ell \in \F$, let $\alpha '' := a_1 \cdots a_{\ell}$, $\beta '' = b_1 \cdots b_{\ell}$. Then, we have that 
\be E_{\alpha, \beta} \star e_{\om} = R_{\alpha ^{'' ; \mrt}, \beta ^{'' ; \mrt}; \om ^\mrt} E_{a_0, b_0} \star e_\emptyset. \label{RONB} \ee  
It follows that if we define the ``flip unitary", $U_\mrt$, on $\scr{F} _n (\C ^d)$ by 
$$ E_{\alpha,\beta} \star e_\om  \, \stackrel{U_\mrt}{\mapsto} \, E_{\alpha ^\mrt, \beta ^\mrt} \star e_{\om ^\mrt}, $$ then $U_\mrt$ is a self-adjoint unitary (hence a unitary involution) obeying
\ba U_\mrt L_{a,b;w} E_{\alpha, \beta } \star e_\om & = & E_{\alpha ^\mrt a, \beta ^\mrt b} \star e_{\om ^\mrt w} \\
& = & R_{a,b;w} U_\mrt E_{\alpha, \beta} \star e_\om, \ea so that $U_\mrt L_{a,b;w} U_\mrt = R_{a,b;w}$ and the left and right creation operators are unitarily equivalent via $U_\mrt$.

Given $\zeta, \xi \in \Fn$ and $\Xi \in \F$ with $|\zeta| =|\xi | = |\Xi | =\ell$, $\zeta = z_1 \cdots z_\ell$, $\xi = y_1 \cdots y_\ell$ and $\Xi = x_1 \cdots x_\ell$, we define 
$$ R_{\zeta, \xi ; \Xi} := R_{z_1, y_1; x_1 } \cdots R_{z_\ell, y_\ell; x_\ell}. $$ 
Let $h \in \scr{F} _n (\C ^d)$ with 
$$ h = \sum _{\substack{\alpha, \beta \in \Fn, \ \om \in \F \\ |\alpha | = | \beta | = | \om | +1}} \hat{h} _{\alpha, \beta ; \om} \, E_{\alpha,\beta} \star e_\om, $$ where $\alpha = a_0 a_1 \cdots a_\ell$, $\beta = b_0 b_1 \cdots b_\ell$ and $\om = w_1 \cdots w_\ell$. By taking inner products against elements of the orthonormal basis, one obtains
\ba (R_{\zeta ^\mrt , \xi ^\mrt ; \Xi ^\mrt}) ^* h & = & R_{z_1, y_1 ; x_1} ^* \cdots R_{z_\ell, y_\ell; x_\ell} ^* h =: (R^*)_{\zeta, \xi; \Xi} h \\
& = & \sum _{\substack{\alpha, \beta \in \Fn, \ \om \in \F \\ |\alpha | = | \beta | = | \om | +1}} \hat{h} _{\alpha \zeta, \beta \xi ; \om \Xi} \, E_{\alpha,\beta} \star e_\om. \ea
Alternatively,
\ba (R_{\zeta ^\mrt , \xi ^\mrt ; \Xi ^\mrt} )^* h & = & \sum _{\substack{\alpha, \beta \in \Fn, \ \om \in \F \\ |\alpha | = | \beta | = | \om | +1}} \hat{h} _{\alpha , \beta ; \om } \, (R^*)_{\zeta  , \xi  ; \Xi }  E_{\alpha,\beta} \star e_\om \\
& = & \sum _{\substack{\alpha, \beta \in \Fn, \ \om \in \F \\ |\alpha | = | \beta | = | \om | \\ 1 \leq a,b \leq n}} \hat{h} _{a \alpha  , b\beta  ; \om } \, (R^*)_{\zeta  , \xi  ; \Xi }  R _{\alpha ^\mrt, \beta ^\mrt; \om ^\mrt} E_{a,b} \star e_\emptyset \\
& = & \sum  _{\substack{\alpha ', \beta' \in \Fn, \ \om' \in \F \\ |\alpha' | = | \beta' | = | \om' | +1}} \hat{h} _{\alpha' \zeta  , \beta ' \xi; \om ' \Xi} E_{\alpha ', \beta '} \star e_{\om '}, \ea where the final line follows as 
$$ (R^*)_{\zeta , \xi ; \Xi } R _{\alpha ^\mrt, \beta ^\mrt; \om ^\mrt} = \left\{ \begin{array}{cc} R _{\ga ^\mrt, \sigma ^\mrt; \rho ^\mrt} & \mbox{if} \ \alpha ^\mrt = \zeta ^\mrt \ga ^\mrt, \beta ^\mrt = \xi ^\mrt  \sigma ^\mrt, \ \mbox{and} \ \om ^\mrt = \Xi ^\mrt \rho ^\mrt, \ \ga, \sigma \in \Fn, \ \rho \in \F, |\ga | = |\sigma | = |\rho |, \\ 
 (R^*) _{\ga, \sigma ; \rho } & \mbox{if} \ \zeta =  \ga \alpha, \xi = \sigma \beta, \ \mbox{and} \ \Xi = \rho \om, \ \ga, \sigma \in \Fn, \ \rho \in \F, |\ga | = |\sigma | = |\rho |, \\
0 & \mbox{else,} \end{array} \right. $$ and $R^* _{a,b;w} E_{a',b'} \star e_{\emptyset} =0$ for any $1 \leq a,a',b,b' \leq n$ and any $1 \leq w \leq d$.

\subsection{Evaluations}

It will be convenient to view elements of $\scr{F}_n(\C ^d)$ as functions on the unit NC uniform row-ball centred at $0 ^{(n)} := (0_n, \cdots, 0_n) \in \cdn$, $\B ^d _{\N n} (0 ^{(n)})$.  Namely, for any $X \in \C ^{(mn\times mn)\cdot d}$, define the linear \emph{formal evaluation} map at $X$ on the standard basis elements $E_{\alpha , \beta} \star e_\om$ by 
$$ E_{\alpha, \beta} \star e_\om \ \mapsto \ I_m \otimes E_{a_0, b_0} X_{w_1} I_m \otimes E_{a_1, b_1} X_{w_2} \cdots I_m \otimes E_{a_{\ell -1}, b_{\ell -1}} X_{w_\ell} I_m \otimes E_{a_{\ell}, b_\ell} =: E_{\alpha, \beta} \star e_\om (X). $$ The proposition below will show that this notion of ``formal evaluation" extends to a continuous linear map on $\scr{F} _n (\C ^d)$ for any $X$ in an NC uniform row-ball centred at $0_n$ of radius $\frac{1}{\sqrt{n}}$.

\begin{prop}
Given $(X,y,v) \in \C ^{(mn \times mn) \cdot d} \times \C ^{mn} \times \C ^{mn}$, the linear functional $\ell _{X,y,v}$ defined by $E_{\alpha, \beta} \star e_\om \mapsto y^* E_{\alpha, \beta} \star e_\om (X) v$ extends to a bounded linear map on $\scr{F} _n (\C ^d)$ if $\| X \| _{\mr{col}}  < \frac{1}{\sqrt{n}}$. 
\end{prop}
\begin{proof}
Consider the formal series,
$$ K^{(n)} \{ X , y ,v \} := \sum _{\substack{\om \in \F; \ \alpha, \beta \in \mathbb{F} ^+ _n \\  |\alpha| = | \beta | = |\om| +1}} \ov{y^* E_{\alpha ,\beta} \star e_\om (X) v} \, E_{\alpha,\beta} \star e_\om. $$
To prove that $\ell _{X,y,v}$ is bounded, it suffices to show that $K^{(n)} \{ X , y , v \} \in \scr{F} _n (\C ^d)$, \emph{i.e.} that it has square--summable coefficients. 

It suffices to consider $y:= x \otimes e_i \in \C ^{m} \otimes \C ^n$ and $v:= u \otimes e_j \in \C ^m \otimes \C ^n$:
\ba & &  \sum _{|\om| =\ell} \sum _{|\alpha| = \ell +1 = |\beta|} \left| y^* E_{\alpha ,\beta} \star e_\om (X) v \right| ^2   \\
& = & \sum _{1 \leq k_1, \cdots, k_\ell \leq d} 
\sum _{\substack{1\leq i_0, \cdots, i_\ell \leq n \\ 1 \leq j_0, \cdots, j_{\ell} \leq n} } u^* \otimes e_j ^* I_m \otimes E_{j_\ell, i_\ell} X_{k_\ell} ^* I_m \otimes E_{j_{\ell-1}, i_{\ell -1}} X_{k_{\ell-1}} ^* \cdots I_m \otimes E_{j_1, i_1} X_{k_1} ^* I_m \otimes E_{j_0, i_0} xx^* \otimes E_{i,i} \\
& & I_m \otimes E_{i_0, j_0} X_{k_1} I_m \otimes E_{i_1, j_1} X_{k_2} \cdots  X_{k_\ell} I_m \otimes E_{i_\ell, j_\ell} u \otimes e_j  \\
& \leq & \| x \| ^2 \sum _{k_1, \cdots, k_\ell} 
\sum _{\substack{i_1, \cdots, i_\ell \\ j_0, \cdots, j_{\ell -1}} }
u^* \otimes e_{i_\ell } ^* X_{k_\ell} ^* \cdots I_m \otimes E_{j_1, i_1} X_{k_1} ^* I_m \otimes E_{j_0,j_0} X_{k_1} I_m \otimes E_{i_1, j_1} \cdots X_{k_\ell} u \otimes e_{i_\ell} \\
& \leq &  \| X \| _{\mr{col}} ^2 \| x \| ^2 \sum _{k_2, \cdots, k_{\ell}} 
\sum _{\substack{i_1, \cdots, i_\ell \\ j_1, \cdots, j_{\ell -1}} } u^* \otimes e_{i_\ell } ^* X_{k_\ell} ^* \cdots X_{k_2} ^* I_m \otimes E_{j_1, j_1} X_{k_2} \cdots X_{k_\ell} u \otimes e_{i_\ell}. \ea 
Summing over the indices $i_1$ and $j_1$ then yields,
\ba & = & n \| X \| _{\mr{col}} ^2 \| x \| ^2 \sum _{k_2, \cdots, k_{\ell} }
\sum _{\substack{i_2, \cdots, i_\ell \\ j_2, \cdots, j_{\ell -1}} }  u^* \otimes e_{i_\ell } ^* X_{k_\ell} ^* \cdots I_m \otimes E_{j_2, i_2} X_{k_2} ^*  X_{k_2} I_m \otimes E_{i_2,j_2} \cdots X_{k_\ell} u \otimes e_{i_\ell} \\
\cdots & \leq & n^{\ell} \| X \| _{\mr{col}} ^{2\ell} \| x \| ^2 \| u \| ^2. \ea

In particular, it follows that if $\| X \| ^2 _{\mr{col}}  < \frac{1}{n}$, then $K^{(n)} \{ X , x \otimes e_i, u \otimes e_j \} \in \scr{F} _n (\C ^d)$. It is then clear that inner products against this vector give the linear functional $\ell _{X,y,v}$, so that this linear functional is bounded on $\scr{F} _n (\C ^d)$.
\end{proof}

\begin{cor}
Any $f \in \scr{F} _n (\C ^d)$ defines a graded and direct sum-preserving function in $\frac{1}{\sqrt{n}} \cdot \B ^d _{\N n} (0 ^{(n)})$ via formal evaluation.
\end{cor}

\begin{remark}
If we equip $\C^{n\times n}$ with the normalized Hilbert--Schmidt inner product instead, then formal point evaluation is a bounded linear homomorphism on $\scr{F} _n (\C ^d)$ for any $X \in \mathbb{B} ^d _{\N n} (0 ^{(n)})$.
\end{remark}

It will be useful, in the sequel, to consider evaluations of $h \in \scr{F} _n (\C ^d)$ in the row-ball of radius $\frac{1}{\sqrt{n}}$ centred at $Y \in \cdn$. 

\begin{defn} \label{eval}
Given any $X \in \frac{1}{\sqrt{n}} \cdot \B ^d _{mn} (0 ^{(n)})$, we define the \emph{formal evaluation} of 
$$ h = \sum _{\substack{\alpha, \beta \in \Fn; \ \om \in \F \\ |\alpha | = |\beta | = |\om | +1}} \hat{h} _{\alpha, \beta; \om} E_{\alpha, \beta} \star e_\om \in \scr{F} _n (\C ^d), $$ as 
$$ \mr{ev} _X (h) := \sum _{\substack{\alpha, \beta \in \Fn; \ \om \in \F \\ |\alpha | = |\beta | = |\om | +1}} \hat{h} _{\alpha, \beta; \om} E_{\alpha, \beta} \star e_\om (X) \in \C ^{mn \times mn}. $$ 

Given any $X \in \frac{1}{\sqrt{n}} \cdot \B ^d _{m n} (Y)$, $m \in \N$, we define the \emph{evaluation at $X$ about $Y$} of $h \in \scr{F} _n (\C ^d)$,
 as 
\be h(X) := \mr{ev}_{X-I_m \otimes Y} (h) =  \sum _{\substack{\alpha, \beta \in \Fn; \ \om \in \F \\ |\alpha | = |\beta | = |\om | +1}} \hat{h} _{\alpha, \beta; \om} E_{\alpha, \beta} \star e_\om (X - I_m \otimes Y). \label{evalY} \ee
\end{defn}

In what follows, we will be primarily interested in evaluation of $h$ at $X \in \B ^d _{\N n} (Y)$ ``about $Y$", and we will henceforth refer to $h(X)$, defined above in Equation (\ref{evalY}), as the \emph{evaluation of $h$ at $X$}.  

\subsection{Matrix-centre realizations for elements of $\scr{F} _n (\C ^d)$}\label{sec:realforFockspace}

We can now construct a matrix centre realization of $h$ at $Y \in \cdn$ in the following way. Firstly, for $Z\in \C^{n\times n}$ and $1\leq k\leq d$, define
$$ A_k (Z) := \sum _{i,j =1} ^n Z E_{i,j} \otimes R_{i,j;k} ^* \in \C ^{n\times n} \otimes \scr{B} (\scr{F} _n (\C ^d)) \subseteq \scr{B} ( \C ^n \otimes \scr{F} _n (\C ^d) ), $$ 
so that $A_k : \C ^{n\times n} \rightarrow \scr{B} (\cH)$ is linear with $\cH := \C ^n \otimes \scr{F} _n (\C ^d)$. Next, define $c : \C ^n \rightarrow \C ^n \otimes \scr{F} _n (\C ^d) = \cH$ by  
$$  c v := v\otimes h, $$ and $b_* : \cH \rightarrow \C ^n$ by 
$$ b_* u \otimes f := \sum_{i,j =1} ^n \ip{E_{i,j} \star e_\emptyset}{f}_{\scr{F}_n(\C ^d)} E_{i,j}u.   $$ We then define $b : \C ^n \rightarrow \cH$ by $b:= b_* ^*$ so that $b_* = b^*$.

We claim that with this choice of $(A,b,c)_Y$ that $h \sim _Y (A,b,c)$. Indeed, let $\omega= w_1\cdots w_\ell\in\F$ be a word of length $\ell$, $X\in\C^{(mn\times mn)\cdot d}$ such that $\|X - I_m \otimes Y \|_\mathrm{col}<\|A\|_\mathrm{CB;row}^{-1}$ and $v\in\C^{mn}$. We then have, by definition,
\ba 
& & A^\omega \circ (X - I_m \otimes Y)^{\odot^\omega}(I_m\otimes c)v=  (\mathrm{id}_m\otimes A_{w_1})(X_{w_1}-I_m\otimes Y_{w_1})\cdots (\mathrm{id}_m\otimes A_{w_\ell})(X_{w_\ell}-I_m\otimes Y_{w_\ell})(v\otimes h)  \\ 
& = & \sum _{\substack{a_1, \cdots, a_\ell \\ b_1, \cdots, b_\ell} =1} ^n (X_{w_1} - I_m \otimes Y_{w_1} ) I_m \otimes E_{a_1,b_1}  \cdots (X_{w_\ell} - I_m \otimes Y_{w_\ell}) I_m \otimes E_{a_\ell, b_\ell} v \otimes R^* _{a_1, b_1; w_1} \cdots R^* _{a_\ell, b_\ell; w_\ell} h \\
& = & \sum _{\substack{a_1, \cdots, a_\ell \\ b_1, \cdots, b_\ell} =1} ^n (X_{w_1} - I_m \otimes Y_{w_1} ) I_m \otimes E_{a_1,b_1} (X_{w_2} - I_m \otimes Y_{w_2} )  \cdots (X_{w_\ell} - I_m \otimes Y_{w_\ell}) I_m \otimes E_{a_\ell, b_\ell} v \otimes (R^*) _{a_1\cdots b_\ell, a_1 \cdots b_\ell; \om} h. \ea
Therefore, setting $\alpha ' = a_1 \cdots a_\ell \in \Fn$ and $\beta ' = b_1 \cdots b_\ell \in \Fn$, $\alpha = a_0 \alpha'$ and $\beta = b_0 \beta'$, \footnotesize
\ba   & &  I_m \otimes b^* A^\om \circ(X- I_m \otimes Y)^{\odot^\om}(I_m \otimes c )v  \\
& = &  \sum _{\substack{\alpha', \beta' \in \Fn \\ |\alpha '| = |\beta ' | = \ell}} \sum _{a_0, b_0 =1} ^n \ip{ E_{a_0, b_0} \star e_\emptyset}{(R^*) _{\alpha', \beta' ; \om} h}_{\scr{F} _n (\C ^d)} I_m \otimes E_{a_0, b_0}  (X_{w_1} - I_m \otimes Y_{w_1} ) I_m \otimes E_{a_1,b_1}   \cdots (X_{w_\ell} - I_m \otimes Y_{w_\ell}) I_m \otimes E_{a_\ell, b_\ell} v \\
& = & \sum _{\substack{\alpha, \beta \in \Fn \\ |\alpha| = |\beta | = \ell +1}} \ip{ R_{\alpha ^{';\mrt}, \beta ^{';\mrt}; \om ^\mrt} E_{a_0, b_0} \star e_\emptyset}{h}_{\scr{F} _n (\C ^d)} E _{\alpha, \beta} \star e_\om (X-I_m \otimes Y)v\\
& = & \sum _{\substack{\alpha, \beta \in \Fn \\ |\alpha| = |\beta | = \ell +1}} \ip{ E_{\alpha,\beta} \star e_\om}{h}_{\scr{F} _n (\C ^d)} E _{\alpha, \beta} \star e_\om (X-I_m \otimes Y)v,  \quad \quad \quad \quad \mbox{by Equation (\ref{RONB}),}  \\
& = & \sum _{\substack{\alpha, \beta \in \Fn \\ |\alpha| = |\beta | = \ell +1}} \hat{h}_{\alpha, \beta ; \om} \, E _{\alpha, \beta} \star e_\om (X-I_m \otimes Y)v. \ea 
\normalsize

Hence,
\ba \sum_{\omega\in \F} I_m\otimes b^* \, A^\omega\circ(X-I_m \otimes Y)^{\odot^\om}\, I_m\otimes c  & = & \sum _{\om \in \F} \sum _{\substack{\alpha, \beta \in \Fn \\ |\alpha | = |\beta | = |\om | +1}} \hat{h}_{\alpha, \beta ; \om} E _{\alpha, \beta} \star e_\om (X-I_m \otimes Y) \\
& = & h (X), \ea 
is the evaluation of $h$ at $X \in\frac{1}{\sqrt{n}} \cdot \mathbb{B} ^d _{mn} (Y)$, as defined in Definition \ref{eval}, and this coincides with the definition of evaluation of $h$ at $X$, viewed as the quantized transfer function of $(A,b,c)_Y$, $h \sim (A,b,c)_Y$. 

\subsection{Matrix-centre realizations for uniformly analytic NC functions}

\label{formalreal}

Let $f$ be a uniformly analytic NC function in a uniformly open neighborhood of the matrix point $Y\in \cdn$. In this section we will construct a matrix-centre realization at $Y$ representing $f$ in a uniformly open neighborhood of $Y$. Recall from Subsection \ref{realTT},   for each $\ell\in\N$ there exists a completely bounded $\ell-$linear map $f_\ell:\C^{(n\times n)\cdot d}\times \cdots\times \C^{(n\times n)\cdot d}\to \C^{n\times n}$ such that
\begin{align*}
    f(X)=\sum_{\ell=0}^\infty f_\ell \circ \left(X-I_m \otimes Y \right)^{\odot ^\ell}; \quad \quad X \in \C ^{(mn \times mn)\cdot d},
\end{align*}
where the series converges absolutely and uniformly in any NC column ball $r\cdot \B^d _{\N n}(Y)$ of radius $r<R_f ^Y$, where $R_f^Y$ is the radius of convergence of $f$,
\begin{align*}
    \frac{1}{R _f ^Y}=\limsup_{\ell\to\infty}\sqrt[\ell]{\|f_\ell\|_\mathrm{CB}},
\end{align*}
and $\|f_\ell\|_\mathrm{CB}$ denotes the completely bounded norm of $f_\ell$. See Subsection \ref{realTT} and \cite[Theorem 7.2, Theorem 8.11]{KVV} for more details.

To be more precise, the maps $f_\ell$ are given by the NC difference--differential operators of $f$. As discussed in Subsection \ref{realTT}, this can be further expanded as a series indexed by the free monoid:
\begin{align*}
    f(X)=\sum_{\omega\in \F} f_\omega \circ(X-I_m\otimes Y)^{\odot^\omega} =\sum_{\ell=0}^\infty \sum_{|\omega|=\ell}  f_\omega \circ(X-I_m\otimes Y)^{\odot^\omega},
\end{align*}
where each $f_\omega:\C^{n\times n}\times\cdots\times\C^{n\times n}\to\C^{n\times n}$ is an $\ell=|\omega|$-linear map defined by the partial difference--differential operators of $f$.
Moreover, following \cite[Arguments following Remark 4.5]{KVV}, for every $\omega\in \F$, $|\omega|=\ell$, there exist $N_\omega\in \N$ and matrices $A_{\omega, a}^{(b)}\in \C^{n\times n}$, $0\leq a\leq \ell$, $1\leq b\leq N_\omega$, such that
\begin{align*}
    f_\omega \circ(X-I_m\otimes Y)^{\odot^\omega} = \sum_{b=1}^{N_\omega}
    I_m\otimes A_{\omega,0}^{(b)}(X_{\omega_1}-I_m\otimes Y_{\omega_1})I_m\otimes A_{\omega,1}^{(b)}\cdots (X_{\omega_\ell}-I_m\otimes Y_{\omega_\ell})I_m\otimes A_{\omega,\ell}^{(b)}.
\end{align*}
We expand this even further by decomposing the $A^{(b)}_{\om, a}$ as linear combinations of the standard matrix units. If $\om = w_1 \cdots w_\ell \in \F$ and we write $\alpha =a_0 \cdots a_\ell$, $\beta = b_0 \cdots b_\ell$ for $\alpha, \beta \in \mathbb{F} ^+ _n$ of length $\ell +1$, we obtain,
\begin{align*} 
    f_\omega \circ(X-I_m\otimes Y)^{\odot^\omega} 
&    =  \sum_{\substack{\alpha, \beta \in \mathbb{F} _n ^+ \\ |\alpha | = |\beta| = \ell +1}} \hat{f} _{\alpha, \beta ; \om} \;I_m\otimes E_{a_0, b_0}(X_{w_1}-I_m\otimes Y_{w_1})I_m\otimes E_{a_1,b_1}\cdots (X_{w_\ell}-I_m\otimes Y_{w_\ell})I_m\otimes E_{a_\ell ,b_\ell} \\
&      =   \sum _{\alpha, \beta} \hat{f} _{\alpha, \beta; \om} \,  E_{\alpha,\beta} \star e_\om (X -I_m \otimes Y),  \end{align*}
for appropriate coefficients $\hat{f} _{\alpha, \beta ; \om} \in\C$.

In particular, the $\ell = |\om |-$linear map $f_\omega$ is given by
\begin{align*}
    f_\omega: \C^{n\times n}\times \cdots\times \C^{n\times n}&\to \C^{n\times n},\\
    (G_1,\cdots,G_\ell)\mapsto \sum_{\substack{\alpha, \beta \in \mathbb{F} ^+ _n \\ |\alpha| = | \beta| = \ell +1}} \hat{f} _{\alpha, \beta; \om} & \; E_{a_0,b_0}G_{1}E_{a_1,b_1}\cdots G_{\ell}E_{a_\ell, b_\ell},
\end{align*}
where $\alpha = a_0 \cdots a_\ell$ and $\beta = b_0 \cdots b_\ell$, $\alpha, \beta \in \Fn$.

In conclusion, for $X\in \C^{(mn\times mn)\cdot d}$ sufficiently close to $Y$, $f(X)$ is given by
$$ f(X) = \sum _{\om\in\F} \sum _{\substack{\alpha, \beta \in \mathbb{F} ^+ _n, \\ |\alpha | = |\beta | = |\om | +1}} \hat{f} _{\alpha, \beta; \om} \, E_{\alpha,\beta} \star e_\om (X - I_m \otimes Y). $$ 

Alternatively, we can identify $f$ with the formal series
\be  f =  \sum _{\substack{\alpha, \beta \in \mathbb{F} ^+ _n, \ \om \in \F \\ |\alpha | = |\beta | = |\om | +1}} \hat{f} _{\alpha, \beta; \om} \, E_{\alpha,\beta} \star e_\om. \label{fformalseries} \ee 

If the coefficients of this series are square summable, then $f$ is an element of the matrix valued Fock space $\scr{F}_n(\C ^d)$, which was introduced in Definition \ref{defn:matrixFockspace}. In this case, $f$ would have a matrix-centre realization at $Y$, as shown in the previous Subsection \ref{sec:realforFockspace}. However, we will show that $f$ has a matrix-centre realization at $Y$ regardless. 

To see this, first consider $V:= \C ^{dn^2}$ with orthonormal basis $\{v_{i,j,k} | \  1\leq i,j\leq n,\; 1\leq k\leq d\}$. We can then define a canonical unitary $U:\scr{F}_n (\C ^d)\rightarrow \scr{F}(\C ^{dn^2})\otimes\C^{n\times n}$ as follows: If $\alpha, \beta \in \Fn$, $\om \in \F$ with $|\alpha| = | \beta | = | \om | $, $\om = w_1 \cdots w_\ell$, $\alpha =  a_0 \cdots a_{\ell -1}$ and $\beta = b_0 \cdots b_{\ell-1}$ and $1 \leq a_\ell, b_\ell \leq n$, then 

\begin{eqnarray} \label{defn:Fockspaceunitary}
U E_{\alpha a_\ell, \beta b_\ell} \star e_\om & = & U \left( E_{a_0, b_0} \otimes e_{w_1} \otimes E_{a_1,b_1} \otimes e_{w_2} \cdots E_{a_{\ell -1} , b_{\ell -1}} \otimes e_{w_\ell} \right) \otimes E_{a_\ell, b_\ell} \nn  \\
& := &  \left( v_{a_0,b_0,w_1 } \otimes v_{a_1,b_1,w_2} \otimes \cdots \otimes v_{a_{\ell -1}, b_{\ell -1}, w_\ell} \right) \otimes E_{a_\ell, b_\ell} \nn \\
    &=:& v _{\alpha, \beta; \om} \otimes E _{a_\ell, b_\ell} \in \scr{F} (\C ^{dn^2} ) \otimes \C ^{n \times n}. \end{eqnarray} In the above we introduced the notation: For any $\alpha=a_0 \cdots a _{\ell -1}, \, \beta = b_0 \cdots b_{\ell -1} \in \Fn$ and $\om = w_1 \cdots w_\ell \in \F$, all of length $\ell$,
$$ v _{\alpha, \beta ; \om} := v_{a_0,b_0,w_1 } \otimes v_{a_1,b_1,w_2} \otimes \cdots \otimes v_{a_{\ell -1}, b_{\ell -1}, w_\ell}. $$ Hence, 
$$ \{ v_{\alpha, \beta; \om} | \ \alpha, \beta \in \Fn, \ \om \in \F, \ |\alpha | = |\beta | = | \om | \}, $$ is the standard orthonormal basis of $\scr{F} (\C ^{dn^2} )$. For $1\leq i,j \leq n$ and $1\leq k \leq d$, define the left creation operators on $\scr{F} (\C ^{dn^2})$ by: For any $\alpha, \beta \in \mathbb{F} _n ^+$ and $\om \in \F$ obeying $|\alpha | = |\beta | = |\om |$,
$$ S_{i,j,k} v_{\alpha, \beta ; \om} := v_{i\alpha,j\beta; k \om}. $$
Clearly, $U L_{i,j;k} U^*=S_{i,j;k}\otimes I_n$, so that the weak operator topology (WOT)--closed unital algebra generated by $\{  L_{i,j;k} | \ 1\leq i,j \leq n,\;1\leq k\leq d\}$ is unitarily equivalent to the WOT--closed unital algebra generated by $\{S_{i,j;k}\otimes I_n | 1\leq i,j \leq n,\;1\leq k\leq d\}$, which is $\mathbb{H} ^\infty_{dn^2}\otimes \C I_n$.

Now apply the unitary $U:\scr{F}_n(\C ^d)\to \scr{F}(\C^{d n^2})\otimes \C^{n\times n}$ (see (\ref{defn:Fockspaceunitary})), to (\ref{fformalseries}) in a formal way to obtain
\be \wt{f} :=  \sum _{a,b =1} ^n \underbrace{\sum _{\substack{\alpha, \beta \in \Fn; \ \om \in \F \\ |\alpha | = |\beta | = |\om|}} \hat{f} _{\alpha a ,\beta b;\om} v _{\alpha,\beta ; \om}}_{=: f _{a,b}} \otimes E_{a,b}. \label{formalseries1} \ee 

Each $f_{a,b}$, $1\leq a,b \leq n$, can be viewed as a formal power series in $dn^2$ variables. Here, evaluation of $f_{a,b}$ at an element $X=(X_{i,j,k})_{1\leq i,j\leq n,\;1\leq k\leq d}\in\C^{(\N\times\N)\cdot dn^2}$ is given by
\begin{eqnarray} f_{a,b} (X) & = & \sum _{\ell=0} ^\infty \sum _{w_1, \cdots, w_\ell =1} ^d \sum _{\substack{a_0, \cdots, a_{\ell-1} \\ b_0, \cdots, b_{\ell -1} } = 1} ^n \hat{f} _{a_0 \cdots a_{\ell-1} a, b_0 \cdots b_{\ell-1} b; w_1 \cdots w_\ell} \; X_{a_0,b_0, w_1} X_{a_1, b_1, w_2} \cdots X_{a_{\ell -1}, b_{\ell -1}, w_\ell} \nn \\
& =: & \sum _{\substack{\alpha, \beta \in \Fn; \ \om \in \F \\ |\alpha | = |\beta | = |\om | }} \hat{f} _{\alpha a, \beta b ; \om} \; X ^{\alpha, \beta , \om} = \sum_{\ell=0}^\infty  \sum _{\substack{\alpha, \beta \in \Fn; \ \om \in \F \\ |\alpha | = |\beta | = |\om |=\ell }} \hat{f} _{\alpha a, \beta b ; \om} \; X ^{\alpha, \beta , \om}. \label{evalform} \end{eqnarray}
We will often make use of this identification and treat corresponding elements as one and the same. However, it is not immediately obvious whether all $f_{a,b}$ have a non-zero radius of convergence. Our next goal is to show that this is indeed the case. But first we require the following two technical lemmas.

\begin{lemma}\label{lem:haagerupinequality}
    Consider the tensor product $\C^{(n\times n) \otimes ^\ell}$, $\ell\in\N$. Then, $\|A\|_{\scr{B}(\C^{n \otimes^\ell})}\leq \|A\|_h$ for all $A\in \C^{(n\times n)\otimes ^\ell}$.
\end{lemma}

\begin{proof}
    Let $A\in \C^{(n\times n)\otimes ^\ell}$ and let $A^{(s)}\in(\C^{n\times n})^{m_{s-1}\times m_s}$, $1\leq s\leq \ell$ with $m_0=m_\ell=1$ such that
    \begin{align*}
        A=A^{(1)}\odot \cdots \odot A^{(\ell)}.
    \end{align*}
    We then have 
    \begin{align*}
        A=B^{(1)}\cdots B^{(\ell)},
    \end{align*}
    where $B^{(s)}\in (\C^{(n\times n)\otimes ^\ell})^{m_s\times m_{s-1}}$, $1\leq s\leq \ell$, and its entry at position $(i,j)$ is given by
    \begin{align*}
        B^{(s)}_{i,j}=I\otimes \cdots\otimes I\otimes A^{(s)}_{i,j}\otimes I\otimes \cdots\otimes I,
    \end{align*}
    where $I$ denotes the identity matrix of size $n\times n$ and $A^{(s)}_{i,j}$ appears at the $s$-th position. Note that $\|B^{(s)}\|=\|A^{(s)}\|$ for all $1\leq s\leq \ell$, where $\|\cdot\|$ denotes the respective operator norms. Therefore,
    \begin{align*}
        \|A\|_{\scr{B}(\C^{n\otimes^\ell})}\leq \|B^{(1)}\|\cdots \|B^{(\ell)}\|= \|A^{(1)}\|\cdots \|A^{(\ell)}\|.
    \end{align*}
    Finally, taking the infimum over all such sets of matrices $A^{(0)},\dots,A^{(\ell)}$ yields the desired inequality, by definition of the Haagerup tensor norm. 
\end{proof}

\begin{lemma}\label{lem:radiusinequality}
    Given $f_{a,b} \in \C \langle \! \langle \, \fz _1, \cdots, \fz _{dn^2} \rangle \! \rangle$ with formal power series coefficients $\hat{f} _{\alpha,\beta; \om}$ as defined above in Equation (\ref{formalseries1}), there exists a constant $C>0$ such that for all $\ell\in\N$ one has
    \begin{align*}
        \sqrt[2\ell]{\sum_{\substack{\alpha, \beta \in \Fn; \ \om \in \F \\ |\alpha | = | \beta | = \ell+1 = |\om | +1}} |\hat{f}_{\alpha, \beta ;\om}|^2}\leq C  \sqrt[\ell]{\|f_\ell\|_{\mathrm{CB}}}  .
    \end{align*}
\end{lemma}

\begin{proof}
    Fix $\ell\in\N$. The proof is broken down into several steps.

    \emph{Step 1}: Fix a word $\om=w_1\cdots w_\ell\in \F$ of length $\ell$. In this step, we show that 
    \begin{align*}
        \frac{1}{n^{\ell+1}}\sum_{\substack{\alpha, \beta \in \Fn \\ |\alpha | = | \beta | = \ell +1}} |\hat{f}_{\alpha, \beta ;\om}|^2 
        \leq \left\|\sum_{\substack{\alpha, \beta \in \Fn \\ |\alpha | = | \beta | = \ell +1}} \hat{f}_{\alpha, \beta ;\om} \; E_{\alpha,\beta}\right\|_{\scr{B}(\C^{n\otimes^{\ell+1}})}^2,
    \end{align*}
    where $\C ^{n \otimes ^{\ell +1}} := \C ^n \otimes \C ^n \otimes \cdots \otimes \C ^n$ denotes the $\ell+1-$fold tensor product of $\C ^n$ with itself. 
    Here, if $\alpha=a_0\cdots a_\ell\in \Fn$ and $\beta=b_0\cdots b_\ell\in \Fn$, we denote
    $$E_{\alpha,\beta}:=E_{a_0,b_0}\otimes \dots\otimes E_{a_\ell,b_\ell}\in \C^{(n\times n)\otimes ^{\ell+1}}.$$ One can readily check that $E_{\alpha, \beta} ^* = E_{\beta, \alpha}$ and that $E_{\alpha, \beta } E_{\ga, \sigma} = E_{\alpha , \sigma} \delta _{\beta, \ga}$. 
    To prove the desired inequality, we start with the right hand side and first apply the $C^*-$identity:
    \begin{align*}
        \left\|\sum_{\substack{\alpha, \beta \in \Fn \\ |\alpha | = | \beta | = \ell +1}} \hat{f}_{\alpha, \beta ;\om} \; E_{\alpha,\beta}\right\|_{\scr{B}(\C^{n\otimes ^{\ell+1}})}^2&=\left\|\sum_{\substack{\alpha, \beta, \alpha', \beta' \in \Fn \\ |\alpha | = | \beta | = |\alpha'| = |\beta'| = \ell +1}} \hat{f}_{\alpha, \beta ;\om}\overline{\hat{f}_{\alpha', \beta' ;\om}} \; E_{\alpha,\beta}E_{\alpha',\beta'}^*\right\|_{\scr{B}(\C^{n\otimes ^{\ell+1}})}\\
       &= \left\|\sum_{\substack{\alpha, \beta, \alpha' \in \Fn \\ |\alpha | = | \beta | = |\alpha'| = \ell +1}} \hat{f}_{\alpha, \beta ;\om}\overline{\hat{f}_{\alpha', \beta ;\om}} \; E_{\alpha,\alpha'}\right\|_{\scr{B}(\C^{n\otimes ^{\ell+1}})}.
    \end{align*}
    Observe that $E_{\alpha,\alpha'}$ is an $n^{\ell+1}\times n^{\ell+1}$ matrix with one of its entries being $1$ and all other entries being $0$. This $1$ will be on the diagonal of the matrix if and only if $\alpha=\alpha'$. Moreover, note that for any matrix $Z$ of size $n^{\ell+1}\times n^{\ell+1}$ with entries in $\C$, we have $|\mathrm{tr}(Z)|\leq n^{\ell+1} \max_{i=1,\dots,n^{\ell+1}} |Z_{ii}|\leq n^{\ell+1} \|Z\|$ and hence $\|Z\|\geq \frac{1}{n^{\ell+1}}|\mathrm{tr}(Z)|$. Applying these two facts, we obtain
\ba \left\|\sum_{\substack{\alpha, \beta, \alpha' \in \Fn \\ |\alpha | = | \beta | = |\alpha'| = \ell +1}} \hat{f}_{\alpha, \beta ;\om}\overline{\hat{f}_{\alpha', \beta ;\om}} \; E_{\alpha,\alpha'}\right\|_{\scr{B}(\C^{n\otimes ^{\ell+1}})}
       & \geq &  \frac{1}{n^{\ell+1}}\mathrm{tr}\left(\sum_{\substack{\alpha, \beta, \alpha' \in \Fn \\ |\alpha | = | \beta | = |\alpha'| = \ell +1}} \hat{f}_{\alpha, \beta ;\om}\overline{\hat{f}_{\alpha', \beta ;\om}} \; E_{\alpha,\alpha'}\right) \\
        &=&\frac{1}{n^{\ell+1}}\sum_{\substack{\alpha, \beta \in \Fn \\ |\alpha | = | \beta | = \ell +1}} |\hat{f}_{\alpha, \beta ;\om}|^2. \ea
    
    \emph{Step 2}: By Lemma \ref{lem:haagerupinequality}, we have $\|A\|_{\scr{B}(\C^{n\otimes^{\ell+1}})}\leq \|A\|_h$ for all $A\in \C^{(n\times n)\otimes^{\ell+1}}$. Moreover, \cite[Theorem 8]{ChatterjeeSinclairIsometry} shows that the map
    \begin{align*}
        \underbrace{\C^{n\times n}\otimes_h\cdots\otimes_h\C^{n\times n}}_{(\ell+1)\times }&\to \mathrm{CB}(\underbrace{\C^{n\times n},\dots,\C^{n\times n}}_{\ell\times};\C^{n\times n})\\
        A_0\otimes \cdots\otimes A_{\ell}&\mapsto \Big((X_1,\dots,X_\ell)\mapsto A_0X_1A_1\cdots A_{\ell-1} X_\ell A_{\ell}\Big)
    \end{align*}
    is an isometry. 
    By combining these facts, we obtain,
    \begin{align*}
        &\left\|\sum_{\substack{\alpha, \beta \in \Fn \\ |\alpha | = | \beta | = \ell +1}} \hat{f}_{\alpha, \beta ;\om} \; E_{\alpha,\beta}\right\|_{\scr{B}(\C^{n\otimes^{\ell+1}})}\leq \left\|\sum_{\substack{\alpha, \beta \in \Fn \\ |\alpha | = | \beta | = \ell +1}} \hat{f}_{\alpha, \beta ;\om} \; E_{\alpha,\beta}\right\|_h \\
        =& \left\|(G_1,\dots,G_\ell)\mapsto \sum_{\substack{\alpha, \beta \in \mathbb{F} ^+ _n \\ |\alpha| = | \beta| = \ell +1}} \hat{f} _{\alpha, \beta; \om} \; E_{a_0,b_0}G_{1}E_{a_1,b_1}\cdots G_{\ell}E_{a_\ell, b_\ell}\right\|_{\mathrm{CB}}= \|f_\omega\|_{\mathrm{CB}}.
    \end{align*}

    \emph{Step 3}:  By \cite[Equation (8.32)]{KVV}, we have
    \begin{align*}
        \|f_\omega\|_{\mathrm{CB}}\leq \|f_\ell\|_{\mathrm{CB}},
    \end{align*}
    for all $\om \in \F$ with $|\om|=\ell$. Therefore
    \begin{align*}
        \sum_{\substack{\om \in \F \\ |\om|=\ell}} \|f_\omega\|_{\mathrm{CB}}^2 \leq \sum_{\substack{\om \in \F \\ |\om|=\ell}} \|f_\ell\|_{\mathrm{CB}}^2= d^\ell \|f_\ell\|_{\mathrm{CB}}^2.
    \end{align*}

    Finally, we can combine steps 1 through 3 to obtain the following.
\ba \frac{1}{n^{\ell+1}}\sum_{\substack{\alpha, \beta \in \Fn; \ \om \in \F \\ |\alpha | = | \beta | = \ell+1 = |\om | +1}} |\hat{f}_{\alpha, \beta ;\om}|^2 & = & \frac{1}{n^{\ell+1}}\sum_{\substack{\om \in \F \\ |\om|=\ell}}\sum_{\substack{\alpha, \beta \in \Fn \\ |\alpha | = | \beta | = \ell +1}} |\hat{f}_{\alpha, \beta ;\om}|^2  \\
       & \leq & \sum_{\substack{\om \in \F \\ |\om|=\ell}} \left\|\sum_{\substack{\alpha, \beta \in \Fn \\ |\alpha | = | \beta | = \ell +1}} \hat{f}_{\alpha, \beta ;\om} \; E_{\alpha,\beta}\right\|_{\scr{B}(\C^{n\otimes^{\ell+1}})}^2\\
         &\leq &  \sum_{\substack{\om \in \F \\ |\om|=\ell}} \|f_\om\|_\mathrm{CB}^2
         \leq d^{\ell}\|f_\ell\|_{\mathrm{CB}}^2. \ea
    Multiplying both sides of the inequality by $n^{\ell+1}$ and then taking the $2\ell$--th root of both sides completes the proof.
\end{proof}

Using the above lemma, we are now able to show that each of the power series $f_{a,b}$ has a positive radius of convergence.

\begin{lemma}\label{lem:radiuspositive}
    Let $f$ and $f_{a,b}$, $1\leq a,b\leq n$ be as in Equation (\ref{formalseries1}) above. Then the radius of convergence of each $f_{a,b}$ (at $0$) is positive.
\end{lemma}

\begin{proof}
    Since $f$ is assumed to be uniformly analytic in a uniformly open neighborhood of $Y$, as discussed in Subsection \ref{realTT}, the radius of convergence of its Taylor-Taylor series at $Y$ will be positive, i.e.
    \begin{align*}
        \frac{1}{R_f ^Y}=\limsup_{\ell\to\infty}  \sqrt[\ell]{\|f_\ell\|_\mathrm{CB}}<\infty.
    \end{align*}
    Fix $1\leq a,b\leq n$, Lemma \ref{lem:radiusinequality} then yields
\ba \frac{1}{R_{a,b}} & =&\limsup_{\ell\to\infty} \sqrt[2\ell]{\sum _{\substack{\alpha', \beta' \in \Fn; \ \om \in \F \\ |\alpha' | = |\beta' | = |\om |=\ell }} |\hat{f} _{\alpha' a, \beta' b ; \om}|^2} \\ & \leq & \limsup_{\ell\to\infty} \sqrt[2\ell]{\sum_{\substack{\alpha, \beta \in \Fn; \ \om \in \F \\ |\alpha | = | \beta | = \ell+1 = |\om | +1}} |\hat{f}_{\alpha, \beta ;\om}|^2}\\
        & \leq& C \limsup_{\ell\to \infty} \sqrt[\ell]{\|f_\ell\|_\mathrm{CB}}<\infty,
    \ea
    where $R_{a,b}$ denotes the radius of convergence of $f_{a,b}$ at $0 \in \C ^{dn^2} \subset \C ^{(\N \times \N) \cdot dn^2}$. This completes the proof.
\end{proof}

Finally, we can now show that $f$ has a matrix-centre realization at $Y$.

\begin{thm}
    Let $f$ be an NC function that is uniformly analytic in a uniformly open neighborhood of the matrix point $Y$. Then $f$ has a matrix-centre realization at $Y$.
\end{thm}

\begin{proof}
    Let $f_{a,b}$, $1\leq a,b\leq n$, be defined as above in Equation (\ref{formalseries1}). By Lemma \ref{lem:radiuspositive} we can find $r>0$ such that the dilation $f_{a,b} ^{(r)}(X)=f_{a,b}(rX)$ will have a radius of convergence strictly larger than $1$ for all $1\leq a,b\leq n$. In particular, $f_{a,b} ^{(r)}\in \mathbb{H}_{dn^2}^\infty$, and hence $f_{a,b}^{(r)}\in \mathbb{H}_{dn^2}^2$ by Proposition \ref{prop:MultiplierHinfty}, for all $1\leq a,b\leq n$. Therefore, the coefficients of $f_{a,b}^{(r)}$ are square-summable, i.e.
    \begin{align*}
        \sum _{\substack{\alpha, \beta \in \Fn; \ \om \in \F \\ |\alpha | = |\beta | = |\om | }} |r^{|\om|}  \hat{f} _{\alpha a, \beta b ; \om}|^2<\infty.
    \end{align*}
    This implies 
    \begin{align*}
        \sum _{\substack{\alpha, \beta \in \Fn; \ \om \in \F \\ |\alpha | = |\beta | = |\om|}} r^{|\om|}\hat{f} _{\alpha a ,\beta b;\om} \;v _{\alpha,\beta ; \om} \in\scr{F}(\C^{dn^2}),
    \end{align*}
    for all $1\leq a,b\leq n$. Hence,
    \begin{align*}
       \sum _{a,b =1} ^n \sum _{\substack{\alpha, \beta \in \Fn; \ \om \in \F \\ |\alpha | = |\beta | = |\om|}} r^{|\om|}\hat{f} _{\alpha a ,\beta b;\om} \;v _{\alpha,\beta ; \om} \otimes E_{a,b}\in\scr{F}(\C^{dn^2})\otimes \C^{n\times n}.
    \end{align*}
    The image of this element under the unitary $U^*:\scr{F}(\C^{dn^2})\otimes \C^{n\times n}\to \scr{F}_n(\C^d)$, see (\ref{defn:Fockspaceunitary}), is given by
    \begin{align*}
         f_r:=\sum _{\substack{\alpha, \beta \in \mathbb{F} ^+ _n, \ \om \in \F \\ |\alpha | = |\beta | = |\om | +1}} r^{|\om|}\hat{f} _{\alpha, \beta; \om} \, E_{\alpha,\beta} \star e_\om\in \scr{F}_n(\C^d).
    \end{align*}
    According to the arguments in Section \ref{sec:realforFockspace}, $f_r$ has a matrix-centre realization $(A',b,c)_Y$, where $A'=(A_1',\cdots,A'_d)$ with $A'_k:\C^{n\times n}\to \scr{B}(\mathcal{H})$, $1\leq k\leq d$, $\mathcal{H}=\C^n\otimes \scr{F}_n(\C^d)$, and 
    \begin{align*}
        A'_k (Z) = \sum _{i,j =1} ^n Z E_{i,j} \otimes R_{i,j;k} ^*.
    \end{align*}
    It is now readily checked that $(A,b,c)_Y$, where $A=(A_1,\cdots,A_d) = \frac{1}{r} A'$,
    \begin{align*}
        A_k (Z) = \frac{1}{r} \sum _{i,j =1} ^n Z E_{i,j} \otimes R_{i,j;k} ^*,
    \end{align*}
    will then be a matrix-centre realization for $f$ at $Y$.
\end{proof}

\begin{cor} \label{anyNCreal}
An NC function is uniformly analytic in a uniformly open neighbourhood of $Y \in \cdn$ if and only if $f$ has a minimal matrix-centre realization at $Y$ that obeys the linearized Lost Abbey conditions at $Y$.
\end{cor}
\begin{proof}
By the previous theorem, if $f$ is uniformly analytic in a uniformly open neighbourhood of $Y$, then $f \sim _Y (\hat{A},\hat{b},\hat{c})$ has a matrix-centre realization at $Y$. By the Kalman decomposition, $f$ then has a minimal matrix-centre realization at $Y$, $f \sim (A,b,c)_Y$. By Theorem \ref{realLAC}, $(A,b,c)_Y$ obeys the linearized Lost Abbey conditions at $Y$. Conversely, if $f \sim _Y (A,b,c)$ has a minimal matrix-centre realization at $Y$ that obeys the linearized Lost Abbey conditions at $Y$, then $f$ is a uniformly analytic NC function in the uniformly open set $r \cdot \B ^d _{\N n} (Y)$ with $r \geq \| A \| _{\mr{CB;row}} ^{-1}$, again by Theorem \ref{realLAC}. 
\end{proof}

\begin{remark}
Given $Y \in \cdn$, consider the matricial Fock space, $\scr{F} _n (\C ^d)$ and consider the set of all $f \in \scr{F} _n (\C ^d)$ that define non-commutative functions via evaluation at $X \in \frac{1}{\sqrt{n}} \B ^d _{\N n } (Y)$. Clearly this is a closed subspace of $\scr{F} _n (\C ^d)$. We can then define the \emph{free Hardy space at $Y$}, $\hardy (Y)$, as the closed subspace consisting of all elements of $\scr{F} _n (\C ^d)$ that define non-commutative functions via evaluation in $\frac{1}{\sqrt{n}} \B ^d _{\N n } (Y)$. Let $P_Y : \scr{F} _n (\C ^d) \rightarrow \hardy (Y)$ denote the orthogonal projection.

Since evaluation at any $X \in \frac{1}{\sqrt{n}} \B ^d _{m n } (Y)$, $m \in \N$, is a bounded linear homomorphism from $\scr{F} _n (\C ^d )$ into $\C ^{mn \times mn}$, it follows that $\hardy (Y)$ is a non-commutative reproducing kernel Hilbert space (NC-RKHS) in the NC uniform row-ball of radius $\sqrt{n} ^{-1}$ centred at $Y$ \cite{NCRKHS}. We plan to develop the theory of $\hardy (Y)$ further in future work. 
\end{remark}

\section{An application} \label{sec:app}

In \cite{KVV-local}, Klep, Vinnikov and Vol\v{c}i\v{c} developed a local theory of non-commutative functions. In particular, they introduced the ring, $\scr{O} _0 ^u$, of \emph{uniformly analytic NC germs at $0$}. That is, given uniformly analytic NC functions $f$ and $g$, one says they are \emph{locally evaluation equivalent at $0$}, $f \sim _0 g$, if their evaluations agree in a uniformly open neighbourhood of $0$. This is an equivalence relation and the $\sim_0$ equivalence classes are then called uniformly analytic NC germs. By \cite[Proposition 5.3]{KVV-local}, $\scr{O} _0 ^u$ is a \emph{semi-free ideal ring} or \emph{semifir}, which is a type of ring that admits a universal localization. Namely, $\scr{O} _0 ^u$, has a \emph{universal skew field of fractions}, $\scr{M} _0 ^u$, the \emph{skew field of uniformly meromorphic NC germs at $0$}. Elements of $\scr{M} _0 ^u$ can be identified with certain ``evaluation" equivalence classes of NC rational expressions composed with elements of $\scr{O} _0 ^u$, see \cite[Section 5.2 and Corollary 5.4]{KVV-local}

A simple observation, \cite[Lemma 3.2]{AMS-opreal}, shows that $f \in \scr{O} _0 ^u$ if and only if it has a realization, $(A,b,c) \in \scr{B} (\cH) ^{1\times d} \times \cH \times \cH$ so that $\scr{O} _0 ^u$ can be identified with $\scr{O} _d ^\scr{B}$, $\scr{B} = \scr{B} (\cH ) ^{1\times d}$, the ring of all (uniformly analytic) NC functions which have realizations (centred at $0$). In general, given a set $\scr{S} (\cH) \subseteq \scr{B} (\cH)$, let $\scr{S} := \scr{S} (\cH) ^{1\times d}$. Further consider the subsets of row $d-$tuples of linear operators, $\scr{F} = \scr{F} (\cH) ^{1 \times d}$, $\scr{C} = \scr{C} (\cH) ^{1\times d}$ and $\scr{T} _p := \scr{T} _p (\cH) ^{1\times d}$ where $\scr{F} (\cH)$ denotes the two-sided ideal of all finite-rank operators on $\cH$, $\scr{C} (\cH)$ denotes the operator-norm closed two-sided ideal of all compact linear operators on $\cH$ and $\scr{T} _p (\cH)$ denotes the Schatten $p-$class operators on $\cH$ for $p \in [1, +\infty)$. The Schatten $p-$classes consist of the compact linear operators, $A$ for which $\mr{tr} \, |A| ^p < +\infty$ and are also (non-closed) two-sided ideals in $\scr{B} (\cH )$.  As observed in \cite[Corollary 6.3, Theorem 6.10]{AMS-opreal}, each of the sets $\scr{O} _0 ^\scr{S}$ of uniformly analytic NC functions with realizations $(A,b,c) \in \scr{S} \times \cH \times \cH$ for $\scr{S} \in \{ \scr{T} _p, \scr{C} \}$ are semifirs and we have the proper inclusions of semifirs,
$$ \ratfps = \scr{O} _d ^\scr{F} \subsetneqq \scr{O} _d ^{\scr{T} _p} \subsetneqq \scr{O} _d ^{\scr{T} _q} \subsetneqq \scr{O} _d ^{\scr{C}} \subsetneqq \scr{O} _d ^\scr{B} = \scr{O} _0 ^u, $$ for $1 \leq p < q < +\infty$. Here, $\ratfps$ denotes the semifir of all NC rational functions that are defined at $0$, and $\fr \in \ratfps$ if and only if it has a finite--dimensional realization. These proper inclusions also hold for the universal skew fields of fractions of these semifirs, 
$$\fskew = \scr{M} _d ^\scr{F} \subsetneqq \scr{M} _d ^{\scr{T} _p} \subsetneqq \scr{M} _d ^{\scr{T} _q} \subsetneqq \scr{M} _d ^{\scr{C}} \subsetneqq \scr{M} _d ^\scr{B} = \scr{M} _0 ^u. $$ The results in this paper readily imply that any NC function in one of these skew fields, $\scr{M} _d ^\scr{S}$, admits matrix-centre realizations that take values in the given class, $\scr{S} (\cH)$, about any point $Y$ in its uniformly open NC domain. 

\begin{cor}
Let $f \in \scr{M} _d ^\scr{S}$, $\scr{S} \in \{ \scr{F}, \scr{T} _p, \scr{C}, \scr{B} \}$, If $Y \in \mr{Dom} _n \, f \subseteq \cdn$, then $f$ has a matrix-centre realization at $Y$, $f \sim _Y (A,b,c)$ that obeys $A_j : \C ^{n\times n} \rightarrow \scr{S} (\cH)$. 
\end{cor}
\begin{proof}
As described in \cite[Section 6]{AMS-opreal}, following \cite[Subsection 5.2]{KVV-local}, elements of $\scr{M} _d ^\scr{S}$ can be identified as NC rational expressions composed with elements of $\scr{O} _d ^\scr{S}$. Namely, if $f \in \scr{M} _d ^\scr{S}$, then there exists $k \in \N$, $\fr \in \fskewk$ and $g_1, \cdots, g_k \in \scr{O} _d ^\scr{S}$ so that $f = \fr (g_1, \cdots, g_k)$. The domain of $f$, is then the largest NC set on which the evaluations of $f$ make sense. By definition, each $g_i \in \scr{O} _d ^\scr{S}$ has a realization (at $0$), $g_i \sim _0 (A,b,c)$, with $A \in \scr{S}$. If $Y \in \mr{Dom} _n \, f$, then necessarily $Y \in \mr{Dom} _n \, g_i$, $1\leq i \leq k$, and by the results of Section \ref{sec:translate} and Theorem \ref{matrixtrans}, we can translate the realizations of $g_i$ at $0$ to obtain realizations of the $g_i$ at $Y \in \cdn$, $g_i \sim _Y (A', b',c')$ which obey $A' _j : \C ^{n \times n} \rightarrow \scr{S} (\cH)$, see Subsection \ref{ss:mtrans}. Each of the sets $\scr{S} (\cH) \in \{ \scr{F} (\cH), \scr{T} _p (\cH), \scr{C} (\cH), \scr{B} (\cH) \}$ are two-sided ideals in $\scr{B} (\cH)$ that are closed under finite--rank perturbations and so it follows that each of the sets $\scr{S}$ are closed under the application of the (FM or descriptor) realization algorithm for sums, products and inverses of NC functions with realizations about $Y$, see Subsection \ref{ss:realalg}. In conclusion, $f \sim _Y (\hat{A}, \hat{b}, \hat{c})$, with $\hat{A} _j : \C ^{n\times n} \rightarrow \scr{S} (\cH)$.  
\end{proof}

\begin{remark}
By Corollary \ref{anyNCreal}, any NC function, $f$, that is uniformly analytic in a uniformly open neighbourhood of a matrix point, $Y \in \cdn$, has a matrix-centre realization at $Y$. This does not mean, however, that $f \in \scr{M} _d ^\scr{B}$, the skew field of uniformly meromorphic germs at $0$. Indeed, \cite[Theorem 7.2]{KVV-local} shows that if $Y$ is a semisimple point one can construct a uniformly analytic NC function, $f$, in a uniformly open neighbourhood of $Y$ whose evaluations take values in nilpotent matrices of any fixed order. (Clearly, such an $f$ cannot belong to any skew field.) On the other hand, any NC function that admits a finite--dimensional matrix-centre realization is necessarily an NC rational function. This raises the open question, posed to us by V. Vinnikov: If an NC function, $f$, has a matrix-centre realization, $(A,b,c)_Y$, taking values in compact linear operators, do we have that $f \in \scr{M} _d ^\scr{C}$? As described in \cite{AMS-opreal}, $\scr{M} _d ^\scr{C}$ is a natural non-commutative and multivariate generalization of the field of meromorphic functions in $\C$. 
\end{remark}

\begin{thm}
Let $(A,b,c)_Y$ and $(A',b',c')_Z$ be two minimal and jointly compact matrix centre realizations about $Y \in \cdm$ and $Z \in \cdn$ that define the same uniformly analytic NC function on some uniformly open NC set. That is, $A_j : \C ^{m \times m} \rightarrow \scr{C} (\cH)$ and $A'_j : \C ^{n\times n} \rightarrow \scr{C} (\cH ')$. Then $\scr{D} ^Y _j (A) = \scr{D} ^Z _k (A')$ for any $j,k \in \N$ so that $jm = kn$. Moreover, given any $X \in \C ^{(jm \times jm) \cdot d} = \C ^{(kn \times kn) \cdot d}$, the resolvent functions, $R _{X \otimes A} (\la )$ and $R_{X \otimes A'} (\la)$, $\la \in \C \sm \{ 0 \}$, have the same poles with the same orders. 
\end{thm}
In the above, if $T \in \scr{B} (\cH)$ is any linear operator, $R_T (\la) :=(\la I_\cH - T) ^{-1}$, $\la \in \C$.
\begin{proof}
By \cite[Subsection 6.1 and Theorem 6.8]{AMS-opreal}, $\scr{D} ^Y _j (A)$ and $\scr{D} ^Z _k (A')$ are both analytic Zariski dense and open sets, and hence are matrix-norm open, dense and connected. In particular, they have non-trivial intersection. If we fix a point $Y' \in \scr{D} ^Y _j (A) \cap \scr{D} ^Z _k (A')$, we can translate both realizations to $Y'$ to obtain two minimal, analytically equivalent and compact realizations at $Y'$ using the results of Section \ref{sec:translate}. That is, we can assume, without loss of generality that $Z=Y \in \cdn$, and that $(A,b,c) \sim _Y (A',b',c')$ are analytically equivalent at $Y$, \emph{i.e.} they define the same uniformly analytic NC function, $f$, in a uniformly open neighbourhood of $Y \in \cdn$. 

Given any free polynomials, $p,q \in \fp$ and $X \in \scr{D} _m ^Y (A) \cap \scr{D} _m ^Y (A')$ consider 
$$ f_{p,q} (X) := I_m \otimes b^* q(A)[\vec{G}] L_A (X - I_m \otimes Y) ^{-1} p(A)[\vec{H}] I_m \otimes c, $$
and $$ f_{p,q} ' (X) = I_m \otimes b^{'*} q(A')[\vec{G}] L_{A'} (X - I_m \otimes Y) ^{-1} p(A')[\vec{H}] I_m \otimes c'. $$ We claim that $f_{p,q} = f ' _{p,q}$. Indeed, replacing $X$ by $X(\la) := \la (X - I_m \otimes Y) + I_m \otimes Y$ for $\la \in \C$ of sufficiently small modulus, $X(\la)$ will belong to a column-ball centred at $Y$ of arbitrarily small radius so that $X (\la) \in \scr{D} _m ^Y (A) \cap \scr{D} _m ^Y (A')$. Then, by assumption, $f_{p,q} (X(\la)) = f_{p,q} ' (X(\la))$. However, 
$$  A (X (\la) - I_m \otimes Y) = \la A (X - I_m \otimes Y), $$ and $A(X - I_m \otimes Y)$ takes values in compact operators, and similarly so does $A' (X - I_m \otimes Y)$. Hence, $\scr{D} ^Y _m (A) \cap \scr{D} ^Y _m (A')$ is open and path-connected, and it follows from the identity theorem that $f_{p,q} (X) = f_{p,q} ' (X)$. 

Fix an arbitrary $X \in \C ^{(mn \times mn) \cdot d}$, let $X(\la) := \la (X - I_m \otimes Y) + I_m \otimes Y$, as before, for $\la \in \C$, and consider
$$ L_A (X (\la) - I_m \otimes Y) = I_m \otimes I_\cH - \la A(X- I_m \otimes Y), $$ as well as $L_{A'} (X(\la) - I_m \otimes Y)$. Then, if we set $z = \la ^{-1}$, 
$$ R_{A (X-I_m \otimes Y)} (z) = (z I_m \otimes I_\cH - A (X - I_m \otimes Y) ) ^{-1} = \la L_A (X(\la) - I_m \otimes Y) ^{-1}. $$ Since $A (X - I_m \otimes Y)$ and $A' (X-I_m \otimes Y)$ are compact, $R(z):= R_{A(X-I_m\otimes Y)} (z)$ and $R' (z) := R_{A' (X-I_m \otimes Y)} (z)$ are meromorphic operator-valued functions in $\C \sm \{ 0 \}$. 

Hence, if $X \notin \scr{D} _m ^Y (A)$, then $z=1$ is a pole of finite order for $R(z)$. 
Let $\ell, m \in \N \cup \{ 0 \}$ be the orders of the poles of $R(z)$ and $R'(z)$ at $z =1$. For any free polynomials, $p,q$, consider the matrix-valued meromorphic function on $\C \sm \{ 0 \}$, 
\ba g_{p,q} (z) & := & z^{-1} f_{p,q}(X(z^{-1})) = I_m \otimes b^* p(A)[\vec{G}] R(z) q(A) [\vec{H}] I_m \otimes c \\
& = & I_m \otimes b^{'*} p(A')[\vec{G}] R'(z) q(A') [\vec{H}] I_m \otimes c'. \ea 
Then, since $g_{p,q}$ has a pole of order at most $\ell$ at $z=1$, for sufficiently small $r>0$, 
$$ 0 = \cint _{r\cdot \partial \D (1)} (z-1) ^\ell g_{p,q} (z) dz, $$ and then by the Riesz--Dunford holomorphic functional calculus, 
\ba 0 & = & \frac{1}{2\pi i} \cint _{r\cdot \partial \D (1)} (1-z) ^\ell g_{p,q} (z) dz \\
& = & I_m \otimes b^* p(A) [\vec{G}] \, \underbrace{\frac{1}{2\pi i} \cint _{r\cdot \partial \D (1)} (z-1) R _{A (X - I_m \otimes Y)} (z) dz}_{(I_m \otimes I_\cH - A(X-I_m \otimes Y) ) ^\ell E_{1} (A (X-I_m\otimes Y)) }  \, q(A) [\vec{H} ] I_m \otimes c \\
& = & I_m \otimes b^{'*} p(A') [\vec{G}] \, (I_m \otimes I_\cH - A'(X-I_m \otimes Y) ) ^\ell E_{1} (A' (X - I_m \otimes Y))  \, q(A') [\vec{H} ] I_m \otimes c', \ea for any tuples of $mn \times mn$ matrices, $\vec{G}, \vec{H}$. In the above, $E_{1} (A (X-I_m\otimes Y))$ denotes the Riesz idempotent corresponding to the isolated eigenvalue, $\la =1$, of the compact linear operator $A (X-I_m \otimes Y)$. That is, $E_1 (A (X-I_m\otimes Y))$ is defined, via the holomorphic functional calculus, by the characteristic function of an open neighbourhood, $\Om$, of $1 \in \C$, so that $\sigma (A (X-I_m\otimes Y)) \cap \Om = \{ 1 \}$. 

By minimality, 
$$ 0 = (I_m \otimes I_\cH - A'(X-I_m \otimes Y) ) ^\ell E_{1} (A' (X-I_m\otimes Y)), $$ and this proves that $m \leq \ell$, by the Laurent series formula for the resolvent function of a compact operator, see \cite[Theorem 5.2]{AMS-opreal}. By symmetry, we obtain $\ell =m$. In particular, if $\ell =0$ so that $X \in \scr{D} ^Y (A)$, then $X \in \scr{D} ^Y (A')$, and we conclude that $\scr{D} ^Y (A) = \scr{D} ^Y (A')$. 
\end{proof}

Recall that by \cite[Section 6.1 and Theorem 6.4]{AMS-opreal}, any $f \in \scr{M} _d ^\scr{C}$ can be viewed as an ``evaluation" equivalence classes of NC rational expressions composed with elements of $\scr{O} _d ^{\scr{C}}$. The above theorem shows that given any $f \in \scr{M} _d ^\scr{C}$, so that $F := \fr \circ (g_1, \cdots, g_k) \in f$, $\fr \in \fskewk$ and $g_i \in \scr{O} _d ^\scr{C}$, \emph{i.e.} $F \in \fskewk \circ \scr{O} _d ^\scr{C}$, that if $Y \in \mr{Dom} _m \, F \subseteq \mr{Dom} _m \, f$, where $m \in \N$ is minimal, then we can construct a minimal, matrix-centre realization of $F$ at $Y$, $F \sim (A,b,c) _Y$, with the $A_j$ taking values in compact operators, and then 
$$ \nbdom F = \nbdom f = \scr{D} ^Y (A). $$

\setstretch{0.9}
\setlength{\parskip}{0pt}
\setlength{\itemsep}{0pt}

\end{document}